\documentclass[a4paper,11pt,twoside]{article}
\usepackage[english]{babel}
\usepackage[utf8x]{inputenc}

\usepackage[margin=2cm]{geometry}

\usepackage{amsmath,amsthm,amssymb}
\usepackage{times}
\usepackage{enumerate}
\usepackage{color}%This command adds different colors and graphs.

\makeatletter
\def\author#1{\gdef\autrun{\def\and{\unskip, }#1}\gdef\@author{#1}}
 \makeatother

\usepackage{amssymb,latexsym}
\usepackage{mathrsfs}
\usepackage{amssymb}
\usepackage{amsmath,latexsym}
\usepackage{amscd,latexsym} %for communtative diagram
\usepackage[all]{xy}
\newcommand{\N}{{\mathbb N}}

\newcommand{\R}{{\mathbb R}}

\newtheorem{theorem}{Theorem}[section]

\newtheorem{corollary}[theorem]{Corollary}

\newtheorem{example}[theorem]{Example}
\newtheorem{remark}[theorem]{Remark}
\newtheorem{hypothesis}[theorem]{Hypothesis}
\newtheorem{lemma}[theorem]{Lemma}

\newtheorem{proposition}[theorem]{Proposition}
\newtheorem{claim}[theorem]{Claim}

\numberwithin{equation}{section}

\begin{document}

\title{Parameterized splitting theorems and bifurcations\\ for potential operators, Part II:\\
Applications to quasi-linear elliptic equations and systems
\thanks
{Partially supported by the NNSF  11271044 of China.
\endgraf\hspace{2mm} 2020 {\it Mathematics Subject Classification.}
Primary: 35B32; Secondary: 58J55, 35J35
 \endgraf\hspace{2mm} {\it Key words and phrases.}
 Bifurcation, potential operator, splitting theorem.
 }}
%\date{January 11, 2020}
\author{Guangcun Lu}

\date{November 11, 2021}
%\date{November 10, 2020}
%\subjclass[2010]{53C25, 53D35, 57R17}
%\footnotetext{{\it Key words and phrases.}  Bifurcation, potential operator, splitting theorem,
%quasi-linear elliptic Euler equations. }
%\footnotetext{{{\it Mathematics Subject Classification.}} 58E05, 49J52 (primary), 49J45 (secondary).}

 \maketitle \vspace{-0.3in}

% \centerline{\scshape Guangcun Lu$^*$}

% \centerline{\scshape Guangcun Lu}

 %35B32   Bifurcations in context of PDEs [See also 34C23, 34F10, 34H20, 37F46, 37Gxx, 37H20, 35J20, 37L10, 37M20, 47J15, 58E05, 58E07, 58J55, 74G60, 74H60]
 %35B32   Bifurcations in context of PDEs [See also 34C23, 34F10, 34H20, 37F46, 37Gxx, 37H20, 35J20, 37L10, 37M20, 47J15, 58E05, 58E07, 58J55, 74G60, 74H60]

\begin{abstract}
This is the second part of a series devoting to the generalizations and applications
of common theorems in variational bifurcation theory. Using abstract theorems in the first part we obtain many new bifurcation results for quasi-linear elliptic  boundary value problems of higher order.
\end{abstract}

\tableofcontents

\section{Introduction}\label{sec:Intro}
\setcounter{equation}{0}

In Part I of the series, we have generalized  some famous bifurcation theorems
for potential operators by weakening standard assumptions on the differentiability of the involved functionals.
This part studies their applications  to bifurcations for quasi-linear elliptic systems.
 For showing our methods  we only consider the Dirichlet boundary conditions in many results.

The precise contents of the paper is as follows.
Section~\ref{sec:BifE.1} contains some fundamental hypotheses and preliminaries
on quasi-linear elliptic systems considered in next sections.
Section~\ref{sec:BifE.2} proves some bifurcations results for quasi-linear elliptic systems with some growth restrictions
 with theorems proved in \cite[Sections~3,4,5]{Lu8}.
In Section~\ref{sec:BifE.3} we use abstract theorems in \cite[Section~6]{Lu8}
to study bifurcations for quasi-linear elliptic systems
without growth restrictions unless higher smoothness requirements for the related functions $F$ and domains $\Omega$.
In Section~\ref{sec:BifE.4} we  study bifurcations  for quasi-linear elliptic Dirichlet problems  from deformations of domains, and generalize previous results
   for semilinear elliptic Dirichlet problems on a ball.

In addition, there are three appendixes.
In Appendix~\ref{app:A} we study differentiability of composition mappings.
Corollaries~\ref{cor:babySmoothness},~\ref{cor:smoothness} of the main result Theorem~\ref{th:babySmoothness}
 were only stated and used in case $m=2$ in \cite{Wend}.
In Appendix~\ref{app:B} we prove three lemmas,
which fill in the detail on much of the proof of Theorem~\ref{th:BifE.9}.
In Appendix~\ref{app:C} we give a slightly generalization of
 the standard result concerning the continuity of the Nemytski operators.

 It is possible that results in this paper are generalized to the case of unbounded domains  $\Omega\subset\R^n$ with \cite{Vo}.
 We believe that theories in this series can also be used to improve bifurcation results
for geometric variational problems such as \cite{Bor} and
\cite{BetPS1, BetPS2}, etc., see \cite{Lu12}. Other applications will be given in \cite{Lu9, Lu10, Lu11}.

\section{Structural hypotheses and preliminaries}\label{sec:BifE.1}
\setcounter{equation}{0}

Without special statements, throughout this paper we always assume that integers $N\ge 1$, $n>1$,
 and that $\Omega\subset\R^n$ is a bounded domain   with  boundary $\partial\Omega$.
(The case $n=1$ will be considered in \cite{Lu9, Lu10, Lu11} independently.)

 In \cite{Lu7} we introduced the following (denoted by \textsf{Hypothesis} $\mathfrak{F}_{2,N}$ in \cite{Lu6}).

\noindent{\textsf{Hypothesis} $\mathfrak{F}_{2,N,m,n}$}.\quad
 For each multi-index $\gamma$ as above, let
 \begin{eqnarray*}
 &&2_\gamma\in (2,\infty)\;\hbox{if}\;
 |\gamma|=m-n/2,\qquad 2_\gamma=\frac{2n}{n-2(m-|\gamma|)}
\;\hbox{if}\; m-n/2<|\gamma|\le m,\\
 &&2'_\gamma=1\;\hbox{if}\;|\gamma|<m-n/2,\qquad
 2'_\gamma=\frac{2_\gamma}{2_\gamma-1} \;\hbox{if}\;m-n/2\le |\gamma|\le m;
 \end{eqnarray*}
 and for each two multi-indexes $\alpha, \beta$ as above, let
 $2_{\alpha\beta}=2_{\beta\alpha}$ be defined by the conditions
\begin{eqnarray*}
&&2_{\alpha\beta}= 1-\frac{1}{2_\alpha}-\frac{1}{2_\beta}\quad
 \hbox{if}\;|\alpha|=|\beta|=m,\\
&&2_{\alpha\beta}=  1-\frac{1}{2_\alpha}\quad \hbox{if}\;m-n/2\le |\alpha|\le
 m,\; |\beta|<m-n/2,\\
&&2_{\alpha\beta}= 1 \quad \hbox{if}\; |\alpha|, |\beta|<m-n/2, \\
&& 0<2_{\alpha\beta}<1-\frac{1}{2_\alpha}-\frac{1}{2_\beta}\quad\hbox{if}\;|\alpha|,\;|\beta|\ge
 m-n/2,\;|\alpha|+|\beta|<2m.
\end{eqnarray*}
Let $M(k)$ be the number of $n$-tuples
$\alpha=(\alpha_1,\cdots,\alpha_n)\in (\mathbb{N}_0)^n$  of length
 $|\alpha|:=\alpha_1+\cdots+\alpha_n\le k$, $M_0(k)=M(k)-M(k-1)$, $k=0,\cdots,m$,
where $M(-1)=\emptyset$ and  $M(0)=M_0(0)$ only consists of
${\bf 0}=(0,\cdots,0)\in (\mathbb{N}_0)^n$. Write $\xi\in \prod^m_{k=0}\mathbb{R}^{N\times M_0(k)}$ as
$\xi=(\xi^0,\cdots,\xi^m)$, where
 $ \xi^0=(\xi^1_{\bf 0},\cdots,\xi^N_{\bf 0})^T\in \mathbb{R}^{N}$ and
 $$
 \xi^k=\left(\xi^i_\alpha\right)\in \mathbb{R}^{N\times M_0(k)},
 \quad k=1,\cdots,m,\quad 1\le i\le N,\quad |\alpha|=k.
 $$
  Denote by  $\xi^k_\circ=\{\xi^k_\alpha\,|\,|\alpha|<m-n/2\}$ for $k=1,\cdots,N$. Let
\begin{eqnarray}\label{e:6.0}
\overline\Omega\times\prod^m_{k=0}\mathbb{R}^{N\times M_0(k)}\ni (x,
\xi)\mapsto F(x,\xi)\in\R
\end{eqnarray}
be twice continuously differentiable in $\xi$ for almost all $x$,
measurable in $x$ for all values of $\xi$, and $F(\cdot,\xi)\in L^1(\Omega)$ for $\xi=0$.
Suppose that derivatives of $F$ fulfill  the following properties:
\begin{enumerate}
\item[(i)]  For $i=1,\cdots,N$ and $|\alpha|\le m$, functions
 $F^i_\alpha(x,\xi):= F_{\xi^i_\alpha}(x,\xi)$ for $\xi=0$
belong to $L^1(\Omega)$ if $|\alpha|<m-n/2$,
and to $L^{2'_\alpha}(\Omega)$ if $m-n/2\le |\alpha|\le m$.
\item[(ii)] There exists a continuous, positive, nondecreasing functions $\mathfrak{g}_1$ such that
for $i,j=1,\cdots,N$ and $|\alpha|, |\beta|\le m$  functions
$$
\overline\Omega\times\R^{M(m)}\to\R,\; (x, \xi)\mapsto
F^{ij}_{\alpha\beta}(x,\xi):=F_{\xi^i_\alpha\xi^j_\beta}(x,\xi)
$$
satisfy:
\begin{eqnarray}\label{e:6.1}
 |F^{ij}_{\alpha\beta}(x,\xi)|\le
\mathfrak{g}_1(\sum^N_{k=1}|\xi_\circ^k|)\left(1+
\sum^N_{k=1}\sum_{m-n/2\le |\gamma|\le
m}|\xi^k_\gamma|^{2_\gamma}\right)^{2_{\alpha\beta}}.
\end{eqnarray}
\item[(iii)] There exists a continuous, positive, nondecreasing functions $\mathfrak{g}_2$ such that
\begin{eqnarray}\label{e:6.2}
\sum^N_{i,j=1}\sum_{|\alpha|=|\beta|=m}F^{ij}_{\alpha\beta}(x,\xi)\eta^i_\alpha\eta^j_\beta\ge
\mathfrak{g}_2(\sum^N_{k=1}|\xi^k_\circ|)
\sum^N_{i=1}\sum_{|\alpha|= m}(\eta^i_\alpha)^2
\end{eqnarray}
for any $\eta=(\eta^{i}_{\alpha})\in\R^{N\times M_0(m)}$.

{\it Note}:  If $m\le n/2$ the functions  $\mathfrak{g}_1$ and $\mathfrak{g}_2$ should be understand as positive constants.
\end{enumerate}

In \cite[Proposition~4.22]{Lu7} it was proved that
\textsf{Hypothesis} $\mathfrak{F}_{2,N,1,n}$ is weaker than the
 {\bf controllable growth conditions} (abbreviated to CGC below) \cite[page 40]{Gi}.
(CGC was called  `common condition of Morrey' or `the natural assumptions of Ladyzhenskaya and Ural'tseva' \cite[page 38,(I)]{Gi}.)\\

\noindent{\bf CGC}: $\overline\Omega\times\mathbb{R}^N\times\mathbb{R}^{N\times n}\ni (x,
z,p)\mapsto F(x, z,p)\in\R$ is of class $C^2$, and
there exist positive constants $\nu, \mu, \lambda, M_1, M_2$,  such that with
$|z|^2:=\sum^N_{l=1}|z_l|^2$ and $|p|^2:=\sum_{|\alpha|=1}\sum^N_{k=1}|p^k_\alpha|^2$,
\begin{eqnarray*}
&\nu\left(1+|z|^2+|p|^2\right)-\lambda\le F(x, z,p)
\le\mu\left(1+|z|^2+|p|^2\right),\\
&|F_{p^i_\alpha}(x,z,p)|, |F_{p^i_\alpha x_l}(x,z,p)|, |F_{z_j}(x,z,p)|, |F_{z_jx_l}(x,z,p)|\le \mu\left(1+|z|^2+|p|^2\right)^{1/2},\\
&\hspace{-10pt}|F_{p^i_\alpha z_j}(x,z,p)|,\quad |F_{z_iz_j}(x,z,p)|\le \mu,\\
&M_1\sum^N_{i=1}\sum_{|\alpha|= 1}(\eta^i_\alpha)^2\le\sum^N_{i,j=1}\sum_{|\alpha|=|\beta|=1}F_{p^i_\alpha p^j_\beta}(x,z,p)\eta^i_\alpha\eta^j_\beta\\&\le
M_2\sum^N_{i=1}\sum_{|\alpha|= 1}(\eta^i_\alpha)^2\\
&\forall\eta=(\eta^{i}_{\alpha})\in\R^{N\times n}.
\end{eqnarray*}
Moreover, if $F=F(x,p)$ does not depend explicitly on $z$, the first three lines are replaced by
\begin{eqnarray*}
&\nu\left(1+|p|^2\right)-\lambda\le F(x,p)
\le\mu\left(1+|p|^2\right)\quad\hbox{and}\\
&|F_{p^i_\alpha}(x,p)|,\quad |F_{p^i_\alpha x_l}(x,p)|\le \mu\left(1+|p|^2\right)^{1/2}.
\end{eqnarray*}

 A bounded domain $\Omega$ in $\R^n$ is said to be a {\bf Sobolev domain} for $(2,m,n)$
 if  the Sobolev embedding theorems for the spaces $W^{m, 2}(\Omega)$ hold.
Let $W^{m,2}_0(\Omega,\mathbb{R}^N)$ be equipped with the  inner product
 \begin{eqnarray}\label{e:6.2.2}
 (\vec{u},\vec{v})_H=\sum^N_{i=1}\sum_{|\alpha|=m}\int_\Omega D^\alpha u^i D^\alpha v^i dx.
\end{eqnarray}

The following two theorems are contained in \cite[Theorem~4.1]{Lu6} (or \cite[Theorems~4.1,4.2]{Lu7}).

\begin{theorem}\label{th:6.1}
 Given  integers $m, N\ge 1$, $n\ge 2$, let $\Omega\subset\R^n$ be a Sobolev domain
 for $(2,m,n)$, and let $V_0$ be a closed subspace of $W^{m,2}(\Omega, \mathbb{R}^N)$ and $V=\vec{w}+V_0$
 for some $\vec{w}\in W^{m,2}(\Omega, \mathbb{R}^N)$.
Suppose that (i)-(ii) in \textsf{Hypothesis} $\mathfrak{F}_{2,N,m,n}$ hold.
 Then we have the following:
 \begin{enumerate}
  \item[\rm (A)] The restriction $\mathfrak{F}_V$ of  the functional
\begin{equation}\label{e:6.3}
W^{m,2}(\Omega, \mathbb{R}^N)\ni \vec{u}\mapsto\mathfrak{F}(\vec{u})=\int_\Omega F(x, \vec{u},\cdots, D^m\vec{u})dx
\end{equation}
 to $V$ is bounded on any bounded subset, of class $C^1$, and has  derivative
$\mathfrak{F}'_V(\vec{u})$  at $\vec{u}\in V$ given by
\begin{equation}\label{e:6.4}
\langle \mathfrak{F}'_V(\vec{u}), \vec{v}\rangle=\sum^N_{i=1}\sum_{|\alpha|\le m}\int_\Omega F^i_\alpha(x,
\vec{u}(x),\cdots, D^m \vec{u}(x))D^\alpha v^i dx,\quad\forall \vec{v}\in V_0.
\end{equation}
Moreover, the map $V\ni\vec{u}\to \mathfrak{F}'_V(\vec{u})\in V_0^\ast$ also maps bounded subset into bounded one.\\
\item[\rm (B)]  At each $\vec{u}\in V$, the map $\mathfrak{F}'_V$  has the G\^ateaux derivative
$D\mathfrak{F}'_V(\vec{u})\in \mathscr{L}(V_0, V^\ast_0)$ given by
  \begin{equation}\label{e:6.5}
   \langle D\mathfrak{F}'_V(\vec{u})[\vec{v}],\vec{\varphi}\rangle=
   \sum^N_{i,j=1}\sum_{|\alpha|,|\beta|\le m}\int_\Omega
  F^{ij}_{\alpha\beta}(x, \vec{u}(x),\cdots, D^m \vec{u}(x))D^\beta v^j\cdot D^\alpha\varphi^i dx.
    \end{equation}
(Equivalently,   the gradient map of $\mathfrak{F}_V$,
$V\ni \vec{u}\mapsto\nabla \mathfrak{F}_V(\vec{u})\in V_0$,  given by
$(\nabla\mathfrak{F}_V(\vec{u}), \vec{v})_{m,2}=\langle \mathfrak{F}'_V(\vec{u}), \vec{v}\rangle\;
 \forall \vec{v}\in V_0$,
has the G\^ateaux derivative $D(\nabla \mathfrak{F}_V)(\vec{u})\\\in\mathscr{L}_s(V_0)$ at every $\vec{u}\in V$.)
Moreover,   $D\mathfrak{F}'_V$ also satisfies the following properties:
\begin{enumerate}
\item[\rm (i)] For every given $R>0$, $\{D\mathfrak{F}'_V(\vec{u})\,|\, \|\vec{u}\|_{m,p}\le R\}$
is bounded in $\mathscr{L}_s(V_0)$.
Consequently,  $\mathfrak{F}_V$ is  of class $C^{2-0}$.
\item[\rm (ii)] For any $\vec{v}\in V_0$, $\vec{u}_k\to
\vec{u}_0$ implies
$D\mathfrak{F}'_V(\vec{u}_k)[\vec{v}]\to D\mathfrak{F}'_V(\vec{u}_0)[\vec{v}]$ in $V^\ast_0$.
\item[\rm (iii)] If $F(x,\xi)$ is independent of all variables $\xi^k_\alpha$, $|\alpha|=m$,
$k=1,\cdots,N$, then
$$
V\to \mathscr{L}(V_0, V^\ast_0),\;\vec{u}\mapsto D\mathfrak{F}'_V(\vec{u})
$$
is  continuous, (namely $\mathfrak{F}_V$ is of class $C^2$),  and  $D(\nabla\mathfrak{F}_V)(\vec{u}): V_0\to V_0$ is a completely continuous  linear operator for each $\vec{u}\in V$.
\end{enumerate}
\end{enumerate}
\end{theorem}

Since (i) implies that the map $\mathfrak{F}'_V$ and so the gradient
$\nabla\mathfrak{F}_V$ is  of class $C^{1-0}$, these two maps are Hadamard differentiable (cf. \cite[page 1252]{EvSt07} or \cite[page 1030]{Stu14A}).

\begin{theorem}\label{th:6.2}
Under assumptions of Theorem~\ref{th:6.1},  suppose that
\textsf{Hypothesis}  $\mathfrak{F}_{2,N,m,n}$(iii) is also satisfied. Then
\begin{enumerate}
\item[\rm (C)]  $\mathfrak{F}': W^{m,2}(\Omega, \mathbb{R}^N)\to (W^{m,2}(\Omega, \mathbb{R}^N))^\ast$ is of class  $(S)_+$ (cf.\cite[\S2]{Lu8}).
\item[\rm (D)]   For $u\in V$, let $D(\nabla\mathfrak{F}_V)(\vec{u})$,
$P(\vec{u})$ and $Q(\vec{u})$ be
  operators in $\mathscr{L}(V_0)$ defined by
  \begin{eqnarray}
  &&\hspace{-2.8cm}(D(\nabla\mathfrak{F}_V)(\vec{u})[\vec{v}],\vec{\varphi})_{m,2}=\nonumber\\
  &&\sum^N_{i,j=1}\sum_{|\alpha|,|\beta|\le m}\int_\Omega
  F^{ij}_{\alpha\beta}(x, \vec{u}(x),\cdots,%\notag\\&&\quad
  D^m \vec{u}(x))D^\beta v^j\cdot D^\alpha\varphi^i dx,\nonumber\\
   (P(\vec{u})\vec{v}, \vec{\varphi})_{m,2}&&\hspace{-0.5cm} %\hspace{-1cm}
   =\sum^N_{i,j=1}\sum_{|\alpha|=|\beta|=m}\int_\Omega
  F^{ij}_{\alpha\beta}(x, \vec{u}(x),\cdots, D^m \vec{u}(x))D^\beta v^j\cdot D^\alpha\varphi^i dx\nonumber\\
  &&+ \sum^N_{i=1}\sum_{|\alpha|\le m-1}\int_\Omega  D^\alpha v^i\cdot D^\alpha\varphi^i dx,\label{e:6.6}\\
   (Q(\vec{u})\vec{v},\vec{\varphi})_{m,2}&&\hspace{-0.5cm}%\hspace{-1cm}
   =\sum^N_{i,j=1}\sum_{|\alpha|+|\beta|<2m}\int_\Omega
  F^{ij}_{\alpha\beta}(x, \vec{u}(x),\cdots, D^m \vec{u}(x))D^\beta v^j\cdot D^\alpha\varphi^i dx\nonumber\\
  &&-\sum^N_{i=1}\sum_{|\alpha|\le m-1}\int_\Omega  D^\alpha v^i\cdot D^\alpha\varphi^i dx,\label{e:6.7}
    \end{eqnarray}
  respectively. (If $V\subset W^{m,2}_0(\Omega, \mathbb{R}^N)$,
  the final terms in the definitions of $P$ and $Q$ can be deleted.)
    Then $D(\nabla\mathfrak{F}_V)=P+ Q$,  and
 \begin{enumerate}
\item[\rm (i)]  for any $\vec{v}\in V_0$, the map $V\ni \vec{u}\mapsto P(\vec{u})\vec{v}\in V_0$ is continuous;
\item[\rm (ii)] for every given $R>0$ there exist positive constants $C(R, n, m, \Omega)$ such that
$$
(P(\vec{u})\vec{v},\vec{v})_{m,2}\ge C\|\vec{v}\|^2_{m,2},\quad\forall \vec{v}\in V_0,\;
\forall\vec{u}\in V\;\hbox{with}\;\|\vec{u}\|_{m,2}\le R;
$$
\item[\rm (iii)] $V\ni \vec{u}\mapsto Q(\vec{u})\in\mathscr{L}(V_0)$ is continuous,
and  $Q(\vec{u})$ is completely continuous
for each $\vec{u}$;
\item[\rm (iv)] for every given $R>0$ there exist positive constants
$C_j(R, n, m, \Omega), j=1,2$ such that
\begin{eqnarray*}
&&(D(\nabla\mathfrak{F}_V)(\vec{u})[\vec{v}],\vec{v})_{m,2}\ge C_1\|\vec{v}\|^2_{m,2}-C_2\|\vec{v}\|^2_{m-1,2},\\
&&\qquad\forall \vec{v}\in V_0,\;\forall\vec{u}\in V\;\hbox{with}\;\|\vec{u}\|_{m,2}\le R.
\end{eqnarray*}
\end{enumerate}
\end{enumerate}
\end{theorem}

\begin{remark}\label{rem:BifE.3}
{\rm As noted in \cite[Remark~4.5]{Lu7}, Theorems~\ref{th:6.1} and \ref{th:6.2} and thus all results above have also more general versions in the setting of
\cite{Sma, Pa2}.  In particular, $\Omega\subset\mathbb{R}^n$ may be replaced by
 the torus $\mathbb{T}^n=\mathbb{R}^n/\mathbb{Z}^n$. In  this situation, $F$ in \textsf{Hypothesis} $\mathfrak{F}_{2,N,m,n}$ is understood as a function
 on $\mathbb{R}^n\times\prod^m_{k=0}\mathbb{R}^{N\times M_0(k)}$, which is not only $1$-periodic in each variable
$x_i$, $i=1,\cdots,n$, but also satisfies \textsf{Hypothesis} $\mathfrak{F}_{2,N,m,n}$  with $\overline\Omega=[0,1]^n$.
Then all previous results in this section  also hold if $W^{m,2}(\Omega, \mathbb{R}^N)$
is replaced by $W^{m,2}(\mathbb{T}^n, \mathbb{R}^N)$.}
\end{remark}

The tori $\mathbb{T}^n$ and $\mathbb{T}^1$ act, respectively, on $W^{m,2}(\mathbb{T}^n,\mathbb{R}^N)$  by the
isometric linear representations
\begin{eqnarray}
&&\hspace{-36pt}([t_1,\cdots,t_n]\cdot\vec{u})(x_1,\cdots,x_n)=\vec{u}(x_1+t_1,\cdots, x_n+t_n),\quad
 [t_1,\cdots,t_n]\in \mathbb{T}^n,\label{e:T^n-action1}\\
&& ([t]\cdot\vec{u})(x_1,\cdots,x_n)=\vec{u}(x_1+t,\cdots, x_n+t),\quad
 [t]\in \mathbb{T}^1.\label{e:T^n-action2}
\end{eqnarray}
The set of fixed points of the action in (\ref{e:T^n-action1}),
${\rm Fix}(\mathbb{T}^n)$, consists of all constant vector functions
from $\mathbb{T}^n$ to $\mathbb{R}^N$.
Under Hypothesis~\ref{hyp:BifE.1} with $\Omega$ replaced by $\mathbb{T}^n$,
 every critical orbit different from points in ${\rm Fix}(\mathbb{T}^n)$ must be homeomorphic to some $T^s$, $1\le s\le n$.

 If $\mathbb{S}^1=\mathbb{T}^1$ acts on $\mathbb{R}^n$ by the orthogonal
 representation, and $\Omega$ is symmetric under the action, we get a
 $\mathbb{S}^1$ action on $W^{m,2}(\Omega,\mathbb{R}^N)$
and $W^{m,2}_0(\Omega,\mathbb{R}^N)$ by
\begin{eqnarray}\label{e:S^1-action}
 ([t]\cdot\vec{u})(x)=\vec{u}([t]\cdot x),\quad
 [t]\in \mathbb{S}^1.
\end{eqnarray}

There exists a natural $\mathbb{Z}_2$-action on $W^{m,2}(\Omega,\mathbb{R}^N)$ given by
\begin{eqnarray}\label{e:Z2-action.1}
 [0]\cdot\vec{u}=\vec{u},\quad [1]\cdot\vec{u}=-\vec{u},\quad\forall\vec{u}\in W^{m,2}(\Omega,\mathbb{R}^N).
\end{eqnarray}
If $\Omega$ is symmetric with respect to the origin, there is also another obvious $\mathbb{Z}_2$-action on $W^{m,2}(\Omega,\mathbb{R}^N)$,
\begin{eqnarray}\label{e:Z2-action.2}
 [0]\cdot\vec{u}=\vec{u},\quad ([1]\cdot\vec{u})(x)=\vec{u}(-x),\quad\forall\vec{u}\in W^{m,2}(\Omega,\mathbb{R}^N).
\end{eqnarray}

\begin{hypothesis}\label{hyp:BifE.1}
{\rm Let $\Omega\subset\R^n$ be a bounded  Sobolev domain, $N\in\mathbb{N}$,
 and let functions
\begin{eqnarray}\label{e:BifE.0}
F:\overline\Omega\times\prod^m_{k=0}\mathbb{R}^{N\times M_0(k)}\to \R\quad\hbox{and}\quad
K:\overline\Omega\times\prod^{m-1}_{k=0}\mathbb{R}^{N\times M_0(k)}\to \R
\end{eqnarray}
satisfy \textsf{Hypothesis} $\mathfrak{F}_{2,N,m,n}$ and (i)-(ii) in \textsf{Hypothesis} $\mathfrak{F}_{2,N,m,n}$, respectively.  Let $V_0$ be a closed subspace of $W^{m,2}(\Omega,\mathbb{R}^N)$ containing $W^{m,2}_0(\Omega,\mathbb{R}^N)$, and $V=\vec{w}+V_0$  for some $\vec{w}\in W^{m,2}(\Omega, \mathbb{R}^N)$.}
\end{hypothesis}

Consider (generalized) bifurcation solutions of the boundary value problem
corresponding to  $V$:
\begin{eqnarray}\label{e:BifE.1}
&&\sum_{|\alpha|\le m}(-1)^{|\alpha|}D^\alpha F^i_\alpha(x, \vec{u},\cdots, D^m\vec{u})\notag\\&&=
\lambda\sum_{|\alpha|\le m-1}(-1)^{|\alpha|}D^\alpha K^i_\alpha(x, \vec{u},\cdots, D^{m-1}\vec{u}),\nonumber\\
&&\hspace{20mm}i=1,\cdots,N.
\end{eqnarray}
Call $\vec{u}\in V$ a {\it generalized solution} of (\ref{e:BifE.1}) if it is a
 critical point of the functional $\mathfrak{F}_V-\lambda \mathfrak{K}_V$, where
 $\mathfrak{F}_V$ is as in Theorem~\ref{th:6.1}, and $\mathfrak{K}_V$ is
 the restrictions of
 $\mathfrak{K}(\vec{u})=\int_\Omega K(x, \vec{u},\cdots, D^{m-1}\vec{u})dx$ to $V$.

\section{Bifurcations for quasi-linear elliptic systems
with growth restrictions}\label{sec:BifE.2}
\setcounter{equation}{0}

By the above Theorem~\ref{th:6.2} and  \cite[Theorem~7.1, Chapter~4]{Skr3}
we immediately obtain the following result, which in $N=1$ is a special case of \cite[Theorem~7.2, Chapter~4]{Skr3}.

\begin{theorem}\label{th:BifE.2}
Under Hypothesis~\ref{hyp:BifE.1}, assume $V=V_0$ and
\begin{enumerate}
\item[\rm (i)] the functionals $\mathfrak{F}_V$ and $\mathfrak{K}_V$ are even, $\mathfrak{F}_V({0})=\mathfrak{K}_V({0})=0$,
$\mathfrak{K}_V(\vec{u})\ne 0$ and $\mathfrak{K}'_V(\vec{u})\ne 0$ for any $\vec{u}\in V\setminus\{{0}\}$;
\item[\rm (ii)] $\langle\mathfrak{F}_V'(\vec{u}), \vec{u}\rangle\ge\nu(\|\vec{u}\|_{m,2})$, where $\nu(t)$ is a continuous function and positive for $t>0$;
\item[\rm (iii)] $\mathfrak{F}_V(\vec{u})\to +\infty$ as $\|\vec{u}\|_{m,2}\to\infty$.
\end{enumerate}
Then for any $c>0$ there exists at least a sequence $(\lambda_j,\vec{u}_j)\subset\mathbb{R}\times
\mathfrak{F}_V^{-1}(c)$ satisfying (\ref{e:BifE.1}).
\end{theorem}

  Corollary~3.3 in \cite{Lu8}
  and Theorem~\ref{th:6.2} directly lead to:

\begin{theorem}\label{th:BifE.3}
Under Hypothesis~\ref{hyp:BifE.1}, let $\vec{u}_0\in V$ satisfy $\mathfrak{F}'_V(\vec{u}_0)=0$ and $\mathfrak{K}'_V(\vec{u}_0)=0$.
Suppose that $(\lambda^\ast, \vec{u}_0)$
is a bifurcation point for (\ref{e:BifE.1}). Then the linear problem \medskip
 \begin{eqnarray}\label{e:BifE.3}
 &&\sum^N_{j=1}\sum_{|\alpha|,|\beta|\le m}
(-1)^{|\alpha|}D^\alpha\bigl[F^{ij}_{\alpha\beta}(x, \vec{u}_0(x),\cdots, D^m \vec{u}_0(x))D^\beta v^j\bigr]
\nonumber\\
&&=\lambda\sum^N_{j=1}\sum_{|\alpha|,|\beta|\le m-1}
(-1)^{|\alpha|}D^\alpha\bigl[K^{ij}_{\alpha\beta}(x, \vec{u}_0(x),\cdots, D^{m-1}\vec{u}_0(x))D^\beta v^j\bigr]\nonumber\\
&&\hspace{30mm}i=1,\cdots,N
 \end{eqnarray}
 with $\lambda=\lambda^\ast$  has a nontrivial solution in $V_0$,
 namely $\vec{u}_0$ is a degenerate critical point of  $\mathfrak{F}_V-\lambda^\ast\mathfrak{K}_V$.
\end{theorem}

This result cannot be derived from any one of \cite[Theorem~6.3(ii)]{Stu14A}
  and \cite[Theorem~4.2]{EvSt07}.
 In fact, define $F:\R\times V_0\to V_0$ by
  $$
  F(\lambda, \vec{u})=\nabla\mathfrak{F}_V(\vec{u}_0+ \vec{u})-\lambda\nabla\mathfrak{K}_V(\vec{u}_0+\vec{u})=0.
  $$
  By the assumption $\nabla\mathfrak{K}_V$ is of class $C^1$.
  It is claimed above Theorem~\ref{th:6.2}
   that  $\nabla\mathfrak{F}_V$ is  of class $C^{1-0}$,
   and therefore Hadamard differentiable.
   However, without further assumptions it seems impossible to check that $F$ satisfies
  other conditions of \cite[Theorem~6.3(ii)]{Stu14A} or \cite[Theorem~4.2]{EvSt07}, for example,   $\lim_{\delta\to 0}\Delta_\delta(F,\lambda)=0$ for all $\lambda$ in an open neighbourhood of $\lambda^\ast$ (see \cite[Section~3.3]{Lu8}).

The following are some sufficient criteria.
By  \cite[Corollary~4.3]{Lu8} and Theorem~\ref{th:6.2} we get:

\begin{theorem}\label{th:BifE.7}
Under Hypothesis~\ref{hyp:BifE.1}, let $\vec{u}_0\in V$ satisfy $\mathfrak{F}'_V(\vec{u}_0)=0$ and $\mathfrak{K}'_V(\vec{u}_0)=0$. Suppose that  $\lambda^\ast$ is an isolated eigenvalue of (\ref{e:BifE.3}) (this is true if
the following Hypothesis~\ref{hyp:BifE.4} is satisfied), and that
$\mathfrak{K}''_V(\vec{u}_0)$ is either semi-positive or semi-negative.
Then $(\lambda^\ast, \vec{u}_0)\in\mathbb{R}\times V$ is a bifurcation point of
(\ref{e:BifE.1}),  and  one of the following alternatives occurs:
\begin{enumerate}
\item[\rm (i)] $(\lambda^\ast, \vec{u}_0)$ is not an isolated solution of (\ref{e:BifE.1}) in  $\{\lambda^\ast\}\times V$;

\item[\rm (ii)]  for every $\lambda\in\mathbb{R}$ near $\lambda^\ast$ there is a nontrivial solution $\vec{u}_\lambda$ of (\ref{e:BifE.1}) converging
    to $\vec{u}_0$ as $\lambda\to\lambda^\ast$;

\item[\rm (iii)] there is an one-sided  neighborhood $\Lambda$ of
$\lambda^\ast$ such that for any $\lambda\in\Lambda\setminus\{\lambda^\ast\}$,
(\ref{e:BifE.1}) has at least two nontrivial solutions converging to
$\vec{u}_0$ as $\lambda\to\lambda^\ast$.
\end{enumerate}
\end{theorem}

The following hypothesis can guarantee that  every eigenvalue
of (\ref{e:BifE.3}) is isolated.

\begin{hypothesis}\label{hyp:BifE.4}
{\rm Under Hypothesis~\ref{hyp:BifE.1},  assume that $\vec{u}_0\in V$ satisfy $\mathfrak{F}'_V(\vec{u}_0)=0$ and $\mathfrak{K}'_V(\vec{u}_0)=0$,
and that the linear problem
 \begin{eqnarray}\label{e:BifE.4}
 \sum^N_{j=1}\sum_{|\alpha|,|\beta|\le m}
(-1)^{|\alpha|}D^\alpha\bigl[F^{ij}_{\alpha\beta}(x, \vec{u}_0(x),\cdots, D^m \vec{u}_0(x))D^\beta v^j\bigr]
=0,\; i=1,\cdots,N
 \end{eqnarray}
has no nontrivial solutions in $V_0$, that is, $\mathfrak{F}''_V(\vec{u}_0)\in\mathscr{L}(V_0)$ has
a bounded linear inverse operator.}
\end{hypothesis}

Under Hypothesis~\ref{hyp:BifE.4}, by the arguments above \cite[Theorem~4.2]{Lu8}
 all eigenvalues of (\ref{e:BifE.3}) form a discrete subset of $\mathbb{R}$,
$\{\lambda_j\}^\infty_{j=1}$, which contains no zero and satisfies $|\lambda_j|\to\infty$
as $j\to\infty$; moreover, each $\lambda_j$ has finite multiplicity.
Let $E_j\subset V_0$ be the eigenspace of  (\ref{e:BifE.3}) associated
with the eigenvalue $\lambda_j$, $j=1,2,\cdots$.
 From \cite[Corollary~3.7]{Lu8}
 (resp. \cite[Corollary~4.4]{Lu8}) and Theorem~\ref{th:6.2} we derive
 the first (resp. second) conclusion of the following theorem.

\begin{theorem}\label{th:BifE.5}
Under Hypothesis~\ref{hyp:BifE.4}, for an eigenvalue $\lambda_j$
of (\ref{e:BifE.3}) as above, the following holds:
\begin{enumerate}
\item[\rm (1)]  $(\lambda_j, \vec{u}_0)\in\mathbb{R}\times V$ is a
bifurcation point of (\ref{e:BifE.1}) provided  that $\mathfrak{F}''_V(\vec{u}_0)$
commutes with $\mathfrak{K}''_V(\vec{u}_0)$ and that the positive and negative
indexes of inertia of the restriction of $\mathfrak{F}''_V(\vec{u}_0)$ to $E_j$
  are different.
\item[\rm (2)] The conclusions in Theorem~\ref{th:BifE.7} hold if
one of the following conditions is satisfied:
  \begin{enumerate}
\item[\rm (i)] $\mathfrak{F}''_V(\vec{u}_0)$ commutes with $\mathfrak{K}''_V(\vec{u}_0)$, and
$\mathfrak{F}''_V(\vec{u}_0)$ is either positive or negative on $E_j$;
\item[\rm (ii)] $\mathfrak{F}''_V(\vec{u}_0)$ is  positive definite, i.e.,
for some $c>0$,
\begin{eqnarray*}
 \quad\sum^N_{i,j=1}\sum_{|\alpha|,|\beta|\le m}\int_\Omega
  F^{ij}_{\alpha\beta}(x, \vec{u}_0(x),\cdots, D^m \vec{u}_0(x))D^\beta v^j\cdot D^\alpha v^i dx\ge c\|\vec{v}\|_{m,2},\;\;\forall \vec{v}\in V.
 \end{eqnarray*}
\end{enumerate}
  \end{enumerate}
 \end{theorem}

\begin{remark}\label{rmk:BifE.6}
{\rm  When $N=1$, $V=W_0^{m,2}(\Omega)$, and for some $c>0$ it holds  that
\begin{eqnarray}\label{e:BifE.7}
(\mathfrak{F}'_V({u}),{u})_{m,2}\ge c\|u\|^2_{m,2}
\end{eqnarray}
near $0\in V$, it was proved in \cite[Chap.1, Theorem~3.5]{Skr2} that
$(\lambda^\ast, 0)$ is a bifurcation point of (\ref{e:BifE.1}) if and only if
 $\lambda^\ast$ is  an eigenvalue of (\ref{e:BifE.3}) with ${u}=0$.
Since $\mathfrak{F}'_V(0)=0$, it is clear that (\ref{e:BifE.7})
implies $\mathfrak{F}''_V(0)$ to be positive definite, that is,
the condition (ii) in Theorem~\ref{th:BifE.5}
is satisfied at $u_0=0$. Hence Theorem~\ref{th:BifE.5}
is a significant generalization of \cite[Chap.1, Theorem~3.5]{Skr2}.}
\end{remark}

Corollary~5.10 in \cite{Lu8}
and Theorem~\ref{th:6.2} yield the following two results.

\begin{theorem}\label{th:BifE.7.1}
Let the assumptions of either Theorem~\ref{th:BifE.5}(2) or Theorem~\ref{th:BifE.7}
hold with $V=V_0$ and $\vec{u}_0={0}$.
Suppose also that both ${F}(x,\xi)$ and ${K}(x,\xi)$ are even with respect to $\xi$,
and that $E_{\lambda^\ast}$ denotes the solution space  of  the linear problem (\ref{e:BifE.3}) with $\lambda=\lambda^\ast$
 in $V_0$.   Then  one of the following alternatives holds:
 \begin{enumerate}
\item[\rm (i)] $(\lambda^\ast, {0})$ is not an isolated solution of (\ref{e:BifE.1})
 in  $\{\lambda^\ast\}\times V$;
 \item[\rm (ii)] there exist left and right  neighborhoods $\Lambda^-$ and $\Lambda^+$
  of $\lambda_\ast$ in $\mathbb{R}$
and integers $n^+, n^-\ge 0$, such that $n^++n^-\ge\dim E_{\lambda^\ast}$
and for $\lambda\in\Lambda^-\setminus\{\lambda^\ast\}$ (resp. $\lambda\in\Lambda^+\setminus\{\lambda^\ast\}$),
(\ref{e:BifE.1}) has at least $n^-$ (resp. $n^+$) distinct pairs of solutions of form $\{\vec{u},-\vec{u}\}$  different from ${0}$, which converge to
 ${0}$ as $\lambda\to\lambda^\ast$.
\end{enumerate}
  \end{theorem}

\begin{theorem}\label{th:BifE.7.2}
Let the assumptions of either Theorem~\ref{th:BifE.5} or Theorem~\ref{th:BifE.7}
hold with $V=V_0$ and $\vec{u}_0={0}$, and let $E_{\lambda^\ast}$ be the solution
space  of  the linear problem (\ref{e:BifE.3}) with $\lambda=\lambda^\ast$
 in $V_0$.  Then:
\begin{enumerate}
\item[\rm (I)]  If $\Omega$ is symmetric with respect to the origin,
 both $\mathfrak{F}_{V_0}$ and $\mathfrak{K}_{V_0}$ are invariant for the $\mathbb{Z}_2$-action  in (\ref{e:Z2-action.2}), and (\ref{e:BifE.3}) with $\lambda=\lambda^\ast$ has no nontrivial solutions
 $\vec{u}\in V_0$ satisfying $\vec{u}(-x)=\vec{u}(x)\;\forall x\in\Omega$,
 then the conclusions in  Theorem~\ref{th:BifE.7.1},
 after ``pairs of solutions of form $\{\vec{u},-\vec{u}\}$" being changed into
 ``pairs of solutions of form $\{\vec{u}(\cdot), \vec{u}(-\cdot)\}$", still holds.
\item[\rm (II)] Let $\mathbb{S}^1$ act on $\mathbb{R}^n$ by the orthogonal representation, and $\Omega$
be symmetric under the action, let both $\mathfrak{F}_{V_0}$ and $\mathfrak{K}_{V_0}$
 are invariant for the $\mathbb{S}^1$ action on $V_0$
 in (\ref{e:S^1-action}). If the fixed point set of the induced $S^1$-action in $E_{\lambda^\ast}$ is $\{0\}$,
  then   the conclusions of Theorem~\ref{th:BifE.7.1},
 after ``$n^++n^-\ge\dim E_{\lambda^\ast}$" and ``pairs of solutions of form $\{\vec{u},-\vec{u}\}$" being changed into ``critical $\mathbb{S}^1$-orbits" and ``$n^++n^-\ge\frac{1}{2}\dim E_{\lambda^\ast}$", hold.
 \end{enumerate}
 \end{theorem}

\begin{remark}\label{rm:BifE.8}
{\rm By Remark~\ref{rem:BifE.3}, some of the above results also hold if
 $\Omega\subset\mathbb{R}^n$ is replaced by the torus $\mathbb{T}^n=\mathbb{R}^n/\mathbb{Z}^n$.}
\end{remark}

The following is a result associated with Theorems~5.4.2 and 5.7.4 in \cite{EKBB}.

\begin{theorem}\label{th:BifE.9}
 Let $\Omega\subset\R^n$ be a bounded  domain with $C^1$ boundary, $N\in\mathbb{N}$. Suppose that
 $$
\overline\Omega\times\prod^m_{k=0}\mathbb{R}^{N\times M_0(k)}\times [0, 1]\ni (x,
\xi,\lambda)\mapsto F(x,\xi;\lambda)\in\R
$$
is  differentiable with respect to $\lambda$, and satisfies the following conditions:
 \begin{enumerate}
 \item[\rm (i)] All $F(\cdot;\lambda)$ satisfy Hypothesis~$\mathfrak{F}_{2,N,m,n}$,
  and the inequalities in (\ref{e:6.1}) and (\ref{e:6.2}) are uniformly satisfied for all $\lambda\in [0,1]$.
     $D_\lambda F(x,\xi;\lambda)$ is differentiable in each $\xi^i_\alpha$,
    each  $F^i_\alpha(x,\xi;\lambda)$ is differentiable in $\lambda$, and
    {\small $D_\lambda F^i_\alpha(x,\xi;\lambda)=D_{\xi^i_\alpha}D_\lambda F(x,\xi;\lambda)$}.

 \item[\rm (ii)] For $2'_\alpha$ in Hypothesis~$\mathfrak{F}_{2,N,m,n}$, it holds that
 $$
\sup_{|\alpha|\le m}\sup_{1\le i\le N}\sup_\lambda\int_\Omega \left[|D_\lambda F(x,0;\lambda)|+ |D_\lambda F^i_{\alpha}(x, 0;\lambda)|^{2'_\alpha}\right]dx<\infty.
$$
 \item[\rm (iii)] For all $i=1,\cdots,N$ and $|\alpha|\le m$,
 \begin{eqnarray*}
&&|D_\lambda F^i_\alpha(x,\xi;\lambda)|\le|D_\lambda F^i_\alpha(x,0;\lambda)|\\
&&+ \mathfrak{g}(\sum^N_{k=1}|\xi_\circ^k|)\sum_{|\beta|<m-n/2}\bigg(1+
\sum^N_{k=1}\sum_{m-n/2\le |\gamma|\le
m}|\xi^k_\gamma|^{2_\gamma }\bigg)^{2_{\alpha\beta}}\nonumber\\
&&+\mathfrak{g}(\sum^N_{k=1}|\xi^k_\circ|)\sum^N_{l=1}\sum_{m-n/2\le |\beta|\le m} \bigg(1+
\sum^N_{k=1}\sum_{m-n/2\le |\gamma|\le m}|\xi^k_\gamma|^{2_\gamma }\bigg)^{2_{\alpha\beta}}|\xi^l_\beta|;
\end{eqnarray*}
 where $\mathfrak{g}:[0,\infty)\to\mathbb{R}$ is a continuous, positive,
 nondecreasing function,  and is constant if $m<n/2$.
 \end{enumerate}
 Let $V_0$ be a closed subspace of $W^{m,2}(\Omega, \mathbb{R}^N)$ and $V=\vec{w}+V_0$
 for some $\vec{w}\in W^{m,2}(\Omega, \mathbb{R}^N)$. For each $\lambda\in [0,1]$,
  suppose that the functional
  $$
V\ni\vec{u}\mapsto\mathfrak{F}_\lambda(\vec{u})=\int_\Omega F(x, \vec{u},\cdots, D^m\vec{u};\lambda)dx
$$
has a critical point $\vec{u}_\lambda$ such that
 $[0, 1]\ni\lambda\mapsto \vec{u}_\lambda\in V$ is continuous. Then
one of the following alternatives occurs:
\begin{enumerate}
\item[\rm (I)] There exists certain $\lambda_0\in [0,1]$ such that $(\lambda_0, \vec{u}_{\lambda_0})$
           is a bifurcation point of $\nabla\mathfrak{F}_\lambda(\vec{u})=0$.
\item[\rm (II)] There exists $\epsilon>0$ such that each $\vec{u}_\lambda$ is a unique critical point of $\mathfrak{F}_\lambda$
in $B_V(\vec{u}_\lambda,\epsilon)$ and
$$
C_\ast(\mathfrak{F}_\lambda, \vec{u}_\lambda;{\bf K})=C_\ast(\mathfrak{F}_0, \vec{u}_0;{\bf K})\quad
\forall \lambda\in [0,1];
$$
 moreover, $\vec{u}_\lambda$ is a local minimizer of $\mathfrak{F}_\lambda$ if and only if $\vec{u}_0$ is a local minimizer of $\mathfrak{F}_0$.
 \end{enumerate}
\end{theorem}

\begin{remark}\label{rm:BifE10}
{\rm {\bf (a)}  When $V=W^{m,2}_0(\Omega, \mathbb{R})$,  Theorem  \ref{th:BifE.9}
generalize Theorems 5.4.2 and 5.7.4 in \cite{EKBB} in some sense.\\
{\bf (b)} Let the assumptions (i)-(ii) in Theorem~\ref{th:BifE.9} be satisfied.
Then the assumption (iii) in Theorem~\ref{th:BifE.9} can be satisfied provided that the following conditions hold.\\
{\bf (b.1)} Each $D_\lambda F^i_\alpha(x,\xi;\lambda)$ is differentiable in each $\xi^j_\beta$,
each  $D_{\xi^j_\beta}F^i_\alpha(x,\xi;\lambda)$ is differentiable in $\lambda$,
and $D_{\xi^j_\beta}D_\lambda F^i_\alpha(x,\xi;\lambda)=
     D_\lambda D_{\xi^j_\beta}F^i_\alpha(x,\xi;\lambda)= D_\lambda F^{ij}_{\alpha\beta}(x,\xi;\lambda)$.\\
{\bf (b.2)} There exists a continuous, positive, nondecreasing functions $\mathfrak{g}^\ast$ such that
\begin{eqnarray*}
 |D_\lambda F^{ij}_{\alpha\beta}(x,\xi;\lambda)|\le
\mathfrak{g}^\ast(\sum^N_{k=1}|\xi_\circ^k|)\left(1+
\sum^N_{k=1}\sum_{m-n/2\le |\gamma|\le
m}|\xi^k_\gamma|^{2_\gamma}\right)^{2_{\alpha\beta}}
\end{eqnarray*}
for all $(x,\xi,\lambda)$ and  $i,j=1,\cdots,N$ and $|\alpha|, |\beta|\le m$.

In fact, by the mean value theorem  we get  $t\in (0,1)$ such that
\begin{eqnarray*}
&&|D_\lambda F^i_\alpha(x,\xi;\lambda))|-|D_\lambda F^i_\alpha(x,0;\lambda)|\le \sum^N_{j=1}\sum_{|\beta|\le m}
|D_{\xi^j_\beta}D_\lambda F^i_\alpha(x,t\xi;\lambda))|\cdot|\xi^j_\beta|\\
&=&\sum^N_{j=1}\sum_{|\beta|\le m}
|D_\lambda F^{ij}_{\alpha\beta}(x,t\xi;\lambda))|\cdot|\xi^j_\beta|\qquad\hbox{(by (b.1))}\\
&\le&\sum^N_{j=1} \sum_{|\beta|\le m} \mathfrak{g}^\ast(\sum^N_{k=1}|t\xi_\circ^k|)\Bigg(1+
\sum^N_{k=1}\sum_{m-n/2\le |\gamma|\le
m}|t\xi^k_\gamma|^{2_\gamma}\Bigg)^{2_{\alpha\beta}}|\xi^j_\beta|\\
&\le&\sum^N_{j=1} \sum_{|\beta|\le m} \mathfrak{g}^\ast(\sum^N_{k=1}|\xi_\circ^k|)\Bigg(1+
\sum^N_{k=1}\sum_{m-n/2\le |\gamma|\le
m}|\xi^k_\gamma|^{2_\gamma}\Bigg)^{2_{\alpha\beta}}|\xi^j_\beta|\\
&\le&\sum^N_{j=1} \sum_{|\beta|<m-n/2} \mathfrak{g}^\ast(\sum^N_{k=1}|\xi_\circ^k|)\Bigg(1+
\sum^N_{k=1}\sum_{m-n/2\le |\gamma|\le
m}|\xi^k_\gamma|^{2_\gamma}\Bigg)^{2_{\alpha\beta}}|\xi^j_\beta|\\
&+&\sum^N_{j=1} \sum_{m-n/2\le |\beta|\le m} \mathfrak{g}^\ast(\sum^N_{k=1}|\xi_\circ^k|)\Bigg(1+
\sum^N_{k=1}\sum_{m-n/2\le |\gamma|\le
m}|\xi^k_\gamma|^{2_\gamma}\Bigg)^{2_{\alpha\beta}}|\xi^j_\beta|\\
&\le&\sum^N_{j=1} \sum_{|\beta|<m-n/2} \mathfrak{g}(\sum^N_{k=1}|\xi_\circ^k|)\Bigg(1+
\sum^N_{k=1}\sum_{m-n/2\le |\gamma|\le
m}|\xi^k_\gamma|^{2_\gamma}\Bigg)^{2_{\alpha\beta}}\\
&+&\sum^N_{j=1} \sum_{m-n/2\le |\beta|\le m} \mathfrak{g}(\sum^N_{k=1}|\xi_\circ^k|)\Bigg(1+
\sum^N_{k=1}\sum_{m-n/2\le |\gamma|\le
m}|\xi^k_\gamma|^{2_\gamma}\Bigg)^{2_{\alpha\beta}}|\xi^j_\beta|,
\end{eqnarray*}
where $\mathfrak{g}(t)=(1+t)\mathfrak{g}^\ast(t)$.
}
\end{remark}

In order to prove Theorem~\ref{th:BifE.9} we need the following three lemmas, whose proofs are involved and are given in Appendix~\ref{app:B}.
The condition (iii) in Theorem~\ref{th:BifE.9} is only used in the proofs of the first two.

\begin{lemma}\label{lem:BifE.10}
Under the assumptions of Theorem~\ref{th:BifE.9},
 there exists a  continuous, positive, nondecreasing function
$\widehat{\mathfrak{g}}:[0,\infty)\to\mathbb{R}$ such that for all $(x,\xi,\lambda)$,
\begin{eqnarray*}
&&\hspace{-0.5cm}|D_\lambda F(x,\xi;\lambda)|\le|D_\lambda F(x,0;\lambda)|+
\Big(\sum^N_{k=1}|\xi^k_\circ|\Big)\sum^N_{i=1}\sum_{|\alpha|<m-n/2}|D_\lambda F^i_{\alpha}(x, 0;\lambda)|\\
&+&\hspace{-0.3cm}\sum^N_{i=1}\sum_{m-n/2\le|\alpha|\le m}|D_\lambda F^i_{\alpha}(x, 0;\lambda)|^{2'_\alpha}+ \widehat{\mathfrak{g}}(\sum^N_{k=1}|\xi^k_\circ|)
\bigg(1+\sum^N_{l=1}\sum_{m-n/2\le|\alpha|\le m}|\xi^l_\alpha|^{2_\alpha}\bigg).
\end{eqnarray*}
\end{lemma}

\begin{lemma}\label{lem:BifE.11}
Under the assumptions of Theorem~\ref{th:BifE.9},
 there exists a constant $C(m,n,N,\Omega)>0$ such that for $\vec{u}\in V$
 and $\lambda_i\in [0,1]$, $i=1,2$,
\begin{eqnarray*}
&&\|\nabla\mathfrak{F}_{\lambda_1}(\vec{u})- \nabla\mathfrak{F}_{\lambda_2}(\vec{u})\|_{m,2}\\
&\le&C(m,n,N,\Omega)|\lambda_1-\lambda_2|\bigg[\sum^N_{j=1}\sum_{|\alpha|\le m}
\left(\int_\Omega |D_\lambda F^j_\alpha(x,0;\lambda)|^{2'_\alpha}dx\right)^{1/2'_\alpha}\nonumber\\
&&\hspace{30mm}+\sup_{k<m-n/2}\mathfrak{g}(\|\vec{u}\|_{C^k}|)(1+\sum_{m-n/2\le |\gamma|\le
m}\|\vec{u}\|_{m,2}^{2_\gamma})\nonumber\\
&&\hspace{30mm}+\sup_{k<m-n/2}\mathfrak{g}(\|\vec{u}\|_{C^k}|)(1+\sum_{m-n/2\le |\gamma|\le
m}\|\vec{u}\|_{m,2}^{2_\gamma})\|\vec{u}\|_{m,2}\bigg].
\end{eqnarray*}
\end{lemma}

\begin{lemma}\label{lem:BifE.12}
Under the assumptions of Theorem~\ref{th:BifE.9}, if $(\vec{u}_k)\in \bar{B}_V(0,R)$
 weakly converges to some $\vec{u}\in \bar{B}_V(0,R)$
and $\nabla\mathfrak{F}_{\lambda}(\vec{u}_k)\to 0$ for some $\lambda\in [0,1]$,
then $\vec{u}_k\to\vec{u}$. In particular, all $\mathfrak{F}_{\lambda}$
satisfy the (PS) condition on  $\bar{B}_V(0,R)$.
\end{lemma}

\begin{proof}[Proof of Theorem~\ref{th:BifE.9}]
{\bf Step 1} ({\it Prove that  the map  $[0,1]\ni\lambda\mapsto{\mathfrak{F}}_\lambda$
is continuous in $C^1(\bar{B}_V(0, R))$ for any given $R>0$}).
For $|\alpha|<m-n/2$, since $W^{m-|\alpha|,2}(\Omega)\hookrightarrow C^0(\overline{\Omega})$,  for every  $R>0$,  we have a constant
$C=C(m,n,N,\Omega, R)>0$ such that
$$
\sup\bigg\{\sum_{|\alpha|<m-n/2}|D^\alpha\vec{u}(x)|\,\bigg|\, x\in\Omega\bigg\}<C,
\quad\forall \vec{u}\in W^{m,2}(\Omega,\mathbb{R}^N)\;\hbox{with}\;\|\vec{u}\|_{m,2}\le 2R.
$$
It follows from this and Lemma~\ref{lem:BifE.10} that
for any $0\le\lambda_1<\lambda_2\le 1$,
\begin{eqnarray*}
&&|\mathfrak{F}_{\lambda_1}(\vec{u})- \mathfrak{F}_{\lambda_2}(\vec{u})|\le |\lambda_2-\lambda_1|\int_\Omega\sup_\lambda
|D_\lambda F(x, \vec{u},\cdots, D^m\vec{u};\lambda)|dx\\
&&\le |\lambda_2-\lambda_1|\biggl[\sup_\lambda\int_\Omega |D_\lambda F(x,0;\lambda)|dx +
C\sum^N_{i=1}\sum_{|\alpha|<m-n/2}\sup_\lambda\int_\Omega |D_\lambda F^i_{\alpha}(x, 0;\lambda)|dx\\
&&+\sum^N_{i=1}\sum_{m-n/2\le|\alpha|\le m}\sup_\lambda\int_\Omega
|D_\lambda F^i_{\alpha}(x, 0;\lambda)|^{2'_\alpha}dx\\
&&\hspace{20mm}+ \widehat{\mathfrak{g}}(C)\int_\Omega\bigg(1+\sum^N_{l=1}\sum_{m-n/2\le|\alpha|\le m}|D^\alpha u^l|^{2_\alpha}\bigg)dx
\biggr].
\end{eqnarray*}
Note that $W^{m-|\alpha|,2}(\Omega)\hookrightarrow L^{2_\alpha}(\Omega)$
for $m-n/2\le|\alpha|\le m$. We get
$$
\int_\Omega\bigg(\sum^N_{l=1}\sum_{m-n/2\le|\alpha|\le m}|D^\alpha u^l|^{2_\alpha}\bigg)dx\le
C'\sum_{m-n/2\le|\alpha|\le m}(\|\vec{u}\|_{m,2})^{2_\alpha}
$$
for some constant $C'=C'(m,n,N,\Omega)>0$. It follows that
$[0,1]\ni\lambda\mapsto \mathfrak{F}_\lambda$
is continuous in $C^0(\bar{B}_V(0, 2R))$.
Combing with Lemma~\ref{lem:BifE.11} we get the desired claim.

\vspace{4pt}\noindent
\noindent{\bf Step 2} ({\it Prove that (II) is true under the assumption
that (I) does not hold}).  Then for each $\mu\in [0,1]$ we have neighborhoods
of it and $u_\mu$ in $[0,1]$ and $V$  respectively,
 $\mathscr{N}_\mu$ and $\mathscr{O}_\mu$, such that $u_\lambda$ is
 a unique critical point of $\mathfrak{F}_\lambda$  in
the closure $\overline{\mathscr{O}_\mu}$ of  $\mathscr{O}_\mu$ for each $\lambda\in\mathscr{N}_\mu$.
 Moreover, we can assume that each  $\mathscr{N}_\mu$ is an interval.
   Since $[0, 1]\ni\lambda\mapsto \vec{u}_\lambda\in V$ is continuous,
we can take $R>0$ such that all $u_\lambda$ and $\overline{\mathscr{O}_\mu}$
are contained in $B_V(0, R)$.

Put $\widehat{\mathfrak{F}}_\lambda(\vec{u}):=\mathfrak{F}_\lambda(\vec{u}+\vec{u}_\lambda)$.
By Lemma~\ref{lem:BifE.12}, for any $R>0$  all $\widehat{\mathfrak{F}}_\lambda$
satisfy the (PS) condition on  $\bar{B}_V(0,R)$. Since  $\|\vec{u}_\lambda\|_{m,2}\le R$ for all $\lambda\in [0, 1]$, by Step 1 we deduce that
the map  $[0,1]\ni\lambda\mapsto\widehat{\mathfrak{F}}_\lambda$
is continuous in $C^1(\bar{B}_V(0, R))$.

Shrinking the above $\mathscr{N}_\mu$ (if necessary) we have $\epsilon>0$ such that
$\vec{u}_\lambda+ B_V(0,\epsilon)\subset \mathscr{O}_\mu$ for all $\lambda\in \mathscr{N}_\mu$.
Then $\widehat{\mathfrak{F}}_\lambda$ has a unique critical point $0$ in $B_V(0,\epsilon)$
 for each $\lambda\in \mathscr{N}_\mu$.
Applying \cite[Theorem~2.2]{Lu8} to the family
$\{\widehat{\mathfrak{F}}_\lambda|_{B_V(0,\epsilon)}\,|\,\lambda\in \mathscr{N}_\mu\}$
we deduce that critical groups  $C_\ast(\mathfrak{F}_\lambda, \vec{u}_\lambda;{\bf K})=
C_\ast(\widehat{\mathfrak{F}}_\lambda, 0;{\bf K})$ are independent of choices of
$\lambda$ in $\mathscr{N}_\mu$. Note that $[0,1]$ may be covered by finitely many
 $\mathscr{N}_\mu$.  We arrive at
 $C_\ast(\mathfrak{F}_\lambda, \vec{u}_\lambda;{\bf K})
 =C_\ast(\mathfrak{F}_0, \vec{u}_0;{\bf K})$
for all $\lambda\in [0,1]$, and after shrinking $\epsilon>0$ the ball
 $B_V(\vec{u}_\lambda,\epsilon)$ contains a unique critical point $\vec{u}_\lambda$
 of $\mathfrak{F}_\lambda$  for any $\lambda\in [0,1]$.

 For the second claim in (II), it suffices to prove that
 $0\in V$ is a local minimizer of  $\widehat{\mathfrak{F}}_0$
 provided $0\in V$ is a local minimizer of $\widehat{\mathfrak{F}}_\lambda$.
 Since  $0\in V$ is an isolated critical point of $\widehat{\mathfrak{F}}_\lambda$,
 by Example~1 in \cite[page 33]{Ch} we have
 $C_q(\widehat{\mathfrak{F}}_\lambda, 0;{\bf K})=\delta_{q0}{\bf K}$
 for each $q\in\mathbb{N}_0$.   It follows that
 $C_q(\widehat{\mathfrak{F}}_0, 0;{\bf K})=\delta_{q0}{\bf K}$
 for each $q\in\mathbb{N}_0$.
 By \cite[Theorem~2.3]{Lu7}  this means that the Morse
  index of $\widehat{\mathfrak{F}}_0$ at $0\in V$  must be zero, and
 $C_q(\widehat{\mathfrak{F}}^\circ_0, 0;{\bf K})=\delta_{q0}{\bf K}$
 for each $q\in\mathbb{N}_0$, where $\widehat{\mathfrak{F}}^\circ_0$ is the finite dimensional reduction  of $\widehat{\mathfrak{F}}_0$ near
 $0\in V^0:={\rm Ker}(D(\nabla\mathfrak{F}_\lambda)(\vec{u}_0))$.
   By Example~4 in \cite[page 43]{Ch},  $0\in V^0$ is a local minimizer of
 $\widehat{\mathfrak{F}}^\circ_0$.
  From this and \cite[Theorem~2.2]{Lu7}
  it follows that  $0\in V$ must be  a local minimizer of
  $\widehat{\mathfrak{F}}_0$.
 \end{proof}
 %\hfill$\Box$\vspace{2mm}

\section{Bifurcations for quasi-linear elliptic systems
without growth restrictions}\label{sec:BifE.3}
\setcounter{equation}{0}

Now let us begin with some applications of results in \cite[Section~6]{Lu8}.
For simplicity we only consider the Dirichlet boundary conditions. So
till the end of this subsection, we take the Hilbert space $H:=W^{m,2}_0(\Omega,\mathbb{R}^N)$ with the usual inner product (\ref{e:6.2.2}).
 The following special case of \cite[Theorem 6.4.8]{Mor} is
 key for our arguments in this subsection.

\begin{proposition}\label{prop:BifE.9.1}
  For a real $p\ge 2$ and an integer $k\ge m+\frac{n}{p}$,
  let $\Omega\subset\R^n$ be a bounded  domain with boundary of class $C^{k-1,1}$, $N\in\mathbb{N}$, and let bounded and measurable functions on $\overline{\Omega}$,
    $A^{ij}_{\alpha\beta}$, $i,j=1,\cdots,N$, $|\alpha|,|\beta|\le m$,  fulfill the following conditions:
    \begin{enumerate}
\item[\rm (i)]  $A^{ij}_{\alpha\beta}\in C^{k+|\alpha|-2m-1,1}(\overline{\Omega})$ if $2m-k<|\alpha|\le m$;
\item[\rm (ii)] there exists $c_0>0$ such that
$$
\sum^N_{i,j=1}\sum_{|\alpha|=|\beta|=m}\int_\Omega
A^{ij}_{\alpha\beta}\eta^i_\alpha\eta^j_\beta\ge c_0\sum^N_{i=1}\sum_{|\alpha|=m}|\eta^i_\alpha|^2,\quad\forall \eta\in\mathbb{R}^{N\times M_0(m)}.
$$
\end{enumerate}
Suppose that $\vec{u}=(u^1,\cdots,u^N)\in W^{m,2}_0(\Omega,\mathbb{R}^N)$
and $\lambda\in (-\infty, 0]$ satisfy
$$
\sum^N_{i,j=1}\sum_{|\alpha|,|\beta|\le m}\int_\Omega
(A^{ij}_{\alpha\beta}-\lambda\delta_{ij}\delta_{\alpha\beta})D^\beta u^i\cdot D^\alpha v^j dx=0,\quad\forall v\in W^{m,2}_0(\Omega,\mathbb{R}^N).
$$
Then $\vec{u}\in W^{k,p}(\Omega,\mathbb{R}^N)$. Moreover, for
$f_j=\sum_{|\alpha|\le m}(-1)^{|\alpha|}D^\alpha f^j_\alpha$,
where
$$
f^j_\alpha\in\left\{\begin{array}{ll}
 W^{k-2m+|\alpha|,p}(\Omega),&\quad\hbox{if}\;|\alpha|>2m-k,\\
  L^p(\Omega),&\quad\hbox{if}\;|\alpha|\le 2m-k,
  \end{array}\right.
  $$
   if $\vec{u}\in W^{m,2}_0(\Omega,\mathbb{R}^N)$ satisfies \medskip
$$
\int_\Omega\sum^N_{j=1}\sum_{|\alpha|\le m}\left[\sum^N_{i=1}\sum_{|\beta|\le m}A^{ij}_{\alpha\beta}D^\beta u^i-f^j_\alpha \right]D^\alpha v^jdx=0,\quad\forall v\in W^{m,2}_0(\Omega,\mathbb{R}^N),
$$
then there also holds $\vec{u}\in W^{k,p}(\Omega,\mathbb{R}^N)$.
\end{proposition}

For $s<0$ let $W^{s,p}(\Omega, \mathbb{R}^N)=[W_0^{-s,p'}(\Omega, \mathbb{R}^N)]^\ast$ as usual, where $p'=p/(p-1)$.
Note that the $m$th power of the Laplace operator, $\triangle^m: W^{k,p}(\Omega,\mathbb{R}^N)\cap W^{m,2}_0(\Omega,\mathbb{R}^N)\\
\to W^{k-2m,p}(\Omega, \mathbb{R}^N)$,
 is an isomorphism, and that its inverse, denoted by $\triangle^{-m}$,
 is from $W^{k-2m,p}(\Omega, \mathbb{R}^N)$ to
 $W^{k,p}(\Omega,\mathbb{R}^N)\cap W^{m,2}_0(\Omega,\mathbb{R}^N)$.

\begin{proposition}\label{prop:Jia1}
Let $k> m+\frac{n}{p}$ and $\Omega\subset\R^n$ be as in Proposition~\ref{prop:BifE.9.1}.
Consider the Banach subspace  of $W^{k,p}(\Omega,\mathbb{R}^N)$,
  $X_{k,p}:=W^{k,p}(\Omega,\mathbb{R}^N)\cap W^{m,2}_0(\Omega,\mathbb{R}^N)$, which
    can be continuously embedded to the space $C^m(\overline{\Omega}, \mathbb{R}^N)$
   and is dense in $H$. Let $\ell\in\N$ and $Z\subset\R^\ell$ be a compact domain.
   Suppose for some $r\in\mathbb{N}$  that
   ${\bf F}:\overline\Omega\times\prod^m_{k=0}\mathbb{R}^{N\times M_0(k)}\times Z\to \R$
   is $C^{k-m+2r}$. Then the functional $\mathcal{F}:Z\times X_{k,p}\to\R$ defined by
 \begin{eqnarray}\label{e:Jia1}
\mathcal{F}(\lambda,\vec{u})=\mathcal{F}_\lambda(\vec{u})=
\int_\Omega {\bf F}(x,\vec{u},D\vec{u},\cdots,D^m\vec{u};\lambda)dx
\end{eqnarray}
is $C^{r}$. If all derivatives
${\bf F}^i_\alpha(x,\xi;\lambda)=D_{\xi^i_\alpha}{\bf F}(x,\xi;\lambda)$
 are also $C^{k-m+2r}$ with respect to $(x,\xi,\lambda)$, then the map
 $Z\times X_{k,p}\ni (\lambda, \vec{u})\to A(\lambda,\vec{u})\in X_{k,p}$
 given by
 \begin{eqnarray}\label{e:Jia2}
 &&(A(\lambda,\vec{u}))^i=\triangle^{-m}\sum_{|\alpha|\le m}(-1)^{m+|\alpha|}
 D^\alpha({\bf F}^i_\alpha(\cdot, \vec{u}(\cdot),\cdots, D^m \vec{u}(\cdot);\lambda)),\notag\\
 &&\quad i=1,\cdots,N,
 \end{eqnarray}
is of class $C^r$.  Moreover, for $(\lambda,\vec{u})\in Z\times X_{k,p}$ and $\vec{v},\vec{w}\in X_{k,p}$ it holds that
 \begin{eqnarray}\label{e:Jia3}
d\mathcal{F}_\lambda(\vec{u})[\vec{v}]=(A_\lambda(\vec{u}),\vec{v})_H\quad\hbox{and}\quad d^2\mathcal{F}_\lambda(\vec{u})[\vec{v},\vec{w}]=(B_\lambda(\vec{u})\vec{v},
\vec{w})_H,
\end{eqnarray}
where  $A_\lambda(\vec{u})=A(\lambda,\vec{u})$ and $B_\lambda(\vec{u})=B(\lambda,\vec{u})\in\mathscr{L}_s(H)$ is defined by
\begin{eqnarray}\label{e:Jia4}
(B(\lambda, \vec{u})\vec{v})^i&=& \triangle^{-m}\sum^N_{j=1}\sum_{\scriptsize\begin{array}{ll}
   &|\alpha|\le m,\\
   &|\beta|\le m
   \end{array}}(-1)^{m+|\alpha|}
D^\alpha({\bf F}^{ij}_{\alpha\beta}(\cdot, \vec{u}(\cdot);\lambda)D^\beta v^j)
\end{eqnarray}
for $i=1,\cdots,N$.
So $B_\lambda(\vec{u})(X_{k,p})\subset X_{k,p}$, $B_\lambda(\vec{u})|_{X_{k,p}}\in\mathscr{L}(X_{k,p})$,
and $A'_\lambda(\vec{u})=B_\lambda(\vec{u})|_{X_{k,p}}$.
\end{proposition}

We  postpone its proof to Appendix~\ref{app:A}.

\begin{proposition}\label{prop:Jia2}
Let $k> m+\frac{n}{p}$, $\Omega\subset\R^n$, $\ell\in\N$ and $Z\subset\R^\ell$
be as in Proposition~\ref{prop:Jia1}. Suppose that
${\bf F}:\overline\Omega\times\prod^m_{k=0}\mathbb{R}^{N\times M_0(k)}\times Z\to \R$
   is $C^{k-m+2}$ and that all derivatives
   ${\bf F}^i_\alpha(x,\xi;\lambda)=D_{\xi^i_\alpha}{\bf F}(x,\xi;\lambda)$
 are also $C^{k-m+2}$ with respect to $(x,\xi,\lambda)$.
If for some $(\lambda, \vec{u})\in Z\times C^k(\overline\Omega,\mathbb{R}^N)$
there exists a real $c>0$ such that for all $x\in\overline{\Omega}$
and  $\eta=(\eta^{i}_{\alpha})\in\R^{N\times M_0(m)}$,
  \begin{eqnarray}\label{e:Jia5}
\sum^N_{i,j=1}\sum_{|\alpha|=|\beta|=m}{\bf F}^{ij}_{\alpha\beta}(x,\vec{u}(x),D\vec{u}(x),\cdots,D^m\vec{u}(x);
\lambda)\eta^i_\alpha\eta^j_\beta\ge c\sum^N_{i=1}\sum_{|\alpha|= m}(\eta^i_\alpha)^2,
\end{eqnarray}
then
\begin{enumerate}
\item[\rm (i)] either $\sigma(B_\lambda(\vec{u})|_{X_{k,p}})$ or $\sigma(B_\lambda(\vec{u})|_{X_{k,p}})\setminus\{0\}$ is bounded away from the imaginary axis;
\item[\rm (ii)] $H^0_\lambda:={\rm Ker}(B_\lambda(\vec{u}))\subset X_{k,p}$, and the negative definite space
 $H^-_{\lambda}$ of $B_\lambda(\vec{u})$ is of finite dimension (and hence is contained in $X$ by  \cite[Lemma~B.13]{Lu8}).
\end{enumerate}
\end{proposition}

\begin{proof}
\noindent{\bf Step 1} ({\it Prove the second claim}).
Clearly, all $A^{ij}_{\alpha\beta}(x):={\bf F}^{ij}_{\alpha\beta}(x,\vec{u}(x),D\vec{u}(x),\cdots,D^m\vec{u}(x);\lambda)$ belong to $C^{k-m}(\overline{\Omega})$.
Therefore by Proposition~\ref{prop:BifE.9.1} we deduce that the solution spaces of
\begin{eqnarray*}
B_\lambda(\vec{u})\vec{v}=\tau\vec{v},\quad(\tau,\vec{v})\in (-\infty, 0]\times H
\end{eqnarray*}
 are contained in $X_{k,p}$. Let $P_\lambda(\vec{u})$ and $Q_\lambda(\vec{u})$ be defined as in (\ref{e:6.6}) and (\ref{e:6.7}), respectively. (\ref{e:Jia5}) implies that $P_\lambda(\vec{u})$ is positive definite.
Moreover, $Q_\lambda(\vec{u})$ is compact.
By \cite[Lemma~B.12]{Lu8} we deduce that each real $\tau<\inf\{(B_\lambda(\vec{u})\vec{v},\vec{v})_H\,|\,\|\vec{v}\|_H=1\}$
  is either a regular value of $B_\lambda(\vec{u})$ or an isolated point of
$\sigma(B_\lambda(\vec{u}))$, which is also an eigenvalue of finite multiplicity.
Hence (ii) follows.

\vspace{4pt}\noindent
\noindent{\bf Step 2} ({\it Prove the first claim}).
Consider the complexification of  $H$ and $X_{k,p}$,
$H^{\mathbb{C}}=H+iH$ and $X^{\mathbb{C}}_{k,p}=X_{k,p}+iX_{k,p}$
(see \cite[page 14]{DaKr}).
 The inner product $(\cdot,\cdot)_H$ in (\ref{e:6.2.2}) is extended
 into a Hermite inner product $\langle\cdot,\cdot\rangle_H$ on $H^{\mathbb{C}}$
 in natural ways. Let $\mathbb{B}^C_\lambda$ be the natural complex linear extension on $H^{\mathbb{C}}$ of  $\mathbb{B}_\lambda:=B_\lambda(\vec{u})$. Then
 $\mathbb{B}^C_\lambda|_{X^{\mathbb{C}}_{k,p}}$ is such an extension on
 $X^{\mathbb{C}}_{k,p}$ of $\mathbb{B}_\lambda|_{X^{\mathbb{C}}_{k,p}}$.
 Both are symmetric with respect to $\langle\cdot,\cdot\rangle_H$.

Let $\vec{g},\vec{h}\in X_{k,p}$. Since $\mathbb{B}^C_\lambda$ is a
self-adjoint operator on $H^{\mathbb{C}}$,
for every $\tau\in\mathbb{R}\setminus\{0\}$ there exist unique
$\vec{u},\vec{v}\in H$ such that
$\mathbb{B}^C_\lambda(\vec{u}+i\vec{v})-i\tau(\vec{u}+i\vec{v})=\vec{g}+i\vec{h}$,
which is equivalent to
\begin{eqnarray}\label{e:Jia6}
\left.\begin{array}{ll}
&\sum^N_{k=1}\sum_{|\alpha|\le m}\sum_{|\beta|\le m}(-1)^{|\alpha|}D^\alpha (A^{jk}_{\alpha\beta}D^\beta u^k)+ \tau\triangle^mv^j=g^j,\\
&-\tau\triangle^mu^j+\sum^N_{k=1}\sum_{|\alpha|\le m}\sum_{|\beta|\le m}(-1)^{|\alpha|}D^\alpha (A^{jk}_{\alpha\beta}D^\beta v^k)=h^j,\\
&j=1,\cdots,N,
\end{array}\right\}
\end{eqnarray}
where  $\triangle^mv^j=(-1)^m\sum_{|\alpha|=m}D^\alpha(D^\alpha v^j)$ and
$$
g^j=\sum_{|\alpha|\le m}(-1)^{|\alpha|}D^\alpha g^j_\alpha,\quad
h^j=\sum_{|\alpha|\le m}(-1)^{|\alpha|}D^\alpha h^j_\alpha.
$$
Define
\begin{eqnarray*}
&&\hat{A}^{jk}_{\alpha\beta}=A^{jk}_{\alpha\beta},\quad j,k=1,\cdots,N,\\
&&\hat{A}^{jk}_{\alpha\beta}=\tau\delta_{\alpha\beta}\delta_{j(k-N)},\quad j=1,\cdots,N,\;k=N+1,\cdots,2N,\\
&&\hat{A}^{jk}_{\alpha\beta}=-\tau\delta_{\alpha\beta}\delta_{(j-N)k},\quad j=N+1,\cdots,2N,\;k=1,\cdots,N,\\
&&\hat{A}^{jk}_{\alpha\beta}=A^{(j-N)(k-N)}_{\alpha\beta},\quad j,k=N+1,\cdots,2N,\\
&&\hat f^j_\alpha=g^j_\alpha,\;j=1,\cdots,N,\quad
\hat f^j_\alpha=h^{j-N}_\alpha,\;j=N+1,\cdots,2N,\\
&&\hat w^j_\alpha=u^j_\alpha,\;j=1,\cdots,N,\quad
\hat w^j_\alpha=v^{j-N}_\alpha,\;j=N+1,\cdots,2N.
\end{eqnarray*}
Then (\ref{e:Jia6}) is equivalent to:
\begin{eqnarray}\label{e:Jia6+}
\sum^{2N}_{k=1}\sum_{|\alpha|\le m}\sum_{|\beta|\le m}(-1)^{|\alpha|}D^\alpha (\hat{A}^{jk}_{\alpha\beta}D^\beta \hat{w}^k)= \hat{f}^j,\quad
j=1,\cdots,2N.
\end{eqnarray}
 Note that (\ref{e:Jia5}) implies \allowdisplaybreaks
\begin{eqnarray*}
&&\sum^{2N}_{j,k=1}\sum_{|\alpha|=m=|\beta|}\hat{A}^{jk}_{\alpha\beta}\xi^j_\alpha\xi^k_\beta\\
&=&\sum^{N}_{j,k=1}\sum_{|\alpha|=m=|\beta|}{A}^{jk}_{\alpha\beta}\xi^j_\alpha\xi^k_\beta
+\sum^{N}_{j,k=1}\sum_{|\alpha|=m=|\beta|}{A}^{jk}_{\alpha\beta}\xi^{j+N}_\alpha\xi^{k+N}_\beta\\
&\ge& c\sum^N_{j=1}\sum_{|\alpha|=m}|\xi^j_\alpha|^2+ c\sum^N_{j=1}\sum_{|\alpha|=m}|\xi^{N+j}_\alpha|^2.
\end{eqnarray*}
Applying Proposition~\ref{prop:BifE.9.1} to (\ref{e:Jia6+}) we get
$\vec{\hat{w}}\in W^{k,p}(\Omega,\mathbb{R}^{2N})$, and thus
$\vec{u},\vec{v}\in X_{k,p}$.
Hence $i\tau$ is a regular value of $\mathbb{B}^C_\lambda|_{X^{\mathbb{C}}_{k,p}}$.

If $0$ is a regular value of $B_\lambda(\vec{u})$,
 $\mathbb{B}^C_\lambda:H^{\mathbb{C}}\to H^{\mathbb{C}}$ is an isomorphism.
 Thus we may take $\tau=0$ in the arguments above. This shows that
 $0$ is also a regular value of $\mathbb{B}^C_\lambda|_{X^{\mathbb{C}}_{k,p}}$.
 Hence the  spectrum of $\mathbb{B}^C_\lambda|_{X^{\mathbb{C}}_{k,p}}$
 is bounded away from the imaginary axis.

If $0\in\sigma(B_\lambda(\vec{u}))$,
then
$N:={\rm Ker}(\mathbb{B}^C_{\lambda^\ast})
={\rm Ker}(\mathbb{B}^C_{\lambda^\ast}|_{X^{\mathbb{C}}_{k,p}})$
 is of finite dimension by Step 1.   Denote by $Y$ the intersection
of $X_{k,p}^{\mathbb{C}}$ and the orthogonal complement of $N$ in $(H^{\mathbb{C}},(\cdot,\cdot)_H)$.
It is an invariant subspace of $\mathbb{B}^C_{\lambda^\ast}|_{X^{\mathbb{C}}_{k,p}}$,
and there exists a direct sum decomposition of Banach spaces,
$X_{k,p}^{\mathbb{C}}=N\oplus Y$.
 The above arguments imply that  the spectrum of the restriction of $\mathbb{B}^C_{\lambda^\ast}|_{X^{\mathbb{C}}_{k,p}}$ to $Y$
  is bounded away from the imaginary axis. Hence
   $\sigma\left(\mathbb{B}^C_{\lambda^\ast}|_{X^{\mathbb{C}}_{k,p}}\right)\setminus\{0\}$
   has a positive distance to the imaginary axis.
\end{proof}

\begin{theorem}\label{th:Jia1}
Let $k> m+\frac{n}{p}$ and $\Omega\subset\R^n$ be as in Proposition~\ref{prop:BifE.9.1}.
Let $\ell\in\N$ and ${\bf D}^\ell=\{\lambda\in\R^\ell\,|\,|\lambda|\le 1\}$. Suppose
   that ${\bf F}:\overline\Omega\times\prod^m_{k=0}\mathbb{R}^{N\times M_0(k)}\times {\bf D}^\ell\to \R$ is  $C^{k-m+6}$, and
\begin{enumerate}
\item[\rm (I)] all derivatives ${\bf F}^i_\alpha(x,\xi;\lambda)=D_{\xi^i_\alpha}{\bf F}(x,\xi;\lambda)$
 are also $C^{k-m+6}$ with respect to $(x,\xi,\lambda)$;
\item[\rm (II)] for each $\lambda\in {\bf D}^\ell$ there exists some $c_\lambda>0$
such that for all $x\in\overline{\Omega}$
and  $\eta=(\eta^{i}_{\alpha})\in\R^{N\times M_0(m)}$,
  \begin{eqnarray}\label{e:Jia7}
\sum^N_{i,j=1}\sum_{|\alpha|=|\beta|=m}{\bf F}^{ij}_{\alpha\beta}(x, 0;\lambda)\eta^i_\alpha\eta^j_\beta\ge
c_\lambda\sum^N_{i=1}\sum_{|\alpha|= m}(\eta^i_\alpha)^2;
\end{eqnarray}
 \item[\rm (III)]  for each $\lambda\in {\bf D}^\ell$, $0\in X_{k,p}$
 is a critical point of the functional $\mathcal{F}_\lambda$ as defined by
 (\ref{e:Jia1}).
\end{enumerate}
 Then $\mathcal{F}_\lambda$ has the nullity $\nu_{\lambda}>0$ (in the sense of \cite[Appendix~B]{Lu8})
 at $0\in X_{k,p}$ provided that $(\lambda, 0)\in{\bf D}^\ell\times X_{k,p}$ is a bifurcation point of
  \begin{eqnarray}\label{e:Jia8}
d_{\vec{u}}\mathcal{F}(\lambda,\vec{u})=0,\quad(\lambda,\vec{u})\in{\bf D}^\ell\times X_{k,p}.
 \end{eqnarray}
 Conversely, let $m_\lambda$ and $n_\lambda$ denote the Morse index and
 the nullity of the quadratic form
 $(B_\lambda(0)\vec{u},\vec{u})_H$ on $H$,
 and for $\ell=1$ suppose that $m_\lambda$
 %take values $m_{0}$ and $m_{0}+ n_{0}$
% as $\lambda\in{\bf D}^1=[-1,1]$ varies in both sides of $0$ and is close to $0$.
  takes, respectively, values $\mu_{0}$ and $\mu_{0}+\nu_{0}$
 as $\lambda\in{\bf D}^1=[-1,1]$ varies in
 two deleted half neighborhoods  of $0$.
 Then $(0, 0)\in{\bf D}^1\times X_{k,p}$ is a bifurcation point of (\ref{e:Jia8})
  and  one of the following alternatives occurs:
\begin{enumerate}
\item[\rm (i)] $(0, 0)\in{\bf D}^1\times X_{k,p}$ is not an isolated solution of (\ref{e:Jia8}) in  $\{0\}\times X_{k,p}$;

\item[\rm (ii)]  for every $\lambda\in{\bf D}^1$ near $0$ there is a non-zero solution $\vec{u}_\lambda$ of (\ref{e:Jia8}) converging to $0$ as $\lambda\to 0$;

\item[\rm (iii)] there is an one-sided  neighborhood $\Lambda$ of $0$ in ${\bf D}^1$
such that for any $\lambda\in\Lambda\setminus\{0\}$, (\ref{e:Jia8}) has at least two nontrivial solutions converging to $0$ as $\lambda\to 0$.
\end{enumerate}
  \end{theorem}

\begin{proof}
\noindent{\bf Step 1} ({\it Prove the first claim}).
Let $\mathcal{F}$, $A$ and $B$ be defined as in (\ref{e:Jia1}), (\ref{e:Jia2}) and (\ref{e:Jia4}) with $Z={\bf D}^\ell$, respectively.
By Proposition~\ref{prop:Jia1}, both $\mathcal{F}:{\bf D}^\ell\times X_{k,p}\to\R$
and ${\bf D}^\ell\times X_{k,p}\ni (\lambda, \vec{u})\to A(\lambda,\vec{u})\in X_{k,p}$ are $C^3$.
These and (\ref{e:Jia3}) show that the conditions (1)-(3) in
\cite[Theorem~6.1]{Lu8} and (sd') in \cite[Theorem~B.8]{Lu8} are satisfied.
By Proposition~\ref{prop:Jia2}, the first condition of (b) in \cite[Theorem~6.1]{Lu8} also holds. Hence the desired claim may follow from the second part of \cite[Theorem~6.1]{Lu8}.
(Actually, the claim can directly derived from the implicit function theorem since $A$ is $C^3$.)

\vspace{4pt}\noindent
\noindent{\bf Step 2} ({\it Prove the second claim}). By (ii) of Proposition~\ref{prop:Jia2},
$m_\lambda=\mu_\lambda$ and $n_\lambda=\nu_\lambda$, where $\mu_{\lambda}$ and $\nu_{\lambda}$ are the Morse index and the nullity of  $\mathcal{F}_{\lambda}$
at $0\in X_{k,p}$ (in the sense of \cite[Appendix~B]{Lu8}), respectively.
These and the proofs in Step 1 show that
the conditions (1)-(3) and (a)-(e) in \cite[Theorem~6.1]{Lu8} are satisfied.
Then the first part of \cite[Theorem~6.1]{Lu8} yields the expected conclusions.
\end{proof}

From \cite[Theorem~6.6]{Lu8} we immediately obtain:

\begin{theorem}\label{th:Jia2}
Under the assumptions (I)-(III) (with $\ell=1$) and ``converse part" in Theorem~\ref{th:Jia1},
let $G$ be a compact Lie group acting on $H$ orthogonally,
which induces a $C^1$ isometric action on $X_{k,p}$. Suppose that each $\mathcal{F}_\lambda$ is $G$-invariant and that
 $A_\lambda, B_\lambda$  are equivariant, and that
$H^0_0:={\rm Ker}({B}_{0}(0))$ only intersects at zero with the fixed point set $H^G$.
Then one of the following alternatives occurs:
\begin{enumerate}
\item[\rm (i)] $0\in X_{k,p}$ is not an isolated critical point of $\mathcal{F}_{0}$;
\item[\rm (ii)] there exist left and right  neighborhoods $\Lambda^-$ and $\Lambda^+$ of $0$ in ${\bf D}^1$
and integers $n^+, n^-\ge 0$, such that $n^++n^-\ge \ell(SH^0_{0})$
and for $\lambda\in\Lambda^-\setminus\{0\}$ (resp. $\lambda\in\Lambda^+\setminus\{0\}$),
$\mathcal{F}_\lambda$ has at least $n^-$ (resp. $n^+$) distinct critical
$G$-orbits different from $0$, which converge to
 $0$ as $\lambda\to 0$.
 \end{enumerate}
In particular,  $(0, 0)\in \Lambda\times X$ is a bifurcation point of (\ref{e:Jia8}).
Here $\ell(SH^0_{0})=\dim H^0_{0}$ (resp. $\frac{1}{2}\dim H^0_{0}$)
if  $G$ is equal to $\mathbb{Z}_2=\{{\rm id}, -{\rm id}\}$ (resp. $S^1$).
  \end{theorem}

When the function ${\bf F}$ in Theorem~\ref{th:Jia1} linearly depends
on the parameter $\lambda$, its smoothness can be weakened.

\begin{theorem}\label{th:BifE.9.2}
Let $k> m+\frac{n}{p}$ and $\Omega\subset\R^n$ be as in Proposition~\ref{prop:BifE.9.1}.
Let  the functions
   \begin{eqnarray*}
F:\overline\Omega\times\prod^m_{k=0}\mathbb{R}^{N\times M_0(k)}\to \R\quad\hbox{and}\quad
K:\overline\Omega\times\prod^{m-1}_{k=0}\mathbb{R}^{N\times M_0(k)}\to \R
\end{eqnarray*}
    be of class $C^{k-m+4}$,  and let $\vec{u}_0\in C^k(\overline\Omega,\mathbb{R}^N)\cap W^{2,m}_0(\Omega,\mathbb{R}^N)$
  be a common critical point of functionals on $X_{k,p}$  given by
 \begin{eqnarray}\label{e:BifE.8.1}
&&\mathscr{L}_1(\vec{u})=\int_\Omega F(x,\vec{u},D\vec{u},\cdots,D^m\vec{u})dx,\quad\vec{u}\in X_{k,p},\\
&&\mathscr{L}_2(\vec{u})=\int_\Omega K(x,\vec{u},D\vec{u},\cdots,D^{m-1}\vec{u})dx,\quad\vec{u}\in X_{k,p}.\label{e:BifE.8.2}
\end{eqnarray}
Suppose also:
\begin{enumerate}
\item[\rm (a)]  there exists some $c>0$ such that for all $x\in\overline{\Omega}$
and for all $\eta=(\eta^{i}_{\alpha})\in\R^{N\times M_0(m)}$,
  \begin{eqnarray}\label{e:BifE.8.3}
\sum^N_{i,j=1}\sum_{|\alpha|=|\beta|=m}F^{ij}_{\alpha\beta}(x, \vec{u}_0(x),\cdots, D^m \vec{u}_0(x))\eta^i_\alpha\eta^j_\beta\ge
c\sum^N_{i=1}\sum_{|\alpha|= m}(\eta^i_\alpha)^2;
\end{eqnarray}
\item[\rm (b)]  either
\begin{eqnarray*}
 \sum^N_{i,j=1}\sum_{|\alpha|,|\beta|\le m-1}
\int_\Omega K^{ij}_{\alpha\beta}(x, \vec{u}_0(x),\cdots, D^{m-1} \vec{u}_0(x))D^\beta v^j(x)D^\alpha v^i(x)dx\ge 0,\,\forall v\in H,
 \end{eqnarray*}
or
\begin{eqnarray*}
 \sum^N_{i,j=1}\sum_{|\alpha|,|\beta|\le m-1}
\int_\Omega K^{ij}_{\alpha\beta}(x, \vec{u}_0(x),\cdots, D^{m-1} \vec{u}_0(x))D^\beta v^j(x)D^\alpha v^i(x)dx\le 0,\,\forall v\in H;
 \end{eqnarray*}
\item[\rm (c)]   $\lambda^\ast\in\mathbb{R}$ is an isolated eigenvalue of
the linear eigenvalue problem  in $H$,
 \begin{eqnarray}\label{e:BifE.3*}
 &&\sum^N_{j=1}\sum_{|\alpha|,|\beta|\le m}
(-1)^{|\alpha|}D^\alpha\bigl[F^{ij}_{\alpha\beta}(x, \vec{u}_0(x),\cdots, D^m \vec{u}_0(x))D^\beta v^j\bigr]
\nonumber\\
&&=\lambda\sum^N_{j=1}\sum_{|\alpha|,|\beta|\le m-1}
(-1)^{|\alpha|}D^\alpha\bigl[K^{ij}_{\alpha\beta}(x, \vec{u}_0(x),\cdots, D^{m-1}\vec{u}_0(x))D^\beta v^j\bigr]\nonumber\\
&&\hspace{30mm}i=1,\cdots,N.
 \end{eqnarray}
\end{enumerate}
 Then $(\lambda^\ast, \vec{u}_0)\in\mathbb{R}\times X_{k,p}$ is a bifurcation point of
 \begin{eqnarray}\label{e:BifE.8.4}
d\mathscr{L}_1(\vec{u})=\lambda d\mathscr{L}_2(\vec{u}),\quad(\lambda,\vec{u})\in\mathbb{R}\times X_{k,p},
 \end{eqnarray}
  and  one of the following alternatives occurs:
\begin{enumerate}
\item[{\rm (i)}] $(\lambda^\ast, \vec{u}_0)$ is not an isolated solution of (\ref{e:BifE.8.4}) in
 $\{\lambda^\ast\}\times X_{k,p}$;

\item[\rm (ii)]  for every $\lambda\in\mathbb{R}$ near $\lambda^\ast$ there is a nontrivial solution $\vec{u}_\lambda$ of (\ref{e:BifE.8.4}) converging to $\vec{u}_0$ as $\lambda\to\lambda^\ast$;

\item[\rm (iii)] there is an one-sided  neighborhood $\Lambda$ of $\lambda^\ast$ such that
for any $\lambda\in\Lambda\setminus\{\lambda^\ast\}$,
(\ref{e:BifE.8.4}) has at least two nontrivial solutions converging to
$\vec{u}_0$ as $\lambda\to\lambda^\ast$.
\end{enumerate}
\end{theorem}
\begin{proof}
We shall use \cite[Corollary~6.4]{Lu8} to prove this theorem.
We need to check that \cite[Hypothesis~6.3]{Lu8} is satisfied with
the Banach space $X=X_{k,p}$, the Hilbert space $H:=W^{m,2}_0(\Omega,\mathbb{R}^N)$
and functionals $\mathscr{L}=\mathscr{L}_1$ and $\widehat{\mathscr{L}}=\mathscr{L}_2$.

Since the functions $F$ and $K$ are $C^{k-m+4}$, and
  $X_{k,p}\hookrightarrow C^m(\overline{\Omega}, \mathbb{R}^N)$ is continuous,
 as proved below \cite[Theorem~4.21]{Lu7} (see also the proof of Proposition~\ref{prop:Jia1})
we can use Corollary~\ref{cor:babySmoothness} to prove: \\
   {\bf I)} The functionals
  $\mathscr{L}_1$ and $\mathscr{L}_2$ are $C^{4}$.\\
   {\bf II)} The operators $A_1, A_2: X_{k,p}\to X_{k,p}$ given by
 \begin{eqnarray}\label{e:BifE.8.5}
 &&(A_1(\vec{u}))^i=\triangle^{-m}\sum_{|\alpha|\le m}(-1)^{m+|\alpha|}D^\alpha(F^i_\alpha(\cdot,
\vec{u}(\cdot),\cdots, D^m \vec{u}(\cdot))),\\
&&(A_2(\vec{u}))^i=\triangle^{-m}\sum_{|\alpha|\le m-1}(-1)^{m+|\alpha|}D^\alpha(K^i_\alpha(\cdot,
\vec{u}(\cdot),\cdots, D^{m-1} \vec{u}(\cdot)))\label{e:BifE.8.6}
 \end{eqnarray}
 for $i=1,\cdots,N$,
are of class $C^{3}$,  and  maps $B_1, B_2: X_{k,p}\to\mathscr{L}_s(X_{k,p})$ given by \allowdisplaybreaks
\begin{eqnarray}\label{e:BifE.8.7}
&&(B_1(\vec{u})\vec{v})^i\notag\\&&= \triangle^{-m}\sum^N_{j=1}\sum_{\scriptsize\begin{array}{ll}
   &|\alpha|\le m,\\
   &|\beta|\le m
   \end{array}}(-1)^{m+|\alpha|}
D^\alpha(F^{ij}_{\alpha\beta}(\cdot, \vec{u}(\cdot),\cdots, D^m \vec{u}(\cdot))D^\beta v^j),\\
&&(B_2(\vec{u})\vec{v})^i\notag\\&&= \triangle^{-m}\sum^N_{j=1}\sum_{\scriptsize\begin{array}{ll}
   &|\alpha|\le m-1,\\
   &|\beta|\le m-1
   \end{array}}(-1)^{m+|\alpha|}
D^\alpha(K^{ij}_{\alpha\beta}(\cdot, \vec{u}(\cdot),\cdots, D^{m-1} \vec{u}(\cdot))D^\beta v^j)\nonumber\\
\label{e:BifE.8.8}
\end{eqnarray}
 for $i=1,\cdots,N$, are of class $C^{2}$.

(\ref{e:BifE.8.7}) and (\ref{e:BifE.8.8}) show that
$B_1(\vec{u})$ and $B_2(\vec{u})$ can be extended into
into operators in $\mathscr{L}(H)$, also denoted by
$B_1(\vec{u})$ and $B_2(\vec{u})$.  (It was also proved  in \cite[Claim~4.18]{Lu7}
that $B_1$ and $B_2$ as maps from $X_{k,p}$ to $\mathscr{L}_s(H)$
are uniformly continuous on any bounded subsets of $X_{k,p}$
though we do not need this fact actually.)
Note that the linear eigenvalue problem (\ref{e:BifE.3*}) with $v\in H$ is equivalent to
\begin{eqnarray}\label{e:BifE.8.11}
B_1(\vec{u}_0)\vec{v}=\lambda B_2(\vec{u}_0)\vec{v},\quad(\lambda,\vec{v})\in\mathbb{R}\times H.
\end{eqnarray}
By the condition {\bf (c)}, $\lambda^\ast$ is an isolated eigenvalue of (\ref{e:BifE.8.11}).
Since $X_{k,p}\subset C^m(\overline{\Omega}, \mathbb{R}^N)$ it is easily checked that
 \begin{eqnarray*}
 &&d\mathscr{L}_1(\vec{u})[\vec{v}]=(A_1(\vec{u}), \vec{v})_H,\quad dA_1(\vec{u})[\vec{v}]=B_1(\vec{u})\vec{v},\quad\forall
 \vec{u},\vec{v}\in X_{k,p},\\
 &&d\mathscr{L}_2(\vec{u})[\vec{v}]=(A_2(\vec{u}), \vec{v})_H,\quad dA_2(\vec{u})[\vec{v}]=B_2(\vec{u})\vec{v},\quad\forall
 \vec{u},\vec{v}\in X_{k,p}.
 \end{eqnarray*}
These show that $\mathscr{L}_1, \mathscr{L}_2:X_{k,p}\to\mathbb{R}$
 satisfy \cite[Hypothesis~6.3]{Lu8}
 for some ball $B(\vec{u},r)\subset X_{k,p}$ centred at any given $\vec{u}\in X_{k,p}$ except properties (i)-(ii) in \cite[Hypothesis~6.3]{Lu8}.

Since $K^{ij}_{\alpha\beta}=0$ for any $|\alpha|=|\beta|$ and $i,j=1,\cdots,N$, each
$$
{\bf F}^{ij}_{\alpha\beta}(\cdot, \vec{u}_0(\cdot),\cdots, D^m \vec{u}_0(\cdot);\lambda)$$ $$
:=F^{ij}_{\alpha\beta}(\cdot, \vec{u}_0(\cdot),\cdots, D^m \vec{u}_0(\cdot))
-\lambda K^{ij}_{\alpha\beta}(\cdot, \vec{u}_0(\cdot),\cdots, D^{m-1} \vec{u}_0(\cdot))
$$
is equal to $F^{ij}_{\alpha\beta}(\cdot, \vec{u}_0(\cdot),\cdots, D^m \vec{u}_0(\cdot))$.
By (\ref{e:BifE.8.3}) we get
 \begin{eqnarray*}
\sum^N_{i,j=1}\sum_{|\alpha|=|\beta|=m}{\bf F}^{ij}_{\alpha\beta}(x, \vec{u}_0(x),\cdots, D^m \vec{u}_0(x);\lambda)\eta^i_\alpha\eta^j_\beta\ge
c\sum^N_{i=1}\sum_{|\alpha|= m}(\eta^i_\alpha)^2
\end{eqnarray*}
for all $x\in\overline{\Omega}$ and for all
$\eta=(\eta^{i}_{\alpha})\in\R^{N\times M_0(m)}$.
From Proposition~\ref{prop:Jia2} we deduce that the properties (i)-(ii) in \cite[Hypothesis~6.3]{Lu8} are satisfied.

For each $\vec{u}\in X_{k,p}$, we may write $B_1(\vec{u})=P_1(\vec{u})+ Q_1(\vec{u})$, where $P_1(\vec{u})$ and $Q_1(\vec{u})$ be defined as in
(\ref{e:6.6}) and (\ref{e:6.7}), respectively.
Then  $Q_1(\vec{u})$ and $B_2(\vec{u})$ as operators in $\mathscr{L}_s(H)$ are compact.
By (\ref{e:BifE.8.3}), $P_1(\vec{u}_0)\in \mathscr{L}_s(H)$ is positive definite.
The assumption {\bf (b)} shows that $B_2(\vec{u}_0)\in\mathscr{L}_s(H)$
is either semi-positive or semi-negative. Hence the conditions (a)-(b) in \cite[Corollary~6.4]{Lu8} hold.
\end{proof}

Similarly, by Corollary~6.5 in \cite{Lu8} %\ref{cor:BBH.3}
we may obtain

\begin{theorem}\label{th:BifE.9.3}
The conclusions  of  Theorem~\ref{th:BifE.9.2} also hold true if the conditions
{\bf (a)}--{\bf (b)} in Theorem~\ref{th:BifE.9.2}
are replaced by the following
\begin{enumerate}
\item[\rm (a')]  there exists some $c>0$ such that for all $x\in\overline{\Omega}$
and for all $\eta=(\eta^{i}_{\alpha})\in\R^{N\times M_0(m)}$,
  \begin{eqnarray}\label{e:BifE.8.3*}
\sum^N_{i,j=1}\sum_{|\alpha|\le m, |\beta|\le m}F^{ij}_{\alpha\beta}(x, \vec{u}_0(x),\cdots, D^m \vec{u}_0(x))\eta^i_\alpha\eta^j_\beta\ge
c\sum^N_{i=1}\sum_{|\alpha|\le m}(\eta^i_\alpha)^2.
\end{eqnarray}
\end{enumerate}
\end{theorem}

By Corollary~6.7 in \cite{Lu8} we immediately obtain the following two results.
\begin{theorem}\label{th:BifE.9.5}
Let the assumptions of one of Theorems~\ref{th:BifE.9.2} and \ref{th:BifE.9.3}
with $\vec{u}_0=0$ be satisfied.
Suppose also that both ${F}(x,\xi)$ and ${K}(x,\xi)$ are even with respect to $\xi$,
and that  $E_{\lambda^\ast}$ denotes the solution space of (\ref{e:BifE.3*})  with $\lambda=\lambda^\ast$. Then  one of the following alternatives occurs:
 \begin{enumerate}
\item[\rm (i)] $(\lambda^\ast, {0})$ is not an isolated solution of (\ref{e:BifE.8.4}) in
 $\{\lambda^\ast\}\times X_{k,p}$.
 \item[\rm (ii)] There exist left and right  neighborhoods $\Lambda^-$ and $\Lambda^+$ of $\lambda_\ast$ in $\mathbb{R}$
and integers $n^+, n^-\ge 0$, such that $n^++n^-\ge\dim E_{\lambda^\ast}$
and for $\lambda\in\Lambda^-\setminus\{\lambda^\ast\}$ (resp. $\lambda\in\Lambda^+\setminus\{\lambda^\ast\}$),
(\ref{e:BifE.8.4}) has at least $n^-$ (resp. $n^+$) distinct pairs of solutions of form $\{\vec{u},-\vec{u}\}$
 different from ${0}$, which converge to
 ${0}$ as $\lambda\to\lambda^\ast$.
\end{enumerate}
\end{theorem}

\begin{theorem}\label{th:BifE.9.6}
Let the assumptions of one of Theorems~\ref{th:BifE.9.2} and \ref{th:BifE.9.3}
with $\vec{u}_0=0$ be satisfied, and let $E_{\lambda^\ast}$ be  the solution space  of  (\ref{e:BifE.8.11})
 with $\lambda=\lambda^\ast$. Then we have:
\begin{enumerate}
\item[\rm (I)]  If $\Omega$ is symmetric with respect to the origin,
 both $\mathscr{L}_1$ and $\mathscr{L}_2$ are invariant for the $\mathbb{Z}_2$-action
 in (\ref{e:Z2-action.2}), and (\ref{e:BifE.3*}) with $\lambda=\lambda^\ast$ has no nontrivial solutions
 $\vec{u}\in H=W^{m,2}_0(\Omega,\mathbb{R}^N)$ satisfying $\vec{u}(-x)=\vec{u}(x)\;\forall x\in\Omega$,
 then the conclusions in  Theorem~\ref{th:BifE.9.5},
 after ``pairs of solutions of form $\{\vec{u},-\vec{u}\}$" being changed into
 ``pairs of solutions of form $\{\vec{u}(\cdot), \vec{u}(-\cdot)\}$", still holds.
\item[\rm (II)] Let $\mathbb{S}^1$ act on $\mathbb{R}^n$ by the orthogonal representation, let $\Omega$
be symmetric under the action,  both $\mathscr{L}_1$ and $\mathscr{L}_2$  be invariant for the $\mathbb{S}^1$
action on $H=W^{m,2}_0(\Omega,\mathbb{R}^N)$  in (\ref{e:S^1-action}).
If the fixed point set of the induced $S^1$-action in $E_{\lambda^\ast}$ is $\{0\}$,
 then   the conclusions of  Theorem~\ref{th:BifE.9.5},
 after ``$n^++n^-\ge\dim E_{\lambda^\ast}$" and ``pairs of solutions of form $\{\vec{u},-\vec{u}\}$" being changed into ``critical $\mathbb{S}^1$-orbits" and ``$n^++n^-\ge\frac{1}{2}\dim E_{\lambda^\ast}$", hold.
 \end{enumerate}
\end{theorem}

\begin{remark}\label{rem:BifE.9.6+}
{\rm {\bf (i)}. For Theorem~\ref{th:BifE.9.6}(I), if both $\mathscr{L}_1$ and $\mathscr{L}_2$ are invariant for the $\mathbb{Z}_2$-action
 in (\ref{e:Z2-action.2}), and both ${F}(x,\xi)$ and ${K}(x,\xi)$ are
 even with respect to $\xi$,
then $\mathscr{L}_1$ and $\mathscr{L}_2$ are invariant for the product
$\mathbb{Z}_2\times \mathbb{Z}_2$-action given by
(\ref{e:Z2-action.1}) and (\ref{e:Z2-action.2}),
\begin{eqnarray*}
 ([0],[0])\cdot\vec{u}=\vec{u},\quad ([1],[0])\cdot\vec{u}=-\vec{u},\quad ([0],[1])\cdot\vec{u}=\vec{u}(-\cdot),\quad
 ([1],[1])\cdot\vec{u}=-\vec{u}(-\cdot)
\end{eqnarray*}
for all $\vec{u}\in W^{m,2}(\Omega,\mathbb{R}^N)$. Since the $\mathbb{Z}_2$-action given by
(\ref{e:Z2-action.1}) has no nonzero fixed point, so is this product action.
By \cite[Remark~5.14]{Lu8} and \cite[Corollary~6.7]{Lu8}
 the conclusions in  Theorem~\ref{th:BifE.9.5},
 after ``pairs of solutions of form $\{\vec{u},-\vec{u}\}$" being changed into
 `` critical $(\mathbb{Z}_2\times \mathbb{Z}_2)$-orbits", still holds.

\noindent{\bf (ii)}. For Theorem~\ref{th:BifE.9.6}(II), if both $\mathscr{L}_1$ and $\mathscr{L}_2$ are invariant
for the $\mathbb{S}^1$ action on $H=W^{m,2}_0(\Omega,\mathbb{R}^N)$  in (\ref{e:S^1-action}),
and both ${F}(x,\xi)$ and ${K}(x,\xi)$ are even with respect to $\xi$,
then $\mathscr{L}_1$ and $\mathscr{L}_2$ are invariant for the product
$(\mathbb{S}^1\times \mathbb{Z}_2)$-action given by
(\ref{e:S^1-action}) and (\ref{e:Z2-action.1}),
\begin{eqnarray*}
 ([t],[0])\cdot\vec{u}(x)=\vec{u}([t]\cdot x),\quad
 ([t],[1])\cdot\vec{u}=-\vec{u}([t]\cdot x), \quad\forall [t]\in \mathbb{S}^1,\quad x\in\overline{\Omega}.
\end{eqnarray*}
This action has no nonzero fixed point yet.
By \cite[Remark~5.14]{Lu8} and \cite[Corollary~6.7]{Lu8}
 the conclusions of  Theorem~\ref{th:BifE.9.5},
 after ``$n^++n^-\ge\dim E_{\lambda^\ast}$" and ``pairs of solutions of form $\{\vec{u},-\vec{u}\}$" being changed
 into ``critical $(\mathbb{S}^1\times \mathbb{Z}_2)$-orbits" and ``$n^++n^-\ge\frac{1}{2}\dim E_{\lambda^\ast}$", hold.}
\end{remark}

\begin{corollary}\label{cor:BifE.9.7}
 For positive integers  $n>1$ and $k,N$, let real $p\ge 2$ be such that $k>1+\frac{n}{p}$,
 and let $\Omega\subset\R^n$  be a bounded  domain with boundary of class $C^{k-1,1}$.
   Suppose that {\color{black}$C^{k+3}$} functions
$G:\overline{\Omega}\times\mathbb{R}^N\to\mathbb{R}$ and
$A^{ij}_{\mu\nu}=A^{ji}_{\nu\mu}:\overline{\Omega}\times\mathbb{R}^N\to\mathbb{R}$, $i,j=1,\cdots,N$ and $\mu,\nu=1,\cdots,n$,  satisfy  the following conditions:
 \begin{enumerate}
\item[\rm A.1)]    $\nabla_{\xi}G(x,0)=0$;
 \item[\rm A.2)]  there exists $c>0$ such that
for $a.e.\;x\in\Omega$ and all $(\eta_\mu^i)\in\mathbb{R}^{N\times n}$,
\begin{eqnarray*}
\sum^N_{i,j=1}\sum^n_{\mu,\nu=1}A_{\mu\nu}^{ij}(x, 0)\eta^i_\mu\eta^j_\nu\ge
c\sum^N_{i=1}\sum^n_{\mu=1}(\eta^i_\mu)^2;
\end{eqnarray*}
    \item[\rm A.3)] $\lambda^\ast$ is an eigenvalue of
    \begin{eqnarray}\label{e:BifE.8.14.1}
&&\int_\Omega\sum^N_{i,j=1}\sum^n_{\mu,\nu=1}A_{\mu\nu}^{ij}(x,0)D_\mu u^i D_\nu v^j dx-\int_\Omega \sum^N_{i,j=1}
\frac{\partial^2 G}{\partial\xi_i\partial\xi_j}(x,0)u^iv^j dx\nonumber\\
&&=\lambda\int_\Omega \sum^N_{i=1}u^iv^idx,\quad\forall \vec{v}\in W^{1,2}_0(\Omega,\mathbb{R}^N).
\end{eqnarray}
   \end{enumerate}
Then with  the Banach subspace of  $W^{k,p}(\Omega,\mathbb{R}^N)$,  $X_{k,p}:=W^{k,p}(\Omega,\mathbb{R}^N)\cap W^{1,2}_0(\Omega,\mathbb{R}^N)$,
    $(\lambda^\ast, 0)\in\mathbb{R}\times X_{k,p}$ is a bifurcation point of the quasilinear eigenvalue problem
  \begin{eqnarray}\label{e:BifE.8.15}
&&\hspace{-.8cm}-\sum^N_{i=1}\sum^n_{\mu,\nu=1}D_l(A_{\mu\nu}^{ij}(x,\vec{u})D_\mu u^i)
+\frac{1}{2}\sum^N_{i,r=1}\sum^n_{\mu,\nu=1}\nabla_{\xi_j}A^{ir}_{\mu\nu}(x,\vec{u}) D_\mu u^iD_\nu u^r-
\nabla_{\xi_j}G(x,\vec{u})\nonumber\\
&&=\lambda u^j\hspace{5mm}\quad\hbox{in}\;\Omega,\quad j=1,\cdots,N,\\
&&\vec{u}=0\quad\hbox{on}\;\partial\Omega,\label{e:BifE.8.16}
\end{eqnarray}
 and  one of the following alternatives occurs:
\begin{enumerate}
\item[\rm (i)] $(\lambda^\ast, 0)$ is not an isolated solution of (\ref{e:BifE.8.15})--(\ref{e:BifE.8.16}) in
 $\{\lambda^\ast\}\times X_{k,p}$;

\item[\rm (ii)]  for every $\lambda\in\mathbb{R}$ near $\lambda^\ast$ there is a nontrivial solution $\vec{u}_\lambda$ of
(\ref{e:BifE.8.15})--(\ref{e:BifE.8.16}) converging to $0$ as $\lambda\to\lambda^\ast$;

\item[\rm (iii)] there is an one-sided  neighborhood $\Lambda$ of $\lambda^\ast$ such that for any $\lambda\in\Lambda\setminus\{\lambda^\ast\}$,
(\ref{e:BifE.8.15})--(\ref{e:BifE.8.16}) has at least two nontrivial
solutions converging to $0$ as $\lambda\to\lambda^\ast$.
\end{enumerate}
Moreover, if both $A^{ij}_{\mu\nu}(x,\xi)$ and $G(x,\xi)$ are even with respect to $\xi$,
and  $E_{\lambda^\ast}$ is the solution space of (\ref{e:BifE.8.14.1})  with $\lambda=\lambda^\ast$,  then  one of the following alternatives holds:
 \begin{enumerate}
\item[\rm (iv)] $(\lambda^\ast, {0})$ is not an isolated solution of (\ref{e:BifE.8.15})--(\ref{e:BifE.8.16}) in
 $\{\lambda^\ast\}\times X_{k,p}$;
 \item[\rm (v)] there exist left and right  neighborhoods $\Lambda^-$ and $\Lambda^+$ of $\lambda^\ast$ in $\mathbb{R}$
and integers $n^+, n^-\ge 0$, such that $n^++n^-\ge\dim E_{\lambda^\ast}$
and for $\lambda\in\Lambda^-\setminus\{\lambda^\ast\}$ (resp. $\lambda\in\Lambda^+\setminus\{\lambda^\ast\}$),
(\ref{e:BifE.8.15})--(\ref{e:BifE.8.16}) has at least $n^-$ (resp. $n^+$) distinct pairs of solutions of form
$\{\vec{u},-\vec{u}\}$  different from ${0}$, which converge to
 ${0}$ as $\lambda\to\lambda^\ast$.
\end{enumerate}
\end{corollary}

\begin{proof}
Note that  A.2) implies  that each eigenvalue of (\ref{e:BifE.8.14.1}) is
isolated (see \cite[\S6.5]{Mor}).
The claims above ``Moreover" can be obtained by
 applying Theorem~\ref{th:BifE.9.2} to the Hilbert space $H=W^{1,2}_0(\Omega,\mathbb{R}^N)$,  the Banach space
   $X_{k,p}:=W^{k,p}(\Omega,\mathbb{R}^N)\cap W^{1,2}_0(\Omega,\mathbb{R}^N)$ and $\vec{u}_0=0$,   and functions
   \begin{eqnarray*}
  && F(x,\xi_1,\cdots,\xi_N;\eta^i_\mu)=\frac{1}{2}\sum^N_{i,j=1}\sum^n_{\mu,\nu=1}
A_{\mu\nu}^{ij}(x,\xi_1,\cdots,\xi_N)\eta_\mu^i\eta_\nu^j-G(x,\xi_1,\cdots,\xi_N),\\
&&K(x,\xi_1,\cdots,\xi_N;\eta^i_\mu)=\frac{1}{2}\sum^N_{\mu=1}\xi_\mu^2.
\end{eqnarray*}
Others follow from Theorem~\ref{th:BifE.9.5} directly.
\end{proof}

Similarly, under suitable assumptions as in Theorem~\ref{th:BifE.9.6}
the corresponding conclusions hold.

\begin{remark}\label{rm:BifE.9.8}
{\rm {\bf (i)} In Corollary~\ref{cor:BifE.9.7}, suppose $p>n$. We can take $k=3$. Then $\partial\Omega$ is of class $C^{2,1}$,  functions
$G, A^{ij}_{\mu\nu}=A^{ji}_{\nu\mu}:\overline{\Omega}\times\mathbb{R}^N\to\mathbb{R}$, $i,j=1,\cdots,N$ and $\mu,\nu=1,\cdots,n$,
   are all $C^6$,  and solutions in $X_{3,p}$  are classical solutions
(since $\vec{u}\in X_{3,p}\subset C^2(\overline{\Omega},\mathbb{R}^N)\cap W^{1,2}_0(\Omega,\mathbb{R}^N)$
satisfies $\vec{u}|_{\partial\Omega}=0$ by \cite[Theorem~9.17]{Bre}).
Corresponding conclusions with Theorems~\ref{th:BifE.9.2},\ref{th:BifE.9.3},\ref{th:BifE.9.5},\ref{th:BifE.9.6}
also hold true.\\
{\bf (ii)} By contrast, when $N=1$, $\Omega\subset\mathbb{R}^n$ is a bounded open subset,
$G, A^{ij}_{\mu\nu}=A^{ji}_{\nu\mu}:\overline{\Omega}\times\mathbb{R}\to\mathbb{R}$
is measurable in $x$ for all $\xi\in \mathbb{R}$, and of class $C^1$ in $\xi$ for $a.e.\;x\in\Omega$, some authors obtained corresponding results in space $W^{1,2}_0(\Omega)\cap L^\infty(\Omega)$
under some additional growth conditions on $A^{ij}$ and $G$,
see \cite{Can,Can1} and references therein.}
\end{remark}

\begin{example}\label{ex:BifE.9.9}
{\rm  Let $p, k, m$ and $\Omega$ be as in Theorem~\ref{th:BifE.9.2}.
Assume
\begin{eqnarray*}
 &&F(x,\xi,\eta_1,\cdots,\eta_n)=(1+\sum^n_{i=1}|\eta_i+D_iu_0(x)|^2)^{1/2}-\mu\xi,\\
&&K(x,\xi,\eta_1,\cdots,\eta_n)=\frac{1}{2}(\xi+u_0(x))^2,
\end{eqnarray*}
where $\mu\in\mathbb{R}$ is a constant and
${u}_0\in C^{k-m+5}(\overline\Omega)$ satisfies the following equation
\begin{eqnarray}\label{e:BifE.8.17}
\sum^n_{i=1}D_i\left(\frac{D_iu_0}{\sqrt{1+|Du_0|^2}}\right)-\mu=\lambda^\ast u_0
\end{eqnarray}
for some constant $\lambda^\ast\in\mathbb{R}$. Consider  the bifurcation problem
\begin{eqnarray}\label{e:BifE.8.18}
\sum^n_{i=1}D_i\left(\frac{D_i(u+u_0)}{\sqrt{1+|D(u+u_0)|^2}}\right)-\mu=\lambda (u+u_0),\quad
u|_{\partial\Omega}=0,
\end{eqnarray}
where the left side is the Euler-Lagrange operator for the corresponding functional
$\mathscr{L}_1$ as in (\ref{e:BifE.8.1}), that is, the area functional if $u_0=0$. The corresponding  linearized problem at the trivial
solution $u=0$ is
\begin{eqnarray}\label{e:BifE.8.19}
\sum^n_{i,j=1}D_i\left(\frac{\delta_{ij}(1+|Du_0|^2)-
D_iu_0D_ju_0}{(1+|Du_0|^2)^{3/2}}D_ju\right)=\lambda u,\quad
u|_{\partial\Omega}=0,
\end{eqnarray}
which is a linear elliptic problem since it is easy to check that for all $\zeta\in\mathbb{R}^n$,
\begin{eqnarray*}
\sum^n_{i,j=1}\frac{\delta_{ij}(1+|Du_0|^2)-
D_iu_0D_ju_0}{(1+|Du_0|^2)^{3/2}}\zeta_i\zeta_j&\ge&
\frac{1}{(1+|Du_0|^2)^{3/2}}|\zeta|^2\\
&\ge& \frac{1}{(1+\max|Du_0|^2)^{3/2}}|\zeta|^2.
\end{eqnarray*}
Hence each eigenvalue of (\ref{e:BifE.8.19}) is isolated.
Suppose that $\lambda^\ast$ is an eigenvalue of (\ref{e:BifE.8.19}). Since
${u}_0\in C^{k-m+5}(\overline\Omega)$ implies that $F$ is of class $C^{k-m+4}$,
 we may obtain some bifurcation results for (\ref{e:BifE.8.18}) near $(\lambda^\ast,0)\in\mathbb{R}\times X_{k,p}$ with Theorems~\ref{th:BifE.9.2},\ref{th:BifE.9.3},\ref{th:BifE.9.5},\ref{th:BifE.9.6},
which are also classical solutions for $k> 2+\frac{n}{p}$.

When $n=2$ (i.e., $\Omega\subset\mathbb{R}^2$),  (\ref{e:BifE.8.18})
occurred as a mathematical model for many problems of  hydrodynamics and theory of spring
 membrane (cf. \cite[\S5]{Bor}). In particular, if $\mu=0$ and $u_0=0$,
(\ref{e:BifE.8.18}) becomes
\begin{eqnarray}\label{e:BifE.8.20}
-(1+|\nabla u|^2)^{-3/2}(\triangle u+ u^2_yu_{xx}-2u_xu_yu_{xy}+ u_x^2u_{yy})=
\lambda u,\quad
u|_{\partial\Omega}=0,
\end{eqnarray}
which  was studied in \cite{Bor, BorMa} with functional-topological
 properties of the Plateau operator. When  $\partial\Omega$ is of class $C^{2,1}$, $n=2$, $k=3$ and $p>2$
 our results  are supplement for \cite{Bor, BorMa}. }
\end{example}

\section{Bifurcations from deformations of domains}\label{sec:BifE.4}
\setcounter{equation}{0}

%\subsection*{Smale Morse index theorem}

We start with reviewing  the Smale's Morse index theorem needed in this section.
Let  $\Omega\subset\R^n$ be a bounded, connected and open subset with smooth boundary $\partial\Omega$,
$\overline{\Omega}=\Omega\cup(\partial\Omega)$ the closure of $\Omega$.
 By a {\it deformation} $\{\overline{\Omega}_t\}_{0\le t\le 1}$
of $\overline{\Omega}$ we mean a continuous curve of $C^{\infty}$ embedding
$\varphi_t:\overline{\Omega}\to\overline{\Omega}$ such that $\overline{\Omega}_t=
\varphi_t(\overline{\Omega})$, $\varphi_0=id_{\overline{\Omega}}$ and
$\partial\Omega_t=\varphi_t(\partial\Omega)$ for $0\le t\le 1$. (Also write ${\Omega}_t=
\varphi_t({\Omega})$.)
Call the deformation  {\it smooth} (resp. {\it contracting}) if
$\partial\Omega_t$ depends smoothly on $t$ in the sense that
$[0,1]\times\overline{\Omega}\to\overline{\Omega},\;(t,x)\mapsto \varphi_t(x)$
is $C^{\infty}$ (resp. $\overline{\Omega}_{t_2}\subset\overline{\Omega}_{t_1}$ (strictly) for all $t_1<t_2$).

Given integers $m$ and $N$,  a differential operator $L$ of divergence form
\begin{equation}\label{e:strong*}
(L\vec{u})(x)=\sum^N_{i,j=1}\sum_{|\alpha|,|\beta|\le m}
(-1)^{|\alpha|}D^\alpha\bigl[A^{ij}_{\alpha\beta}(x)D^\beta u^j\bigr]
\end{equation}
where $A_{\alpha\beta}^{ij}$ are real-valued $C^\infty$ functions on $\overline{\Omega}$
for $i,j=1,\cdots,N$, and $\alpha,\beta\in(\mathbb{N}\cup\{0\})^n$ of length $\le m$,
is called {\it strongly elliptic} on $\Omega$ if there exists $M>0$ such that for all $\lambda\in\mathbb{R}^n$ and $\xi\in\mathbb{R}^N$,
$$
\sum^N_{i,j=1}\sum_{|\alpha|=|\beta|=m}
(A_{\alpha\beta}^{ij}(x)+ A_{\alpha\beta}^{ji}(x))\lambda^{\alpha+\beta}\xi_i{\xi}_j\ge M|\lambda|^{2m}\sum^N_{j=1}|\xi_j|^{2}.
$$
Let $C^\infty_{m-1}(\overline{\Omega},\R^N)=\{\vec{u}\in C^{\infty}(\overline{\Omega},\R^N)\;|\; D^\alpha\vec{u}|_{\partial\Omega}=0\,\;\forall |\alpha|\le m-1\}$.
The operator $L$  is called {\it self-adjoint} if the bilinear form $B_L$ on $C^\infty_{m-1}(\overline{\Omega},\R^N)$ given by
$$
B_L(\vec{u},\vec{v})=\int_{\Omega}((L\vec{u})(x), \vec{v}(x))_{\mathbb{R}^N}dx=\sum^N_{i,j=1}\sum_{|\alpha|,|\beta|\le m}
\int_\Omega A^{ij}_{\alpha\beta}(x)D^\alpha v^i(x)D^\beta u^j(x) dx
$$
is symmetric. Clearly,  $B_L$ is symmetric if $A^{ij}_{\alpha\beta}(x)=A^{ji}_{\beta\alpha}(x)$ for all $x,i,j,\alpha,\beta$.

Since $C^\infty(\overline{\Omega},\R^N)\cap W^{m,2}_0(\Omega,\mathbb{R}^N)=C^\infty_{m-1}(\overline{\Omega},\R^N)$,
the form $B_L$ extends continuously to\linebreak
  $W^{m,2}_0(\Omega,\mathbb{R}^N)\times W^{m,2}_0(\Omega,\mathbb{R}^N)$,
which corresponds to an operator
$$
\tilde{L}:W^{m,2}_0(\Omega,\mathbb{R}^N)\to W^{-m,2}_0(\Omega,\mathbb{R}^N)=(W^{m,2}_0(\Omega,\mathbb{R}^N))^\ast\;\hbox{
(dual space)}
$$
via $B_L(\vec{u},\vec{v})=(\tilde{L}\vec{u},\vec{v})_{L^2}$ for all $\vec{u},\vec{v}\in W^{m,2}_0(\Omega,\mathbb{R}^N)$.
Clearly, $\tilde{L}\vec{u}=L\vec{u}\;\forall \vec{u}\in C^\infty_{m-1}(\overline{\Omega},\R^N)$.
When $L$ is strongly elliptic self-adjoint, $\tilde{L}$ is self-adjoint and
Fredholm (index zero).

From now on we assume that the operator $L$ is strongly elliptic self-adjoint. Then $B_L$
has the same nullity and Morse index on $C^\infty_{m-1}(\overline{\Omega},\R^N)$
and $W^{m,2}_0(\Omega,\mathbb{R}^N)$, denoted by $\nu(L)$ and $\mu(L)$, respectively.
It was proved that both $\nu(L)$ and $\mu(L)$ are finite.
For the deformation $\{\overline{\Omega}_t\}_{0\le t\le 1}$
of $\overline{\Omega}$ at the beginning of this section, $L$ restricts to a strongly elliptic self-adjoint operator
$L_t:C^\infty_{m-1}(\overline{\Omega}_t,\R^N)\to C^\infty(\overline{\Omega}_t,\R^N)$ for each $t$.
The corresponding  symmetric bilinear form on $W^{m,2}_0(\Omega_t,\mathbb{R}^N)$
is given by
\begin{eqnarray}\label{e:BifE-bilinear}
B_{L_t}(\vec{u},\vec{v})=\sum^N_{i,j=1}\sum_{|\alpha|,|\beta|\le m}
\int_{\Omega_t} A^{ij}_{\alpha\beta}(x)D^\alpha v^i(x)D^\beta u^j(x) dx.
\end{eqnarray}
Using the Banach space isomorphism
\begin{eqnarray}\label{e:BifE.10-2}
\varphi_t^\ast:W^{m,2}_0(\Omega,\mathbb{R}^N)\to
W^{m,2}_0(\Omega_t,\mathbb{R}^N),\;\vec{u}\mapsto \vec{u}\circ\varphi_t^{-1}
\end{eqnarray}
we get a continuous symmetric  bilinear form on $W^{m,2}_0(\Omega,\mathbb{R}^N)$  given by
$$
B_{L}^t(\vec{u},\vec{v}):=B_{L_t}(\vec{u}\circ\varphi_t^{-1},\vec{v}\circ\varphi_t^{-1}).
$$
Then $B_{L}^t$ and $B_{L_t}$ have the same nullity and Morse index.
According to Smale \cite{Sma1} (and  Uhlenbeck \cite[Theorem 3.5]{Uh} and Swanson \cite[Theorem 5.7]{Swa}), we have:

\begin{proposition}[\hbox{Smale Morse index theorem}]\label{prop:BifE.10}
Let  $\Omega\subset\R^n$ be a bounded, connected and open subset with smooth boundary $\partial\Omega$,
and let $\{\overline{\Omega}_t\}_{0\le t\le 1}$ be the smooth deformation
of $\overline{\Omega}$ as above.
For a strongly elliptic self-adjoint operator $L$ on $\Omega$ as in (\ref{e:strong*}),
suppose that one of the following two conditions is satisfied:
\begin{enumerate}
\item[\rm (i)] If $\vec{u}\in C^\infty(\overline{\Omega},\R^N)$ satisfies $L\vec{u}=0$ and vanishes on some
non-empty open set in $\Omega$, then $\vec{u}=0$.
\item[\rm (ii)] The deformation $\{\overline{\Omega}_t\}_{0\le t\le 1}$ is contracting
and $L\vec{u}=0$ has no nontrivial solutions in $C^\infty(\overline{\Omega},\R^N)$ with compact
support in any of the manifolds $\Omega_t$.
\end{enumerate}
 Then
$$
\mu(L)=\mu(L_0)=\mu(L_1)+\sum_{0< t\le 1}\nu(L_t).
$$
 Moreover, if $\Omega_1$ has sufficiently small volume, then $\mu(L_1)=0$.
\end{proposition}
 The time $t\in [0, 1]$ with $\nu(L_t)\ne 0$ is
called a {\it conjugate point}.

 When $m=N=1$,  Cox,  Jones and Marzuola \cite{CoChMa} weakened smoothness of $\Omega$, $\varphi_t$ and coefficients of $L$ recently,
Let $\Omega\subset\R^n$ be a bounded domain with Lipschitz boundary and $\{{\Omega}_t\}_{0\le t\le 1}$
 be a family of domains given by a family of Lipschitz diffeomorphisms $\varphi_t:\Omega\to\Omega_t$.
For each $t$, the Jacobian of $\varphi_t$, $D\varphi_t:\Omega\to\R^{n\times n}$,
belongs to $L^\infty(\Omega,\R^{n\times n})$.
The family $\{\varphi_t\}$ is said to be of class $C^k$ if
$t\mapsto\varphi_t$  is in $C^k([0,1], L^\infty(\Omega,\R^{n}))$ and
$t\mapsto D\varphi_t$  is in $C^k([0,1], L^\infty(\Omega,\R^{n\times n}))$.
Let $L$ be a strongly elliptic self-adjoint operator of the form
\begin{equation}\label{e:strong**}
Lu=-\sum^n_{i,j=1}D_i(a_{ij}D_ju)+ cu,
\end{equation}
where $a_{ij}, c\in L^\infty(\Omega)$ are real-valued functions with $a_{ij} = a_{ji}$.
The strong elipticity means that the matrixes $(a_{ij}(x))$ are uniformly positive definite. As above  $\nu(L)$ and $\mu(L)$  denote the nullity and Morse index of $B_L$ on $W^{1,2}_0(\Omega)=H^1_0(\Omega)$,  respectively.

\begin{proposition}[\hbox{\cite[Corollary~2.3]{CoChMa}}]\label{prop:BifE.10*}
 For a bounded domain $\Omega\subset\R^n$ with Lipschitz boundary let
 $\{{\Omega}_t\}_{0\le t\le 1}$
 be a family of domains given by a $C^1$ family of Lipschitz diffeomorphisms $\varphi_t:\Omega\to\Omega_t$.
 Let $L$ be  a strongly elliptic self-adjoint operator as in (\ref{e:strong**}).
 Suppose that $a_{ij}, c\in C^1(\overline{\Omega})$ and that each
 $\partial\Omega_t$ is of class $C^{1,1}$,
 $\cup_{0\le t\le 1}\Omega_t\subset\Omega$ and $\Omega_{t_1}\subset\Omega_{t_2}$
  for $t_1<t_2$.
 Then the number of conjugate times in $[0,1]$ is finite and
$$
\mu(L)=\mu(L_0)=\mu(L_1)+\sum_{0< t\le 1}\nu(L_t).
$$
 \end{proposition}

We always assume that $\{\overline{\Omega}_t\}_{0\le t\le 1}$
is a smooth deformation of $\overline{\Omega}$ without special statements.
(Such a deformation can always be obtained by the negative flow of a suitable Morse function on $\overline{\Omega}$.)

Let $F$  satisfy \textsf{Hypothesis} $\mathfrak{F}_{2,N,m,n}$.
It gives a family of functionals
\begin{eqnarray}\label{e:BifE.10-1}
\mathfrak{F}_t(\vec{u})=\int_{\Omega_t} F(x, \vec{u},\cdots, D^m\vec{u})dx,\quad\forall\vec{u}\in W^{m,2}_0(\Omega_t,\mathbb{R}^N),
\;0\le t\le 1.
\end{eqnarray}
The critical points of $\mathfrak{F}_t$ correspond to weak solutions of
\begin{eqnarray}\label{e:BifE.10}
\left.\begin{array}{ll}
&\sum_{|\alpha|\le m}(-1)^{|\alpha|}D^\alpha F^i_\alpha(x, \vec{u},\cdots, D^m\vec{u})=0\quad
\hbox{on}\;\Omega_{t},\quad i=1,\cdots,N,\\
&\hspace{20mm}D^i\vec{u}|_{\partial\Omega_{t}}=0,\quad i=0,1,\cdots,m-1.
\end{array}\right\}
\end{eqnarray}
Let $\vec{u}=0$ be a solution of (\ref{e:BifE.10}) with $t=0$.
We say $t^\ast\in [0,1]$ to be a {\it bifurcation point} for the
system (\ref{e:BifE.10}) if there exist  sequences $t_k\to t^\ast$
and $(\vec{u}_k)\subset W^{m,2}_0(\Omega_{t_k}, \mathbb{R}^N)$
such that each $u_k$ is a nontrivial weak solution of (\ref{e:BifE.10}) with $t=t_k$ and
$\|u_k\|_{m,2}\to 0$. Recently, such a problem was studied for semilinear elliptic Dirichlet problems on a ball
 in \cite{PoWa}. We shall here  generalize their result in different directions.
 To this goal,
define $\mathcal{F}:[0,1]\times W^{m,2}_0(\Omega,\mathbb{R}^N)\to\mathbb{R}$  by
\begin{eqnarray}\label{e:BifE.11}
\mathcal{F}(t,\vec{u})=\int_{\Omega_t} F(x, \vec{u}\circ\varphi_t^{-1},\cdots, D^m(\vec{u}\circ\varphi_t^{-1}))dx.
 \end{eqnarray}
Then $\mathcal{F}_t:=\mathcal{F}(t,\cdot)=\mathfrak{F}_t\circ\varphi_t^\ast$,
that is, $\mathcal{F}_t$ is the pull-back of $\mathfrak{F}_t$ via $\varphi_t^\ast$.
Clearly, $\vec{u}\in W^{m,2}_0(\Omega,\mathbb{R}^N)$ is a critical point of $\mathcal{F}_t$ if and only if $\varphi_t^\ast(\vec{u})$
is that of $\mathfrak{F}_t$ and both have the same Morse indexes and nullities. Note that $\vec{u}=0\in W^{m,2}_0(\Omega,\mathbb{R}^N)$ is
the critical point of each $\mathcal{F}_t$ (and so $\mathfrak{F}_t$).
Denote by $\mu_t$ and $\nu_t$ the common Morse index and nullity of $\mathfrak{F}_t$ and
$\mathcal{F}_t$ at zeros in the sense of \cite{Lu7}, that is,
those of the continuous symmetric  bilinear form on
$W^{m,2}_0(\Omega,\mathbb{R}^N)$ given by
$$
\mathcal{F}_t''(0)[\vec{u},\vec{v}]=\mathfrak{F}_t''(0)[\varphi_t^\ast\vec{u}, \varphi_t^\ast\vec{v}]$$ $$=\sum^N_{i,j=1}\sum_{|\alpha|,|\beta|\le m}
\int_{\Omega_t} A^{ij}_{\alpha\beta}(x)D^\alpha (v^i\circ\varphi_t^{-1})(x)D^\beta (u^j\circ\varphi_t^{-1})(x) dx
$$
with $A^{ij}_{\alpha\beta}(x)=F^{ij}_{\alpha\beta}(x, 0)$.
We say that the Lagrangian $F$ is {\it strongly elliptic on the function} $\vec{u}(x)$,
$x\in\Omega$, if there exists $M>0$ such that for all $\lambda\in\mathbb{R}^n$ and $\xi\in\mathbb{R}^N$,
$$
\sum^N_{i,j=1}\sum_{|\alpha|=|\beta|=m}
F^{ij}_{\alpha\beta}(x, \vec{u}(x),\cdots, D^m\vec{u}(x))\lambda^{\alpha+\beta}\xi_i{\xi}_j\ge M|\lambda|^{2m}\sum^N_{j=1}|\xi_j|^{2}.
$$

\begin{claim}\label{cl:strong}
If the Lagrangian $F$ satisfies (\ref{e:6.2}), then it is strongly elliptic on any
$\vec{u}\in W^{m,2}(\Omega,\mathbb{R}^N)$.
\end{claim}

\begin{proof}
For all $\lambda\in\mathbb{R}^n$ and $\xi\in\mathbb{R}^N$, by (\ref{e:6.2}) we have
\begin{eqnarray*}
&&\sum^N_{i,j=1}\sum_{|\alpha|=|\beta|=m}
F^{ij}_{\alpha\beta}(x, \vec{u}(x),\cdots, D^m\vec{u}(x))\lambda^{\alpha+\beta}\xi_i{\xi}_j\\
&\ge&\mathfrak{g}_2(0)
\sum^N_{i=1}\sum_{|\alpha|= m}(\lambda^\alpha\xi_i)^2=\mathfrak{g}_2(0)
\sum_{|\alpha|= m}(\lambda_1^{\alpha_1}\cdots\lambda_n^{\alpha_n})^2\sum^N_{j=1}|\xi_j|^{2}.
\end{eqnarray*}
Hence the claim follows because
$$
|\lambda|^{2m}=\left(\sum^n_{i=1}\lambda_i^2\right)^m$$ $$=\sum_{\alpha_1+\cdots+\alpha_n=m}
\frac{m!}{\alpha_1!\cdots\alpha_n!}\lambda_1^{2\alpha_1}\cdots\lambda_n^{2\alpha_n}\le
n^m\sum_{\alpha_1+\cdots+\alpha_n=m}\lambda_1^{2\alpha_1}\cdots\lambda_n^{2\alpha_n}
$$
and
$$
\sum_{\alpha_1+\cdots+\alpha_n=m}\frac{m!}{\alpha_1!\cdots\alpha_n!}=n^m.
$$
\end{proof}

Having above preparation we turn to main research questions in this section.
For the sake of simplicity,
we assume  that \textsf{$\Omega$ is star-shaped}
 with respect to the origin, and take
\begin{eqnarray}\label{e:BifE.star}
\Omega_t=\{tx\,|\,x\in\Omega\}\quad\hbox{and}\quad
\varphi_t(x)=tx\quad\hbox{for}\quad t\in (0, 1].
\end{eqnarray}
 Then since
$D^\alpha(\vec{u}\circ\varphi_t^{-1})(x)=\frac{1}{t^{|\alpha|}}(D^\alpha\vec{u})(x/t)$
we have
\begin{eqnarray}\label{e:BifE.13}
\mathcal{F}(t,\vec{u})&=&\int_{\Omega_t} F(x, \vec{u}(x/t),\cdots,
\frac{1}{t^m}(D^m\vec{u})(x/t))dx\nonumber\\
&=&\int_{\Omega} {\bf F}(x, \vec{u}(x),\cdots,
(D^m\vec{u})(x);t)dx
 \end{eqnarray}
 where ${\bf F}:\overline\Omega\times\prod^m_{k=0}\mathbb{R}^{N\times M_0(k)}\times(0, 1]\to\R$ is defined by
 \begin{eqnarray}\label{e:BifE.14}
 {\bf F}(x,\xi;t)=t^nF(tx, \xi^0, \frac{1}{t}\xi^1,\cdots, \frac{1}{t^m}\xi^m).
\end{eqnarray}
 Then
 \begin{eqnarray}\label{e:BifE.15}
  {\bf F}^i_\alpha(x,\xi;t)&=&t^{n-|\alpha|}F^i_\alpha(tx, \xi^0, \frac{1}{t}\xi^1,\cdots, \frac{1}{t^m}\xi^m),\\
  {\bf F}^{ij}_{\alpha\beta}(x,\xi;t)&=&t^{n-|\alpha|-|\beta|}F^{ij}_{\alpha\beta}(tx, \xi^0, \frac{1}{t}\xi^1,\cdots, \frac{1}{t^m}\xi^m).\label{e:BifE.16}
  \end{eqnarray}
  Fix a real $\epsilon\in (0, 1)$. Since
  $$\sum^N_{k=1}|\xi_\circ^k|=\sum^N_{k=1}(\sum_{|\alpha|<m-n/2}|\xi_\alpha^k|^2)^{1/2}\le M(m)\sum^N_{k=1}|\xi_\circ^k|,$$
  (\ref{e:6.1}) and (\ref{e:6.2}) imply
  {\scriptsize
  \begin{eqnarray}\label{e:BifE.17}
  |{\bf F}^{ij}_{\alpha\beta}(x,\xi;t)|&=&|t^{n-|\alpha|-|\beta|}F^{ij}_{\alpha\beta}(tx, \xi^0, \frac{1}{t}\xi^1,\cdots, \frac{1}{t^m}\xi^m)|\nonumber\\
  &&\hspace{-1cm}\le t^{n-2m}\mathfrak{g}_1\left(\sum^N_{k=1}(\sum_{|\alpha|<m-n/2}
  \frac{1}{t^{2|\alpha|}}|\xi_\alpha^k|^2)^{1/2}\right)
  \left(1+ \sum^N_{k=1}\sum_{m-n/2\le |\gamma|\le
m}\frac{1}{t^{|\gamma|}}|\xi^k_\gamma|^{2_\gamma}\right)^{2_{\alpha\beta}}\nonumber\\
&&\hspace{-1cm}\le \max\{1,\epsilon^{n-2m}\}\mathfrak{g}_1(\epsilon^{n-2m}M(m)\sum^N_{k=1}|\xi_\circ^k|)
\left(1+ \sum^N_{k=1}\sum_{m-n/2\le |\gamma|\le
m}|\xi^k_\gamma|^{2_\gamma}\right)^{2_{\alpha\beta}}
  \end{eqnarray}}
 for all $|\alpha|, |\beta|\le m$, $i,j=1,\cdots,N$ and $(x,\xi,t)\in\overline\Omega\times\prod^m_{k=0}\mathbb{R}^{N\times M_0(k)}\times [\epsilon, 1]$,  and
 {\footnotesize
  \begin{eqnarray}\label{e:BifE.17*}
\sum^N_{i,j=1}\sum_{|\alpha|=|\beta|=m}{\bf F}^{ij}_{\alpha\beta}(x,\xi;t)\eta^i_\alpha\eta^j_\beta&\ge&
t^{n-2m}\mathfrak{g}_2\left(\sum^N_{k=1}(\sum_{|\alpha|<m-n/2}
\frac{1}{t^{2|\alpha|}}|\xi_\alpha^k|^2)^{1/2}\right)
\sum^N_{i=1}\sum_{|\alpha|= m}(\eta^i_\alpha)^2\nonumber\\
&\ge&\min\{1, \epsilon^{n-2m}\}\mathfrak{g}_2(\sum^N_{k=1}|\xi_\circ^k|)
\sum^N_{i=1}\sum_{|\alpha|= m}(\eta^i_\alpha)^2
\end{eqnarray}}
for any $\eta=(\eta^{i}_{\alpha})\in\R^{N\times M_0(m)}$.
Hence (i) of the following proposition holds true.

 \begin{proposition}\label{prop:BifE.11}
Let $\Omega\subset\R^n$ be a star-shaped bounded domain,
$F$ be as in \textsf{Hypothesis} $\mathfrak{F}_{2,N,m,n}$.
 Suppose that $F$ is also $C^2$ and that ${\bf F}$ is defined by (\ref{e:BifE.14}).
Then for any fixed $\epsilon\in (0,1)$ the following holds.
\begin{enumerate}
\item[\rm (i)] All ${\bf F}(\cdot; t)$ satisfy Hypothesis~$\mathfrak{F}_{2,N,m,n}$, and the inequalities in (\ref{e:6.1}) and (\ref{e:6.2}) are uniformly
satisfied for all $t\in [\epsilon,1]$.
\item[\rm (ii)] For each $t\in (0,1]$, $D_t{\bf F}(x,\xi;t)$ is differentiable in each $\xi^i_\alpha$,
    each  ${\bf F}^i_\alpha(x,\xi;t)$ is differentiable in $t$, and
     $D_t{\bf F}^i_\alpha(x,\xi;t)=D_{\xi^i_\alpha}D_t{\bf F}(x,\xi;t)$. And
      $$
\sup_{|\alpha|\le m}\sup_{1\le i\le N}\sup_{t\in[\epsilon,1]}\int_\Omega
\left[|D_t {\bf F}(x,0;t)|+ |D_t {\bf F}^i_{\alpha}(x, 0;t)|^{2'_\alpha}\right]dx<\infty.
$$
\item[\rm (iii)] Suppose that  for all $\alpha, i, l$  and
$(x, \xi)\in\overline\Omega\times\prod^m_{k=0}\mathbb{R}^{N\times M(m)}$,
$F$ also satisfies:
\begin{eqnarray}\label{e:BifE.18}
&&|F^i_{\alpha x_l}(x,\xi)|\le  \mathfrak{g}_0(\sum^N_{k=1}|\xi_\circ^k|)\sum_{|\beta|<m-n/2}\bigg(1+
\sum^N_{k=1}\sum_{m-n/2\le |\gamma|\le
m}|\xi^k_\gamma|^{2_\gamma }\bigg)^{2_{\alpha\beta}}\nonumber\\
&&+\mathfrak{g}_0(\sum^N_{k=1}|\xi^k_\circ|)\sum^N_{l=1}\sum_{m-n/2\le |\beta|\le m} \bigg(1+
\sum^N_{k=1}\sum_{m-n/2\le |\gamma|\le m}|\xi^k_\gamma|^{2_\gamma }\bigg)^{2_{\alpha\beta}}|\xi^l_\beta|,
\end{eqnarray}
 where $\mathfrak{g}_0:[0,\infty)\to\mathbb{R}$ is a continuous, positive, nondecreasing function,  and is constant if $m<n/2$.
 Then for all $i=1,\cdots,N$, $|\alpha|\le m$ and $(x, \xi, t)\in\overline\Omega\times\prod^m_{k=0}\mathbb{R}^{N\times M(m)}\times [\epsilon, 1]$,
 \begin{eqnarray}\label{e:BifE.19}
&&|D_\lambda {\bf F}^i_\alpha(x,\xi;t)|\le|D_t {\bf F}^i_\alpha(x,0;t)|\nonumber\\
&&+ \mathfrak{g}(\sum^N_{k=1}|\xi_\circ^k|)\sum_{|\beta|<m-n/2}\bigg(1+
\sum^N_{k=1}\sum_{m-n/2\le |\gamma|\le
m}|\xi^k_\gamma|^{2_\gamma }\bigg)^{2_{\alpha\beta}}\nonumber\\
&&+\mathfrak{g}(\sum^N_{k=1}|\xi^k_\circ|)\sum^N_{l=1}\sum_{m-n/2\le |\beta|\le m} \bigg(1+
\sum^N_{k=1}\sum_{m-n/2\le |\gamma|\le m}|\xi^k_\gamma|^{2_\gamma }\bigg)^{2_{\alpha\beta}}|\xi^l_\beta|,
\end{eqnarray}
 where $\mathfrak{g}:[0,\infty)\to\mathbb{R}$ is a continuous, positive,
 nondecreasing function.
\end{enumerate}
\end{proposition}
\begin{proof}
 Since $F$ is $C^2$ and
 \begin{eqnarray*}
&& D_t{\bf F}(x,\xi;t)\\&=&nt^{n-1}F(tx, \xi^0, \frac{1}{t}\xi^1,\cdots, \frac{1}{t^m}\xi^m)+
 t^{n+1}\sum^n_{l=1}D_{x_l}F(tx, \xi^0, \frac{1}{t}\xi^1,\cdots, \frac{1}{t^m}\xi^m)\nonumber\\
 &+& \sum^N_{j=1}\sum^m_{|\gamma|=1}t^{n-|\gamma|}F^j_{\gamma}(tx, \xi^0, \frac{1}{t}\xi^1,\cdots, \frac{1}{t^m}\xi^m),
   \end{eqnarray*}
 and thus
 \begin{eqnarray}\label{e:BifE.20}
 D_t{\bf F}^i_\alpha(x,\xi;t)&=&D_{\xi^i_\alpha}D_t{\bf F}(x,\xi;t)\nonumber\\
 &=&n t^{n-|\alpha|-1}F^i_\alpha(tx, \xi^0, \frac{1}{t}\xi^1,\cdots, \frac{1}{t^m}\xi^m)\nonumber\\
 &+&t^{n-|\alpha|+1}\sum^n_{l=1}D_{x_l}F^i_{\alpha}(tx, \xi^0, \frac{1}{t}\xi^1,\cdots, \frac{1}{t^m}\xi^m)\nonumber\\
 &+& \sum^N_{j=1}\sum^m_{|\gamma|=1} t^{n-|\alpha|-|\gamma|}F^{ij}_{\alpha\gamma}(tx, \xi^0, \frac{1}{t}\xi^1,\cdots, \frac{1}{t^m}\xi^m),
   \end{eqnarray}
   (ii) is clear.
By \cite[Proposition~4.3]{Lu7},
$\mathfrak{F}_{2,N,m,n}$ implies that for $\mathfrak{g}_4(t):=\mathfrak{g}_1(t)t+\mathfrak{g}_1(t)$,
\begin{eqnarray*}
|F^k_\alpha(x,\xi)|&\le&|F^k_\alpha(x,0)|
+ \mathfrak{g}_4(\sum^N_{i=1}|\xi^i_\circ|)\sum_{|\beta|<m-n/2}\Bigg(1+
\sum^N_{i=1}\sum_{m-n/2\le |\gamma|\le
m}|\xi^i_\gamma|^{2_\gamma }\Bigg)^{2_{\alpha\beta}}\nonumber\\
&+&\mathfrak{g}_4(\sum^N_{i=1}|\xi^i_\circ|)\sum_{m-n/2\le |\beta|\le m} \Bigg(1+
\sum^N_{i=1}\sum_{m-n/2\le |\gamma|\le
m}|\xi^i_\gamma|^{2_\gamma }\Bigg)^{2_{\alpha\beta}}\sum^N_{j=1}|\xi^j_\beta|.
\end{eqnarray*}
As arguments above Proposition~\ref{prop:BifE.11},
 (\ref{e:BifE.19}) easily follows from (\ref{e:BifE.20}), this and (\ref{e:BifE.17})-(\ref{e:BifE.18}).
\end{proof}

\begin{remark}\label{rm:BifE11}
{\rm Under the assumptions of Proposition~\ref{prop:BifE.11} without (iii),
suppose that $F$ is $C^3$ and that
 there exists a continuous, positive, nondecreasing functions
 $\overline{\mathfrak{g}}$ such that
\begin{eqnarray}\label{e:BifE.21}
 |D_{x_l}F^{ij}_{\alpha\beta}(x,\xi)|+ |F^{ijk}_{\alpha\beta\gamma}(x,\xi)|\le
\overline{\mathfrak{g}}(\sum^N_{k=1}|\xi_\circ^k|)\left(1+
\sum^N_{k=1}\sum_{m-n/2\le |\gamma|\le
m}|\xi^k_\gamma|^{2_\gamma}\right)^{2_{\alpha\beta}}
\end{eqnarray}
for $l=1,\cdots,n$, $i,j,k=1,\cdots,N$ and $|\alpha|, |\beta|\le m, 1\le|\gamma|\le m$
 and $(x, \xi)\in\overline\Omega\times\R^{M(m)}$.
Since by (\ref{e:BifE.20}) we may obtain
 \begin{eqnarray}\label{e:BifE.22}
 D_t{\bf F}^{ij}_{\alpha\beta}(x,\xi;t)&=&D_{\xi^j_\beta}D_t{\bf F}^i_\alpha(x,\xi;t)\nonumber\\
 &=&nt^{n-|\alpha|-|\beta|-1}F^{ij}_{\alpha\beta}(tx, \xi^0, \frac{1}{t}\xi^1,\cdots, \frac{1}{t^m}\xi^m)\nonumber\\
 &+&t^{n-|\alpha|-|\beta|+1}\sum^n_{l=1}D_{x_l}F^{ij}_{\alpha\beta}(tx, \xi^0, \frac{1}{t}\xi^1,\cdots, \frac{1}{t^m}\xi^m)\nonumber\\
 &+& \sum^N_{k=1}\sum^m_{|\gamma|=1} t^{n-|\alpha|-|\beta|-|\gamma|}F^{ijk}_{\alpha\beta\gamma}(tx, \xi^0, \frac{1}{t}\xi^1,\cdots, \frac{1}{t^m}\xi^m),
   \end{eqnarray}
 it follows from  (\ref{e:BifE.21}) and (\ref{e:6.1}) that
\begin{eqnarray}\label{e:BifE.23}
 |D_t{\bf F}^{ij}_{\alpha\beta}(x,\xi;\lambda)|\le
\mathfrak{g}^\star(\sum^N_{k=1}|\xi_\circ^k|)\left(1+
\sum^N_{k=1}\sum_{m-n/2\le |\gamma|\le
m}|\xi^k_\gamma|^{2_\gamma}\right)^{2_{\alpha\beta}}
\end{eqnarray}
for some continuous, positive, nondecreasing functions $\mathfrak{g}^\star$ and all
$(x,\xi,t)\in\overline\Omega\times\prod^m_{k=0}
\mathbb{R}^{N\times M_0(k)}\times [\epsilon, 1]$
 and  $i,j=1,\cdots,N$ and $|\alpha|, |\beta|\le m$.
By (b) of Remark~\ref{rm:BifE10}, (\ref{e:BifE.23}) may lead to (\ref{e:BifE.19}).}
\end{remark}

The following is the first of main results in this section.

\begin{theorem}\label{th:BifE.11}
Let $\Omega\subset\mathbb{R}^n$ be a star-shaped bounded  domain
with $C^{\infty}$ boundary, and let $\Omega_t$ and
$\varphi_t$ be as in (\ref{e:BifE.star}) for $t\in (0, 1]$.
Let $F$ be as in \textsf{Hypothesis} $\mathfrak{F}_{2,N,m,n}$, $C^2$ and satisfy (\ref{e:BifE.18}).
Suppose that $\vec{u}=0$ is a solution of (\ref{e:BifE.10}) for each $t\in (0, 1]$, i.e.,
  \begin{eqnarray}\label{e:BifE.24-}
 \sum_{|\alpha|\le m}(-1)^{|\alpha|}D^\alpha F^i_\alpha(x, 0)=0\quad
\forall x\in\Omega,\quad i=1,\cdots,N,
  \end{eqnarray}
 all $F^{ij}_{\alpha\beta}(\cdot, 0)$ belong to $C^\infty(\overline{\Omega})$,
 and that either
 \begin{eqnarray}\label{e:BifE.24}
 \sum^N_{i,j=1}\sum_{|\alpha|,|\beta|\le m}
(-1)^{\alpha}D^\alpha\bigl[F^{ij}_{\alpha\beta}(x, 0)D^\beta v^j\bigr]=0
 \end{eqnarray}
have no nontrivial solutions in $C^\infty(\overline{\Omega},\mathbb{R}^N)$ with compact
support in any of the manifolds $\Omega_t$, or (\ref{e:BifE.24}) has no solutions $\vec{u}\ne 0$ such that $\vec{u}$
vanishes on some open set in $\Omega$. (This implies $\nu_0=0$.)
Then there is only a finite number of $t\in (0, 1]$ with positive $\nu_t$.
Moreover,
$(t,0)\in (0, 1]\times W^{m,2}_0(\Omega,\mathbb{R}^N)$
is a bifurcation point  for the equation
\begin{eqnarray}\label{e:BifE.25}
\mathcal{F}_{\vec{u}}'(t, \vec{u})=0
\end{eqnarray}
 if and only if $\nu_t\ne 0$.
\end{theorem}

 If $n=\dim\Omega=1$ and $m\ge 1$ the similar result
 can be proved with \cite[Theorem~2.4]{Uh}, see \cite{Lu9}.

 \begin{proof}[Proof of Theorem~\ref{th:BifE.11}]
 {\bf Step 1}. By the assumptions on $F$ and Claim~\ref{cl:strong},
 the operator $L$ in (\ref{e:strong*}) with $A^{ij}_{\alpha\beta}(x)=F^{ij}_{\alpha\beta}(x,0)$
 is strongly elliptic and self-adjoint.
 Hence $\mu(L_t)=\mu_t$ and $\nu(L_t)=\nu_t$ by the arguments below
 (\ref{e:BifE.11}) and definitions of
 $\mu(L_t)$ and $\nu(L_t)$. From Proposition~\ref{prop:BifE.10}
  we deduce that
\begin{eqnarray}\label{e:BifE.26}
\mu_0-\mu_s=\sum_{0< t\le s}\nu_t\quad\forall s\in (0, 1]
\end{eqnarray}
and that $\mu_1=0$ if $\Omega_1$ has sufficiently small volume.
  The first claim follows directly. ({\it Note}: We do not use
  the assumptions that $\Omega$ is star-shaped
  and the special deformation is as in (\ref{e:BifE.star}).)

\vspace{4pt}\noindent
\noindent{\bf Step 2} ({\it Prove the second claim}).
Assume $\nu_s\ne 0$ for some $s\in (0, 1]$.  By the first claim we may choose  $0<\epsilon<s$
such that $\nu_\epsilon=0$ and $s=s_1<\cdots<s_k$ are all points in
$[\epsilon,1]$ in which $\nu_t\ne 0$.
Clearly, (\ref{e:BifE.26}) implies  that
$\mu_t=\mu_s$ for all $t\in [\epsilon, s]$, and $\mu_t=\mu_s+\nu_s$ for all
$t\in (s,s_2)$. Thus in the sense of \cite[\S2.1]{Lu7} $0\in W^{m,2}_0(\Omega,\mathbb{R}^N)$
is a nondegenerate critical point of
$\mathcal{F}_{t}$ for each $t\in [\epsilon, s)\cup (s,s_2)$.
It follows from \cite[Theorem~2.1]{Lu7}
that
\begin{eqnarray}\label{e:BifE.27}
C_q(\mathcal{F}_{t}, 0;{\bf K})=\delta_{q\mu_{s}}\;\forall t\in [\epsilon, s)\quad\hbox{and}\quad
C_q(\mathcal{F}_{t}, 0;{\bf K})=\delta_{q(\mu_{s}+\nu_s)}\;\forall t\in (s,s_2).
\end{eqnarray}
By Proposition~\ref{prop:BifE.11},  the  function
$$
\overline\Omega\times\prod^m_{k=0}\mathbb{R}^{N\times M_0(k)}\times [\epsilon, 1]\ni (x,
\xi, t)\mapsto {\bf F}(x,\xi;t)=t^nF(tx, \xi^0, \frac{1}{t}\xi^1,\cdots, \frac{1}{t^m}\xi^m)\in\R
$$
satisfies the conditions of  Theorem~\ref{th:BifE.9}. Hence
for any $t_1<t_2$ satisfying $\epsilon\le t_1<s<t<t_2<s_2$,
applying  Theorem~\ref{th:BifE.9} to the family $[t_1,t_2]\ni t\mapsto\mathcal{F}_{t}$
we derive from (\ref{e:BifE.27}) that
 for some $\bar{t}\in[t_1, t_2]$, $(\bar{t},0)\in (0, 1]\times W^{m,2}_0(\Omega,\mathbb{R}^N)$ is a bifurcation point  for
(\ref{e:BifE.25}). Since $t_1$ and $t_2$ may be  arbitrarily close to $s$,
$(s,0)\in (0, 1]\times W^{m,2}_0(\Omega,\mathbb{R}^N)$ must be a bifurcation point for (\ref{e:BifE.25}).

Conversely, let $(s,0)\in (0, 1]\times W^{m,2}_0(\Omega,\mathbb{R}^N)$
be a bifurcation point  for the equation
(\ref{e:BifE.25}). By the first claim we may choose  $0<\rho<s$ such that
$\nu_t=0$ for any $s\ne t\in\Lambda:=[s-\rho, s+\rho]\cap (0,1]$.
If $\{\mathcal{F}_t\,|\, t\in\Lambda\}$ satisfies the conditions of
\cite[Theorem~3.1]{Lu8} then $\nu_s\ne 0$.

\vspace{4pt}\noindent
\noindent{\bf Step 3} ({\it Check that \cite[Theorem~3.1]{Lu8} is applicable to $\{\mathcal{F}_t\,|\, t\in\Lambda\}$}).
Let $(\vec{u}_k)\subset W^{m,2}_0(\Omega,\mathbb{R}^N)$ converge to zero, and let $(t_k)\subset\Lambda$ converge to $s$.
By (\ref{e:6.6}) and (\ref{e:BifE.16})
\begin{eqnarray*}
   && ([P_{t_k}(\vec{u}_k)-P_{s}(0)]\vec{w},\vec{\varphi})_{m,2}\\
   &&\hspace{-.7cm}=\sum^N_{i,j=1}\sum_{|\alpha|=|\beta|=m}
   \int_\Omega [{\bf F}^{ij}_{\alpha\beta}(x, \vec{u}_k(x),\cdots, D^m \vec{u}_k(x);t_k)- {\bf F}^{ij}_{\alpha\beta}(x, 0;s)]D^\beta w^j\cdot D^\alpha \varphi^i dx
   \end{eqnarray*}
    for any $\vec{w}, \vec{\varphi}\in W^{m,2}_0(\Omega,\mathbb{R}^N)$.
    It follows from  the H\"older inequality that
 \begin{eqnarray*}
   && \|P_{t_k}(\vec{u}_k)-P_{s}(0)]\vec{w}\|_{m,2}\\
   &\le&\sum^N_{i,j=1}\sum_{|\alpha|=|\beta|=m}
  \\&& \left(\int_\Omega |{\bf F}^{ij}_{\alpha\beta}(x, \vec{u}_k(x),\cdots, D^m \vec{u}_k(x);t_k)- {\bf F}^{ij}_{\alpha\beta}(x, 0;s)|^2|D^\beta w^j|^2 dx\right)^{1/2}
   \end{eqnarray*}
 According to the condition (i) of \cite[Theorem~3.1]{Lu8},
 it suffices to prove that
\begin{eqnarray}\label{e:BifE.28}
\left(\int_\Omega |{\bf F}^{ij}_{\alpha\beta}(x, \vec{u}_k(x),\cdots, D^m \vec{u}_k(x);t_k)- {\bf F}^{ij}_{\alpha\beta}(x, 0;s)|^2|D^\beta w^j|^2 dx\right)^{1/2}\to 0
\end{eqnarray}
for given $|\alpha|=|\beta|=m$ and $1\le i,j\le N$.
Note that $2_{\alpha\beta}=  1-\frac{1}{2_\alpha}-\frac{1}{2_\beta}=0$ for $|\alpha|=|\beta|=m$. By (\ref{e:BifE.17}) we obtain
\begin{eqnarray}\label{e:BifE.29}
  |{\bf F}^{ij}_{\alpha\beta}(x,\xi;t)|
&\le&\max\{1,(s-\rho)^{n-2m}\}\mathfrak{g}_1
\left((s-\rho)^{n-2m}M(m)\sum^N_{l=1}|\xi_\circ^l|\right)
  \end{eqnarray}
 for all $|\alpha|=|\beta|=m$, $i,j=1,\cdots,N$ and $(x,\xi,t)\in\overline\Omega\times\prod^m_{k=0}\mathbb{R}^{N\times M_0(k)}\times \Lambda$.
Since $\|\vec{u}_k\|_{m,2}\to 0$, by the Sobolev embedding theorem there exists
  a constant $R>0$ such that
\begin{eqnarray}\label{e:BifE.30}
\sum^N_{l=1}\sum_{|\alpha|<m-n/2}|D^\alpha u_k^l(x)|^2<R^2\quad\forall (k, x)\in\N\times\overline{\Omega}.
\end{eqnarray}
Take a continuous function $\chi:\mathbb{R}\to\mathbb{R}$ such that
$$
\chi(t)=t\;\forall |t|\le 2R,\quad \chi(t)=\pm 3R\;\forall \pm t\ge 3R,\quad\hbox{and}\; |\chi(t)|\le 3R,
$$
and define a function
$\tilde{\bf F}^{ij}_{\alpha\beta}:\overline\Omega\times\R^{M(m)}\times\Lambda\to \R$ by
\begin{eqnarray}\label{e:BifE.31}
\tilde{\bf F}^{ij}_{\alpha\beta}(x,\xi, t)={\bf F}^{ij}_{\alpha\beta}(x, \tilde\xi,t),
\end{eqnarray}
where $\tilde\xi^k=(\tilde\xi^i_\alpha)\in\R^{N\times M_0(k)}$,
$\tilde\xi^i_\alpha=\chi(\xi^i_\alpha)$ if $|\alpha|<m-n/2$, and
$\tilde\xi^i_\alpha=\xi^i_\alpha$ if $m-n/2\le |\alpha|\le m$.
Clearly, ${\bf F}^{ij}_{\alpha\beta}(x, 0;t)=\tilde{\bf F}^{ij}_{\alpha\beta}(x, 0;t)\;\forall t$ and $\tilde{\bf F}^{ij}_{\alpha\beta}$ is continuous.
Recall  $\sum^N_{l=1}|\xi_\circ^l|=\sum^N_{l=1}\sum_{|\alpha|<m-n/2}|\xi_\alpha^l|^2$.
 We have
\begin{eqnarray}\label{e:BifE.32}
\sum^N_{l=1}|\tilde{\xi}_\circ^l|=
\sum^N_{l=1}\sum_{|\alpha|<m-n/2}|\tilde{\xi}_\alpha^l|^2\le 9NM(m)R^2,\quad\forall\xi\in\R^{M(m)}.
\end{eqnarray}
Hence $\tilde{\bf F}^{ij}_{\alpha\beta}$ is bounded by (\ref{e:BifE.29}).
It is easy to see that  (\ref{e:BifE.30}) implies
$$
{\bf F}^{ij}_{\alpha\beta}(x, \vec{u}_k(x),\cdots, D^m \vec{u}_k(x);t_k)=
\tilde{\bf F}^{ij}_{\alpha\beta}(x, \vec{u}_k(x),\cdots, D^m \vec{u}_k(x);t_k)\quad\forall(k,x)\in\N\times\overline{\Omega}.
$$
Since   $\|D^\alpha{u}^i_k\|_{2}\to 0$ for all $|\alpha|\le m$ and $i=1,\cdots,N$,
 we can apply Proposition~\ref{prop:C.2} to $f(x,\xi;t)=\tilde{\bf F}^{ij}_{\alpha\beta}(x,\xi;t)D^\beta w^j(x)$ to get
\begin{eqnarray*}
\left(\int_\Omega |\tilde{\bf F}^{ij}_{\alpha\beta}(x, \vec{u}_k(x),\cdots,
D^m \vec{u}_k(x);t_k)-
\tilde{\bf F}^{ij}_{\alpha\beta}(x, 0;t_k)|^2|D^\beta w^j|^2 dx\right)^{1/2}\to 0.
\end{eqnarray*}
But the Lebesgue dominated convergence theorem leads to
\begin{eqnarray*}
\left(\int_\Omega |\tilde{\bf F}^{ij}_{\alpha\beta}(x, 0;t_k)- \tilde{\bf F}^{ij}_{\alpha\beta}(x, 0;s)|^2|D^\beta w^j|^2 dx\right)^{1/2}\to 0.
\end{eqnarray*}
(\ref{e:BifE.28}) follows from these immediately.

  For any $\vec{u}, \vec{w}\in W^{m,2}_0(\Omega,\mathbb{R}^N)$ and
  $0<\epsilon\le t\le 1$, from (\ref{e:BifE.17*}) we derive
\begin{eqnarray*}
   && (P_{t}(\vec{u})\vec{w},\vec{w})_{m,2}\\
   &=&\sum^N_{i,j=1}\sum_{|\alpha|+|\beta|=2m}
   \int_\Omega {\bf F}^{ij}_{\alpha\beta}(x, \vec{u}(x),\cdots, D^m \vec{u}(x);t_k)D^\beta w^j\cdot D^\alpha w^i dx\nonumber\\
&\ge&\min\{1, \epsilon^{2n-2m}\}\mathfrak{g}_2(0)\sum^N_{i=1}\sum_{|\alpha|= m}\int_\Omega |D^\alpha w^i|^2 dx
   \end{eqnarray*}
 and hence  the condition (ii) of \cite[Theorem~3.1]{Lu8} is satisfied for
 $\mathcal{F}_t$ with $t\in\Lambda$.

Let $(t_k)\subset\Lambda$ converge to $s$. For any
$\vec{w}, \vec{\varphi}\in W^{m,2}_0(\Omega,\mathbb{R}^N)$,
by (\ref{e:6.7}) and (\ref{e:BifE.16}) we get
\begin{eqnarray*}
   && ([Q_{t_k}(0)-Q_{s}(0)]\vec{w},\vec{\varphi})_{m,2}\nonumber\\
   &=&\sum^N_{i,j=1}\sum_{|\alpha|+|\beta|<2m}\int_\Omega [{\bf F}^{ij}_{\alpha\beta}(x, 0;t_k)- {\bf F}^{ij}_{\alpha\beta}(x, 0;
   s)]D^\beta w^j\cdot D^\alpha \varphi^i dx
     \end{eqnarray*}
and therefore using the H\"older inequality and the Sobolev embedding theorem
we obtain  a constant $C=C(m,n,N,\Omega)>0$ such that
\begin{eqnarray}\label{e:BifE.33}
   && \|Q_{t_k}(0)-Q_{s}(0)\|_{\mathscr{L}(W^{m,2}(\Omega,\R^N)}\nonumber\\
    &\le&C\sum^N_{i,j=1}\sum_{|\alpha|+|\beta|<2m}\left(\int_\Omega |{\bf F}^{ij}_{\alpha\beta}(x,0;t_k)-
   {\bf F}^{ij}_{\alpha\beta}(x,0;s)|^{1/2_{\alpha\beta}} dx\right)^{2_{\alpha\beta}}.
  \end{eqnarray}
Note that $2_{\alpha\beta}\in (0, 1]$ for all $|\alpha|+|\beta|<2m$ and that (\ref{e:BifE.17}) implies
 \begin{eqnarray*}
 |{\bf F}^{ij}_{\alpha\beta}(x, 0;t)|\le\max\{1,\epsilon^{n-2m}\}\mathfrak{g}_1(0),\quad\forall (x,t)\in\overline{\Omega}\times [\epsilon,1].
  \end{eqnarray*}
 From the Lebesgue dominated convergence we deduce that
 $$
 \|Q_{t_k}(0)-Q_{s}(0)\|_{\mathscr{L}(W^{m,2}(\Omega,\R^N)}\to 0.
 $$
  That is, the condition (iv) of \cite[Theorem~3.1]{Lu8} is satisfied for
 $\mathcal{F}_t$ with $t\in\Lambda$.

  It remains to be proved  that  $\{\mathcal{F}_t\,|\,t\in\Lambda\}$
  satisfy the condition (iii) of \cite[Theorem~3.1]{Lu8}.
As in (\ref{e:BifE.33}), for the constant $C$ therein and all
$(t,\vec{u})\in (0, 1]\times W^{m,2}_0(\Omega,\mathbb{R}^N)$ we have
\begin{eqnarray}\label{e:BifE.34}
   && \|[Q_{t}(\vec{u})-Q_{t}(0)\|_{\mathscr{L}(W^{m,2}(\Omega,\R^N)}\\
   &\le&\sum^N_{i,j=1}\sum_{|\alpha|+|\beta|<2m}\biggl(
 \int_\Omega \Big|{\bf F}^{ij}_{\alpha\beta}(x, \vec{u}(x),\cdots, D^m \vec{u}(x);t)-
 {\bf F}^{ij}_{\alpha\beta}(x, 0;t)\Big|^{1/2_{\alpha\beta}} dx\biggr)^{2_{\alpha\beta}}.\nonumber
 \end{eqnarray}
Fix a $R>0$. Let $\tilde{\bf F}^{ij}_{\alpha\beta}:\overline\Omega\times\R^{M(m)}\times\Lambda\to \R$ be defined by
(\ref{e:BifE.31}). If $\vec{u}\in W^{m,2}_0(\Omega,\mathbb{R}^N)$ satisfies $\|\vec{u}\|_{m,2}\le R$ then
$$
{\bf F}^{ij}_{\alpha\beta}(x, \vec{u}(x),\cdots, D^m \vec{u}(x);t)=
\tilde{\bf F}^{ij}_{\alpha\beta}(x, \vec{u}(x),\cdots, D^m \vec{u}(x);t)\quad\forall(t,x)\in (0, 1]\times\overline{\Omega}.
$$
Clearly,  (\ref{e:BifE.17}) and (\ref{e:BifE.32}) imply that for all $(x, \xi, t)\in \overline{\Omega}\times \R^{M(m)}\times [\epsilon, 1]$,
 { \footnotesize
  \begin{eqnarray*}
  |\tilde{\bf F}^{ij}_{\alpha\beta}(x,\xi, t)|&=&|{\bf F}^{ij}_{\alpha\beta}(x, \tilde\xi,t)|\\
 &&\hspace{-1cm}\le \max\{1,\epsilon^{n-2m}\}\mathfrak{g}_1(\epsilon^{n-2m}M(m)9NM(m)R^2)\left(1+
\sum^N_{k=1}\sum_{m-n/2\le |\gamma|\le
m}|\xi^k_\gamma|^{2_\gamma}\right)^{2_{\alpha\beta}}\\
&&\hspace{-1cm}\le \max\{1,\epsilon^{n-2m}\}\mathfrak{g}_1(9N\epsilon^{n-2m}M(m)^2R^2)\left(1+
\sum^N_{k=1}\sum_{m-n/2\le |\gamma|\le
m}|\xi^k_\gamma|^{2_\gamma 2_{\alpha\beta}}\right)
  \end{eqnarray*}}
because $(x_1+\cdots+x_\nu)^q\le x_1^q+\cdots+x_\nu^q$ for any $q\in (0,1)$ and real $x_j\ge 0$, $j=1,\cdots,\nu$.
From Proposition~\ref{prop:C.2} we deduce that maps
$$
\prod_{|\gamma|\le m}(L^{2_\gamma}(\Omega))^N\to L^{1/2_{\alpha\beta}}(\Omega),\;
{\bf u}=\{u^i_\gamma\,|\,|\gamma|\le m,\;1\le i\le N\}\to \tilde{\bf F}^{ij}_{\alpha\beta}(\cdot, {\bf u}, t)
$$
are uniformly continuous at $0\in\prod_{|\gamma|\le m}(L^{2_\gamma}(\Omega))^N$ with respect to $t\in [\epsilon,1]$.
But $\|\vec{u}\|_{m,2}\to 0$ leads to $\|D^\gamma u^i\|_{2_\gamma}\to 0$ for all
$|\gamma|\le m$ and $1\le i\le N$. It follows from (\ref{e:BifE.34}) that
$$
\|\vec{u}\|_{m,2}\to 0\;\Longrightarrow\;
\|Q_{t}(\vec{u})-Q_{t}(0)\|_{\mathscr{L}(W^{m,2}(\Omega,\R^N)}\to 0\;
\hbox{uniformly in $t\in [\epsilon,1]$}.
$$
\end{proof}

\begin{corollary}\label{cor:BifE.12}
In Theorem~\ref{th:BifE.11},  the sentence ``Let $F$ be as in \textsf{Hypothesis} $\mathfrak{F}_{2,N,m,n}$, $C^2$ and satisfy (\ref{e:BifE.18})."
can be replaced by ``Let $F$ be as in \textsf{Hypothesis} $\mathfrak{F}_{2,N,m,n}$, $C^3$ and satisfy (\ref{e:BifE.21})."
\end{corollary}

\begin{corollary}\label{cor:BifE.13}
In Theorem~\ref{th:BifE.11}, if $m=1$ then the sentence ``Let $F$ be as in \textsf{Hypothesis} $\mathfrak{F}_{2,N,m,n}$, $C^2$ and satisfy (\ref{e:BifE.18})."
can be replaced by ``Let $F$ be as in the {\bf CGC} above Theorem~\ref{th:6.1}."
\end{corollary}

\begin{proof}
 Note that the function $F$ had been required to be $C^2$ in the {\bf CGC}.
It was proved in \cite[Proposition~4.22]{Lu7} that the {\bf CGC} implies \textsf{Hypothesis} $\mathfrak{F}_{2,N,1,n}$  above Theorem~\ref{th:6.1}. Moreover, in the present case we can write $F$ as
$$
\overline\Omega\times\mathbb{R}^N\times\mathbb{R}^{N\times n}\ni (x,
z,p)\mapsto F(x, z,p)\in\R.
$$
 Let $\kappa_n=2n/(n-2)$ for $n>2$, and $\kappa_n\in (2,\infty)$ for $n=2$.
Then (\ref{e:BifE.18}) can be equivalently expressed as:  There exist positive constants  $\mathfrak{g}'_1$,  $\mathfrak{g}'_2$ and
$s\in (0, \frac{\kappa_n-2}{\kappa_n})$, $r_\alpha\in (0,\frac{\kappa_n-2}{2\kappa_n})$
 for each $\alpha\in \mathbb{N}_0^n$ with $|\alpha|=1$,  such that for $i=1,\cdots,N$, $l=1,\cdots,n$ and $|\alpha|=1$,
\begin{eqnarray*}
&|F_{z_jx_l}(x,z,p)|\le \mathfrak{g}'_1\left(1+\sum^N_{l=1}|z_l|^{\kappa_n}+
\sum^N_{k=1}|p^k_\alpha|^2\right)^{s+\frac{1}{2}},\\
&|F_{p^i_\alpha x_l}(x,z,p)|\le \mathfrak{g}'_2\left(1+\sum^N_{l=1}|z_l|^{\kappa_n}+
\sum^N_{k=1}|p^k_\alpha|^2\right)^{r_\alpha+\frac{1}{2}}.
\end{eqnarray*}
These are clearly implied in the {\bf CGC}.
\end{proof}

When $m=N=1$, instead of using Proposition~\ref{prop:BifE.10}
 we can get the following generalization of Theorem~\ref{th:BifE.11} with Proposition~\ref{prop:BifE.10*}, the second of main results in this section.

\begin{theorem}\label{th:BifE.12}
Let $\Omega\subset\mathbb{R}^n$ be a star-shaped bounded domain
with $C^{1,1}$ boundary,  $\Omega_t=\{tx\,|\,x\in\Omega\}$ and
$\varphi_t(x)=tx$ for $t\in (0, 1]$. For a $C^3$ function
$\overline\Omega\times\mathbb{R}\times\mathbb{R}^{n}\ni (x,z, p)\mapsto F(x, z, p)\in\R$
satisfying {\bf CGC}, that is, there exist positive constants
$\nu, \mu, \lambda, M_1, M_2$,  such that
\begin{eqnarray*}
&\nu\left(1+|z|^2+|p|^2\right)-\lambda\le F(x, z,p)
\le\mu\left(1+|z|^2+|p|^2\right),\\
&|F_{p_i}(x,z,p)|, |F_{p_ix_l}(x,z,p)|, |F_{z}(x,z,p)|, |F_{zx_l}(x,z,p)|\le \mu\left(1+|z|^2+|p|^2\right)^{1/2},\\
&|F_{p_i z}(x,z,p)|,\quad |F_{zz}(x,z,p)|\le \mu,\\
&M_1|\eta|^2\le\sum^n_{i,j=1}F_{p_ip_j}(x,z,p)\eta_i\eta_j\le
M_2|\eta|^2\;\forall \eta\in\R^{n},
\end{eqnarray*}
suppose that $F_z(x,0)-\sum^n_{i=1}\partial_{p_i}F_{p_i}(x,0)=0$ for all $x\in\Omega$.
Then for the functional
$$
\mathcal{F}:(0,1]\times W^{1,2}_0(\Omega,\mathbb{R}^N)\to\mathbb{R},\;(t,\vec{u})\mapsto
\mathcal{F}_t(\vec{u})=
\int_{\Omega_t} F\left(x, \vec{u}(x/t), \frac{1}{t}(D\vec{u})(x/t)\right)dx,
$$
there is only a finite number of $t\in (0, 1]$ with positive $\nu_t$.
Moreover,
$(t,0)\in (0, 1]\times W^{1,2}_0(\Omega,\mathbb{R}^N)$ is a
bifurcation point  for the equation
$\mathcal{F}_{\vec{u}}'(t, \vec{u})=0$ if and only if $\nu_t\ne 0$.
\end{theorem}

\begin{proof}
In Theorem~\ref{th:BifE.11} we require that
 $\partial\Omega$ is $C^{\infty}$ and all $F^{ij}_{\alpha\beta}(\cdot, 0)$
 belong to $C^\infty(\overline{\Omega})$
 because of using Proposition~\ref{prop:BifE.10}. In the present
 case we only need to use Proposition~\ref{prop:BifE.10*}, and
(\ref{e:BifE.18}) is implied in the {\bf CGC} by the proof of
Corollary~\ref{cor:BifE.13}.
\end{proof}

Clearly, Theorem~\ref{th:BifE.12} and the case of $m=N=1$ in the following Theorem~\ref{th:BifE.13}  are generalizations of a recent result for
semilinear elliptic Dirichlet problems on a ball
 in \cite{PoWa} in different direction.

 By increasing smoothness of $F$ and $\partial\Omega$, in a
 Banach space of functions with higher smoothness we can use
\cite[Theorem~6.1]{Lu8} and Proposition~\ref{prop:BifE.10}
to obtain a similar result to Rabinowitz bifurcation theorem \cite{Rab}
with neither the assumption  \textsf{Hypothesis} $\mathfrak{F}_{2,N,m,n}$
for $F$ nor the requirement that $\Omega\subset\mathbb{R}^n$ is  star-shaped.
Here is the third of main results in this section.

 \begin{theorem}\label{th:BifE.13}
  Let $N\in\mathbb{N}$, a real $p\ge 2$, integers $k$ and $m$
  satisfy $k> m+\frac{n}{p}$,  and let
   $\Omega\subset\R^n$  be a star-shaped bounded  domain
with $C^{\infty}$ boundary, and let $\Omega_t$ and
$\varphi_t$ be as in (\ref{e:BifE.star}) for $t\in (0, 1]$.  Let
$$
F:\overline\Omega\times\prod^m_{k=0}\mathbb{R}^{N\times M_0(k)}\to\R
$$
be $C^{k-m+7}$, and $\mathcal{F}:(0,1]\times X_{k,p}\to\mathbb{R}$ be still defined by
the right side of (\ref{e:BifE.13}). Suppose:
\begin{enumerate}
 \item[\rm (I)] (\ref{e:BifE.24-}) is satisfied,   all $F^{ij}_{\alpha\beta}(\cdot, 0)$ belong to $C^\infty(\overline{\Omega})$.
 \item[\rm (II)]  There exists some $c>0$ such that for all $x\in\overline{\Omega}$
and  $\eta=(\eta^{i}_{\alpha})\in\R^{N\times M_0(m)}$,
  \begin{eqnarray*}
\sum^N_{i,j=1}\sum_{|\alpha|=|\beta|=m}F^{ij}_{\alpha\beta}(x, 0)\eta^i_\alpha\eta^j_\beta\ge
c\sum^N_{i=1}\sum_{|\alpha|=m}(\eta^i_\alpha)^2.
\end{eqnarray*}
 \item[\rm (III)]  Either (\ref{e:BifE.24}) have no nontrivial solutions in $C^\infty(\overline{\Omega},\mathbb{R}^N)$ with compact
support in any of the manifolds $\Omega_t$, or (\ref{e:BifE.24})
 has no solutions $\vec{u}\ne 0$ such that $\vec{u}$
vanishes on some open set in $\Omega$. (This implies $\nu_0=0$.)
\end{enumerate}
Then there is only a finite number of $t\in (0, 1]$ with positive nullity
$\nu_t$ (in the sense of \cite[Appendix~B]{Lu8}).
 Moreover,  $(t_0, 0)\in (0, 1]\times X_{k,p}$ is a bifurcation point  for
(\ref{e:BifE.10}) if and only if $t_0$ is a conjugate point, and in this case
  one of the following alternatives occurs:
\begin{enumerate}
\item[\rm (i)] $(t_0, 0)$ is not an isolated solution of (\ref{e:BifE.10}) in
 $\{t_0\}\times X_{k,p}$;

\item[\rm (ii)]  for every $t\in (0,1)$ near $t_0$ there is a nontrivial solution
$\vec{u}_t$ of (\ref{e:BifE.10}) converging to $0$ as $t\to t_0$;

\item[\rm (iii)] there is an one-sided  neighborhood $\mathfrak{T}$ of $t_0$ such that
for any $t\in \mathfrak{T}\setminus\{t_0\}$,
(\ref{e:BifE.10}) has at least two nontrivial solutions converging to
zero as $t\to t_0$.
\end{enumerate}

In addition, if $m=N=1$, (III) and the second assumption in (I) are not needed,
and $\partial\Omega$ is only needed to be $C^{1,1}$. (See the proof of
Theorem~\ref{th:BifE.12}.)
 \end{theorem}
\begin{proof}
Since $F$ is $C^{k-m+7}$, so is ${\bf F}:\overline\Omega\times\prod^m_{k=0}\mathbb{R}^{N\times M_0(k)}\times(0, 1]\to\R$
defined by  (\ref{e:BifE.14}). It follows from Proposition~\ref{prop:Jia1} that
  the corresponding map $A:(0, 1]\times X_{k,p}\to  X_{k,p}$ defined as in (\ref{e:Jia2})
is $C^3$. Moreover, for all $\eta=(\eta^{i}_{\alpha})\in\R^{N\times M_0(m)}$, the condition (II) and (\ref{e:BifE.16})  lead to
  \begin{eqnarray*}
 \sum^N_{i,j=1}\sum_{|\alpha|=|\beta|=m}{\bf F}^{ij}_{\alpha\beta}(x, 0;t))\eta^i_\alpha\eta^j_\beta&=&
  \sum^N_{i,j=1}\sum_{|\alpha|=|\beta|=m}t^{n-|\alpha|-|\beta|}F^{ij}_{\alpha\beta}(tx, 0)\eta^i_\alpha\eta^j_\beta\\
&\ge&ct^{n}\sum^N_{i=1}\sum_{|\alpha|=m}(\eta^i_\alpha)^2.
\end{eqnarray*}
As in the proof of Theorem~\ref{th:Jia1} we derive from  Proposition~\ref{prop:Jia2} that
the Morse index $\mu_t$ and the nullity $\nu_t$ of  $\mathcal{F}_{t}$ at $0\in X_{k,p}$
(in the sense of \cite[Appendix~B]{Lu8}) are, respectively,
equal to the Morse index $m_t$ and the nullity $n_t$ of the quadratic form $(B_t(0)\vec{u},\vec{u})_H$ on $H$.
But $m_t$ and $n_t$ are equal to the Morse index  and the nullity  of
the symmetric bilinear form $B_{L_t}$ on $W^{m,2}_0(\Omega_t,\mathbb{R}^N)$  given by
(\ref{e:BifE-bilinear})  with $A^{ij}_{\alpha\beta}(x)=F^{ij}_{\alpha\beta}(x,0)$, respectively.
Because of these and the assumptions (I) and (III) we may use Proposition~\ref{prop:BifE.10} to get
\begin{eqnarray*}
\mu_0-\mu_s=\sum_{0< t\le s}\nu_t\quad\forall s\in (0, 1],
\end{eqnarray*}
and $\mu_1=0$ if $\Omega_1$ has sufficiently small volume.
The first claim follows directly.

Therefore  we may choose   $0<\epsilon<\min\{t_0,1-t_0\}$ such that $T:=[t_0-\epsilon, t_0+\epsilon]$
 contains no conjugate points other than $t_0$, and hence
$$
\mu_t=\left\{\begin{array}{ll}
\mu_{t_0-\delta},&\quad\forall t\in [t_0-\delta,t_0),\\
\mu_{t_0-\delta}+\nu_{t_0},&\quad\forall t\in (t_0, t_0+\delta].
\end{array}\right.
$$
By Theorem~\ref{th:Jia1} (replacing ${\bf D}^1$ with $T$) the second claim is proved.

For the final part, we only need to replace Proposition~\ref{prop:BifE.10} by Proposition~\ref{prop:BifE.10*}
in the proof of the above first claim.
\end{proof}

\begin{remark}\label{rm:BifE.14}
{\rm Under the assumptions of Theorem~\ref{th:BifE.13}, let $G$ be a compact Lie group acting on $H$ orthogonally,
which induces a $C^1$ isometric action on $X_{k,p}$. Suppose that each functional $\mathcal{F}_t$ is $G$-invariant and that
maps $A_t, B_t$  are equivariant. Then we can use Theorem~\ref{th:Jia2} to improve
 the second conclusion of Theorem~\ref{th:BifE.13}.
Similarly, we have a corresponding version of Theorem~\ref{th:BifE.11}
with Theorem~\ref{th:Jia2}.}
\end{remark}

\appendix
\section{Appendix:\quad
 Differentiability of composition mappings}\label{app:A}\setcounter{equation}{0}

We recall notations and notions in \cite{Ir72}.
Let ${\bf E}$, ${\bf F}$ and ${\bf G}$ denote normed linear spaces.
For a subset $X\subset{\bf E}$ such that $X\subset\overline{{\rm int}(X)}$,
a map $f:X\to{\bf F}$ is called $C^r$ in \cite{Ir72} iff $f$ and its first $r$ (Fr\'echet) derivatives
(defined via an extension of $f$ to a neighbourhood of $X$ in ${\bf E}$)
exist and are continuous on $X$.
Let $C^r_b(X, {\bf F})$ be the normed linear space of
those $C^r$ maps $f:X\to{\bf F}$ (called $C^r$-{\bf bounded} maps) for which the norm
$$
\|f\|_{r,b}=\sup\{|D^i f(x)|\,|\, x\in X,\;0\le i\le r\}
$$
is finite. If $g\in C^r_b({\bf F}, {\bf G})$ and
$Df\in C^{r-1}_b(X, \mathcal{L}({\bf E}, {\bf F}))$
then $g\circ f\in C^r_b(X, {\bf G})$. Moreover, if
$Dg$ and $Df$ are $C^r$-bounded then so is $D(g\circ f)$.
For $Y\subset{\bf F}$, let $C^r_b(X,Y)=\{f\in C^r_b(X,{\bf F})\,|\, f(X)\subset Y\}$.
 When $C^r_b(X,Y)$ is itself to be the domain of a $C^1$ map, it was also
required that $C^r_b(X,Y)\subset\overline{{\rm int}C^r_b(X,Y)}$.
Clearly, any open subset $X\subset{\bf E}$ satisfies $X\subset\overline{{\rm int}(X)}$,
and in this case $C^r_b(X, {\bf F})$ is exactly
the space of $C^r$ maps from $X$ to ${\bf F}$
whose first $r$ derivatives are bounded on $X$.

\begin{theorem}[\hbox{\cite[Theorem~14]{Ir72}}]\label{th:Ir}
Let ${\bf F}$, and, in case $r>0$, also ${\bf E}$, be finite-dimensional. Then %the map
$$
\mathfrak{{ comp}}:C^{r+s}_b(Y,{\bf G})\times C^r_b(X,Y)\to C^r_b(X,{\bf G}),\;(f,g)\mapsto f\circ g
$$
is $C^s$.
\end{theorem}

\begin{corollary}\label{cor:Ir}
Under the assumptions of Theorem~\ref{th:Ir}
let ${\bf H}$ be a finite-dimensional normed linear space. Suppose that $Y$ is bounded and that $Z\subset {\bf H}$
satisfies $Z\times Y\subset\overline{{\rm int}(Z\times Y)}$. Then for any
$F\in C^{r+s}_b(Z\times Y,{\bf G})$, the map
\begin{equation}\label{e:Ir}
Z\ni z\mapsto F(z,\cdot)\in C^{r}_b(Y,{\bf G})
\end{equation}
is $C^s$.
\end{corollary}

\begin{proof}
For each $z\in Z$, define a map $g_z:Y\to Z\times Y$ by $g_z(y)=(z,y)$.
Then for any $q\in\mathbb{N}\cup\{0\}$,
since $Y$ has a compact closure,
$g_z\in C^{q}_b(Y, Z\times Y)$ and the map
$$
\mathfrak{G}:Z\to C^{q}_b(Y, Z\times Y),\;z\mapsto g_z
$$
is smooth. By Theorem~\ref{th:Ir} we have a $C^s$ map
$$
\mathfrak{{ comp}}:C^{r+s}_b(Z\times Y,{\bf G})\times C^r_b(Y, Z\times Y)\to
C^r_b(Y,{\bf G}),\;(f,g)\mapsto f\circ g.
$$
Thus the composition of this and the smooth map
$$
Z\to C^{r+s}_b(Z\times Y,{\bf G})\times C^r_b(Y, Z\times Y),\;
z\mapsto (F, \mathfrak{G}(z))
$$
is also $C^s$. Obverse that
$\mathfrak{{ comp}}(F, \mathfrak{G}(z))=F\circ \mathfrak{G}(z)=F(z,\cdot)$
for any $z\in Z$. The desired claim is proved.
\end{proof}

Let $\mathcal{U} \subset \mathbb{R}^m$ and $\Omega \subset \mathbb{R}^n$
be open subsets, and let $\mathbf{X}(\mathcal{U},\mathbb{R}^n)$
denote some Banach space of maps $\mathcal{U} \to \mathbb{R}^n$
that admits a continuous inclusion into
$C^0_b(\mathcal{U},\mathbb{R}^n)$. Then
\begin{equation}\label{e:restrict}
\mathbf{X}(\mathcal{U},\Omega):= \{ u \in \mathbf{X}(\mathcal{U},\mathbb{R}^n)\ |\
\overline{u(\mathcal{U})}\subset \Omega \}
\end{equation}
is an open subset of $\mathbf{X}(\mathcal{U},\mathbb{R}^n)$
due to the continuous inclusion assumption.

We have the following convenient special version of \cite[Lemma~4.1]{Eliasson},
which was proved in \cite[Lemma~2.96]{Wend} for convex $\Omega$.
See \cite{PT01} for a proof without the convexity assumption.

\begin{lemma}[\hbox{\cite[Lemma~2.96]{Wend}}]
\label{lemma:babySmoothness}
Suppose $\mathcal{U} \subset \mathbb{R}^m$ denotes an open subset, and the symbol
$\mathbf{X}$ associates to any Euclidean space $\mathbb{R}^N$ a Banach space
$\mathbf{X}(\mathcal{U},\mathbb{R}^N)$ consisting of bounded continuous
maps $\mathcal{U} \to \mathbb{R}^N$ such that the
following hypotheses are satisfied:
\begin{itemize}
\item \textsc{($C^0$-inclusion)} The inclusion
$\mathbf{X}(\mathcal{U},\mathbb{R}^N) \hookrightarrow C^0_b(\mathcal{U},\mathbb{R}^N)$
is continuous.
\item \textsc{(Banach algebra)} The natural bilinear pairing
$$
\mathbf{X}(\mathcal{U},\mathcal{L}(\mathbb{R}^n,\mathbb{R}^N)) \times \mathbf{X}(\mathcal{U},\mathbb{R}^n) \to
\mathbf{X}(\mathcal{U},\mathbb{R}^N) : (A,u) \mapsto Au
$$
is well defined and continuous.
\item \textsc{($C^k$-continuity)}
For some integer $k \ge 0$, if $\Omega \subset \mathbb{R}^n$ is any open set and
$f \in C^k(\Omega,\mathbb{R}^N)$, the map
\begin{equation}
\label{eqn:Phif}
\Phi_f : \mathbf{X}(\mathcal{U},\Omega) \to \mathbf{X}(\mathcal{U},\mathbb{R}^N) : u \mapsto f \circ u
\end{equation}
is well defined and continuous.
\end{itemize}
If $f \in C^{k+r}(\Omega,\mathbb{R}^N)$ for some $r \in \mathbb{N}$, then the map
$\Phi_f$ defined in \eqref{eqn:Phif} is of class $C^r$ and has derivative
\begin{equation}
\label{eqn:theDerivative}
d\Phi_f(u)\eta = (df \circ u) \eta.
\end{equation}
\end{lemma}

By the $C^0$-inclusion hypothesis, each
 $u_0\in\mathbf{X}(\mathcal{U},\mathbb{R}^N)$ has a neighborhood $\mathscr{V}(u_0)$
 such that \linebreak $\cup_{u\in\mathscr{V}(u_0)}\overline{u(\mathcal{U})}$ is contained in
 a compact subset of $\Omega$. Hence it is unnecessary to assume
$f \in C^k_b(\Omega,\mathbb{R}^N)$ (resp. $f \in C^{k+r}_b(\Omega,\mathbb{R}^N)$)
in Lemma~\ref{lemma:babySmoothness}.

\begin{lemma}[\hbox{\cite[Lemma~2.98]{Wend}}]
\label{lemma:smoothness}
Suppose $\mathcal{U}$, $\Omega$ and $\mathbf{X}(\mathcal{U},\mathbb{R}^n)$ are as in
Lemma~\ref{lemma:babySmoothness}, and in addition that the pairing
$T(u) f := f \circ u$ defines $T$ as a continuous map
\begin{equation}
\label{eqn:T2}
T : \mathbf{X}(\mathcal{U},\Omega) \to \mathcal{L}(C^k_b(\Omega,\mathbb{R}^N), \mathbf{X}(\mathcal{U},\mathbb{R}^N) ).
\end{equation}
Then for any $r \in \mathbb{N}$, the map
$$
\Psi : C^{k+r}_b(\Omega,\mathbb{R}^N) \times \mathbf{X}(\mathcal{U},\Omega) \to
\mathbf{X}(\mathcal{U},\mathbb{R}^N) : (f,u) \mapsto f \circ u
$$
is of class $C^r$ and has derivative
$$
d\Psi(f,u)(g,\eta) = g \circ u + (df \circ u) \eta.
$$
\end{lemma}

Suppose that $\mathcal{U}$ satisfies the cone condition.
By the Sobolev imbedding theorem (\cite[Theorem~4.12]{AdFo}),  if either $kp>m$ or $k=m$ and $p=1$ then
the inclusion $W^{k,p}(\mathcal{U},\mathbb{R}^N) \hookrightarrow C^0_b(\mathcal{U},\mathbb{R}^N)$ is continuous.
In particular,  there exists, for any multi-index  $\beta\in (\mathbb{N}\cup\{0\})^m$
of length $|\beta|\le k$, a constant $C_\beta=C(\beta,m,k,p,V)$, where $V$ is the finite cone determining the cone condition for $\mathcal{U}$, such that
\begin{equation}\label{e:Sobolev1}
\int_{\mathcal{U}}|D^\beta w(x)|^r\le C_\beta\|w\|^r_{k,p}\quad\forall w\in W^{k,p}(\mathcal{U})
\end{equation}
provided $(k-|\beta|)p\le m$ and $p\le r\le mp/(m-[k-|\beta|]p)$ or $(k-|\beta|)p=m$ and $r\in [p, \infty)$, and that
\begin{equation}\label{e:Sobolev2}
|D^\beta w(x)|\le C_\beta\|w\|_{k,p}\quad\hbox{\rm a.e. in $\mathcal{U}$}
\end{equation}
for any $w\in W^{k,p}(\mathcal{U})$ provided $(k-|\beta|)p>m$. Denote by
\begin{equation}\label{e:Sobolev3}
C(m,k,p,V)=\max\{C(\beta,m,k,p,V)\,|\, \beta\in (\mathbb{N}\cup\{0\})^m\;\&\;|\beta|\le k\}.
\end{equation}

\begin{lemma}\label{lem:BanAlgebra}
Suppose that $\mathcal{U} \subset \mathbb{R}^m$ is a nonempty open subset
 satisfying the cone condition, and that $kp>m$. %either $kp>m$ % or $k\ge m$ and $p=1$.
For positive constants $C_l(m,k,p,V)$ inductively defined by
\begin{eqnarray*}
C_{l+1}(m,k,p,V)=(C(m,k,p,V))^{\frac{1}{p}}\max\{C_l(m,k,p,V),1\}
\end{eqnarray*}
with $C_1(m,k,p,V)=1$,  $l=1,2,\cdots$, it holds that
 \begin{equation}\label{e:BanachAlge2}
\|D^{\alpha^1}w_1D^{\alpha^2}w_2\cdots D^{\alpha^l}w_l\|_{p}\le C_l(m,k,p,V)\prod^l_{j=1}\|w_j\|_{k,p}
\end{equation}
for any $w_j\in W^{k,p}(\mathcal{U}),\;j=1,\cdots,l$, and
 multi-indexes  $\alpha^j\in (\mathbb{N}\cup\{0\})^m$, $j=1,\cdots,l$,
 whose sum $\alpha:=\alpha^1+\cdots+\alpha^l$ has length $\le k$.
 \end{lemma}
\begin{proof}
(\ref{e:BanachAlge2}) clearly holds if $l=1$ and $C_1(m,k,p,V)=1$. The case $l=2$ was proved in the proof of \cite[Theorem~3.17]{AdFo}. Their proof can be generalized to the general case by induction.

Suppose that (\ref{e:BanachAlge2}) holds with some $l\ge 1$.
Consider  any $w_j\in W^{k,p}(\mathcal{U}),\;j=1,\cdots,l+1$, and
 multi-indexes  $\alpha^j\in (\mathbb{N}\cup\{0\})^m$, $j=1,\cdots,l+1$,
 whose sum $\alpha:=\alpha^1+\cdots+\alpha^{l+1}$ has length $\le k$.
Let $s$ be the largest integer such that $(k-s)p>m$. Then $s\ge 0$ because $kp>m$.

Suppose that $|\alpha^j|\le s$ for some $j\in\{1,\cdots,l+1\}$.
Without loss of generality we can assume $j=1$.
Then $(k-|\alpha^1|)p>m$, and so by (\ref{e:Sobolev2}) we get
\begin{eqnarray*}
&&\int_{\mathcal{U}}|D^{\alpha^1}w_1(x)D^{\alpha^2}w_2(x)\cdots D^{\alpha^l}w_l(x)D^{\alpha^{l+1}}w_{l+1}(x)|^pdx\\
&\le& C_{\alpha^1}\|w_1\|^p_{k,p}\int_{\mathcal{U}}|D^{\alpha^2}w_2(x)\cdots D^{\alpha^l}w_l(x)D^{\alpha^{l+1}}w_{l+1}(x)|^pdx
\\
&\le&C_{\alpha^1}\|w_1\|^p_{k,p}(C_l(m,k,p,V)\prod^{l+1}_{j=2}\|w_j\|_{k,p})^p
\qquad\hbox{(by induction hypothesis)}\\
&\le&(C'_{l+1}(m,k,p,V)\prod^{l+1}_{j=1}\|w_j\|_{k,p})^p
\end{eqnarray*}
where $C'_{l+1}(m,k,p,V)=C_l(m,k,p,V) (C(m,k,p,V))^{\frac{1}{p}}$.

Suppose  $|\alpha^j|>s$ for all $j\in\{1,\cdots,l+1\}$. Then $m\ge (k-|\alpha^j|)p$, $j=1,\cdots, l+1$, and so
\begin{eqnarray*}
\sum^{l+1}_{j=1}\frac{m-(k-|\alpha^j|)p }{m}&=&l+1-\frac{(k(l+1)-\sum^{l+1}_{j=1}|\alpha^j|)p}{m}\\
&=&l+1-\frac{(k(l+1)-|\alpha|)p}{m}\\
&=&l+1-\frac{k(l+1)p}{m}+\frac{|\alpha|p}{m}\\
&\le& l+1-\frac{k(l+1)p}{m}+\frac{kp}{m}\\
&=&1+l-\frac{lkp}{m}\\
&=&1+l(1-\frac{kp}{m})<1.
\end{eqnarray*}
Let $\varepsilon>0$ be determined by $1-(l+1)\varepsilon=\sum^{l+1}_{j=1}\frac{m-(k-|\alpha^j|)p }{m}$.
Put
$$
r_j=\left(\varepsilon+ \frac{m-(k-|\alpha^j|)p }{m}\right)^{-1},\quad j=1,\cdots,l+1.
$$
 Then
$$
\sum^{l+1}_{j=1}\frac{1}{r_j}=1\quad\hbox{and}\quad 1\le r_j\le \frac{m}{m-(k-|\alpha^j|)p}, \quad j=1,\cdots, l+1.
$$
It follows from H\"older inequality and (\ref{e:Sobolev1}) that \allowdisplaybreaks
\begin{eqnarray*}
&&\int_{\mathcal{U}}|D^{\alpha^1}w_1(x)D^{\alpha^2}w_2(x)\cdots D^{\alpha^l}w_l(x)D^{\alpha^{l+1}}w_{l+1}(x)|^pdx\\
&\le&\prod^{l+1}_{j=1}\left(\int_{\mathcal{U}}|D^{\alpha^j}w_j(x)|^{r_jp}dx\right)^{1/r_j}\\
&\le& \prod^{l+1}_{j=1}(C_{\alpha^j})^{1/r_j}\prod^{l+1}_{j=1}\|w_j\|^p_{k,p}\\
&\le& C(m,k,p,V)\prod^{l+1}_{j=1}\|w_j\|^p_{k,p}.
\end{eqnarray*}
Hence we can take
\begin{eqnarray*}
C_{l+1}(m,k,p,V)&=&\max\{C'_{l+1}(m,k,p,V), (C(m,k,p,V))^{1/p}\}\\
&=&\max\{C_l(m,k,p,V)(C(m,k,p,V))^{\frac{1}{p}} , (C(m,k,p,V))^{1/p}\}\\
&=&(C(m,k,p,V))^{\frac{1}{p}}\max\{C_l(m,k,p,V),1\}
\end{eqnarray*}
with $C_1(m,k,p,V)=1$.
\end{proof}

By the definition in (\ref{e:restrict}),
$W^{k,p}(\mathcal{U},\Omega)= \{ u \in W^{k,p}(\mathcal{U},\mathbb{R}^n)\ |\
\overline{u(\mathcal{U})}\subset \Omega \}$. Fix a $u\in W^{k,p}(\mathcal{U},\Omega)$.
Then $\overline{u(\mathcal{U})}$ is compact and is contained in $\Omega$.

If $1\le p<\infty$, by \cite[Theorem~3.17]{AdFo}, $W^{k,p}(\mathcal{U})$
is equal to the completion $H^{k,p}(\mathcal{U})$ of $\{v\in C^k(\mathcal{U})\,|\,\|v\|_{k,p}<\infty\}$
with respect to the norm $\|\cdot\|_{k,p}$. Actually, it was proved in
\cite[Theorem~3.17]{AdFo} that $\{v\in C^\infty(\mathcal{U})\,|\,\|v\|_{k,p}<\infty\}$
is dense in $W^{k,p}(\mathcal{U})$.

For a multi-index $\alpha\in(\mathbb{N}\cup\{0\})^m$,
a decomposition of it was defined in \cite{Ma} as a  list $(s,\sigma,\omega)$ consisting of \\
$\bullet$ a positive integer $s$;\\
$\bullet$ $\sigma=\{\sigma_1,\cdots,\sigma_s\}$, where $\sigma_j=(\sigma_{j1},\cdots,\sigma_{jm})\in(\mathbb{N}\cup\{0\})^m$,
$j=1,\cdots,s$, satisfy $\sigma_1\ll\sigma_2\ll\cdots\ll\sigma_s$,
where $\sigma_j\ll\sigma_{j+1}$ means that there is $i\in\{1,\cdots,m\}$ such that
$$
\sigma_{j1}=\sigma_{(j+1)1},\cdots, \sigma_{j(i-1)}=\sigma_{(j+1)(i-1)}\quad\hbox{but}\quad \sigma_{ji}<\sigma_{(j+1)i};
$$
$\bullet$ $\omega=\omega_1+\cdots+\omega_s$, where $\omega_j=(\omega_{j1},\cdots,\omega_{jn})\in(\mathbb{N}\cup\{0\})^n$, $j=1,\cdots,s$, satisfying
$\alpha=|\omega_1|\sigma_1+\cdots+|\omega_s|\sigma_s$.

Let $r_i=\omega_{1i}+\cdots+\omega_{si}$, $i=1,\cdots,n$. Then $\omega=(r_1,\cdots,r_n)$.
Denote by $\mathscr{D}(\alpha)$  the set of all decompositions of $\alpha$.
By the end of \cite{Ma},  $\sharp\mathscr{D}(\alpha)\le |\alpha|(1+|\alpha|)^{m+n}$.

The author cannot find the following theorem  in literatures though their Corollaries~\ref{cor:babySmoothness},~\ref{cor:smoothness}
in case $m=2$ were used in \cite{Wend}. There exists a similar result to Corollary~\ref{cor:smoothness}, \cite[Theorem~1.2]{InKaTo},
which was proved for the case in which the domains of mappings are closed manifolds. As stated in \cite{InKaTo}, most of various versions of
this kind of results hardly have proofs.

\begin{theorem}
\label{th:babySmoothness}
Let $f\in C^k(\Omega,\mathbb{R}^N)$. Suppose that
$kp>m$ and $v\in W^{k,p}(\mathcal{U},\Omega)$ (and thus $\overline{v(\mathcal{U})}$ is compact and is contained in $\Omega$).
Fix  a compact neighborhood $K$ of $\overline{v(\mathcal{U})}$ contained in $\Omega$
(and so $v$ has a neighborhood $\mathscr{N}(v)$ in $W^{k,p}(\mathcal{U},\Omega)$
such that $\overline{u(\mathcal{U})}\subset K$ for all $u\in\mathscr{N}(v)$).
Then
\begin{eqnarray*}
\|f\circ u\|_{k,p}&:=&\sum^N_{\nu=1}\|f_\nu\circ u\|_{k,p}
=\sum^N_{\nu=1}\sum_{|\alpha|\le k}
\left\|\frac{\partial^{|\alpha|}(f_\nu\circ u)}{\partial x^\alpha}\right\|_p\\
&&\hspace{-1cm}\le\sum^N_{\nu=1}\sum_{|\alpha|\le k}\alpha!
\sum_{(s,\sigma, \omega)\in\mathscr{D}(\alpha)}C_{l_\omega}(m,k,p,V)\sup_{y\in K}|D^\omega f_\nu|\prod^s_{j=1}\prod^n_{i=1}(\left\|u_i\right\|_{k,p})^{\omega_{ji}}
\end{eqnarray*}
for any $u\in\mathscr{N}(v)$, where $l_\omega=\sum^s_{j=1}\sum^n_{i=1}\omega_{ji}$.
\end{theorem}
\begin{proof}
Since $\{w\in C^k(\mathcal{U},\Omega)\,|\,\|w\|_{k,p}<\infty\}$
is dense in $W^{k,p}(\mathcal{U},\Omega)$, we can assume $u\in C^k(\mathcal{U},\Omega)$.
Let $|\alpha|\le k$.
By the higher chain formula in \cite{Ma}, for any $x\in\mathcal{U}$ we get
{\small \begin{eqnarray}\label{e:Hchain}
&&\frac{\partial^{|\alpha|}(f_\nu\circ u)}{\partial x^\alpha}(x)\nonumber\\
&&\hspace{-.7cm}=\alpha!\sum_{(s,\sigma, \omega)\in\mathscr{D}(\alpha)}
\frac{\partial^{r_1+\cdots+r_n}f_\nu}{\partial y_1^{r_1}\cdots\partial y_n^{r_n}}(u(x))
\prod^s_{j=1}\prod^n_{i=1}\frac{1}{\omega_{ji}!}\left[\frac{1}{\sigma_{j1}!\cdots\sigma_{jm}!}
\frac{\partial^{\sigma_{j1}+\cdots+\sigma_{jm}}u_i}{\partial x_1^{\sigma_{j1}}\cdots\partial x_m^{\sigma_{jm}}}(x)\right]^{\omega_{ji}}\nonumber\\
&&\hspace{-.7cm}=\alpha!\sum_{(s,\sigma, \omega)\in\mathscr{D}(\alpha)}
\frac{\partial^{r_1+\cdots+r_n}f_\nu}{\partial y_1^{r_1}\cdots\partial y_n^{r_n}}(u(x))
\prod^s_{j=1}\prod^n_{i=1}\frac{1}{\omega_{ji}!}\left[\frac{1}{\sigma_{j1}!\cdots\sigma_{jm}!}
D^{\sigma_j}u_i(x)\right]^{\omega_{ji}}.
\end{eqnarray}}
Note that $|\omega|=|\omega_1|+\cdots+|\omega_s|\le|\alpha|$, and
$$
\sup_{x\in\mathcal{U}}\left|\frac{\partial^{r_1+\cdots+r_n}f_\nu}{\partial y_1^{r_1}\cdots\partial y_n^{r_n}}(u(x))\right|\le
\sup_{y\in K}|D^\omega f_\nu(y)|.
$$
Then
\begin{eqnarray*}
&&\left\|\frac{\partial^{|\alpha|}(f_\nu\circ u)}{\partial x^\alpha}\right\|_p\\&\le&\alpha!\sum_{(s,\sigma, \omega)\in\mathscr{D}(\alpha)}
\sup_{y\in K}|D^\omega f_\nu(y)|\left\|\prod^s_{j=1}\prod^n_{i=1}\frac{1}{\omega_{ji}!}
\left[\frac{1}{\sigma_{j1}!\cdots\sigma_{jm}!}
D^{\sigma_j}u_i\right]^{\omega_{ji}}\right\|_p\\
&\le&\alpha!\sum_{(s,\sigma, \omega)\in\mathscr{D}(\alpha)}\sup_{y\in K}|D^\omega f_\nu(y)|
\left\|\prod^s_{j=1}\prod^n_{i=1}\left[D^{\sigma_j}u_i\right]^{\omega_{ji}}\right\|_p.
\end{eqnarray*}
Recall that $l_\omega=\sum^s_{j=1}\sum^n_{i=1}\omega_{ji}$. By Lemma~\ref{lem:BanAlgebra}
$$
\left\|\prod^s_{j=1}\prod^n_{i=1}\left[D^{\sigma_j}u_i\right]^{\omega_{ji}}\right\|_p\le  C_{l_\omega}(m,k,p,V)
\prod^s_{j=1}\prod^n_{i=1}(\left\|u_i\right\|_{k,p})^{\omega_{ji}}
$$
and hence
\begin{eqnarray*}
&&\left\|\frac{\partial^{|\alpha|}(f_\nu\circ u)}{\partial x^\alpha}\right\|_p\\&\le&\alpha!
\sum_{(s,\sigma, \omega)\in\mathscr{D}(\alpha)}C_{l_\omega}(m,k,p,V)
\sup_{y\in K}|D^\omega f_\nu(y)|
\prod^s_{j=1}\prod^n_{i=1}(\left\|u_i\right\|_{k,p})^{\omega_{ji}}.
\end{eqnarray*}
The desired inequality follows.
\end{proof}

\begin{corollary}\label{cor:babySmoothness}
If $f\in C^k(\Omega,\mathbb{R}^N)$ and $kp>m$ then the map
\begin{equation}
\label{eqn:Phif+}
\Phi_f : W^{k,p}(\mathcal{U},\Omega) \to W^{k,p}(\mathcal{U},\mathbb{R}^N):
u\mapsto f \circ u
\end{equation}
is well defined and continuous.
Moreover, if $f \in C^{k+r}(\Omega,\mathbb{R}^N)$ for some $r \in \mathbb{N}$,
then the map $\Phi_f$  is of class $C^r$ and has derivative
$$
d\Phi_f(u)\eta = (df \circ u) \eta.
$$
\end{corollary}
\begin{proof}
For a fixed $v\in W^{k,p}(\mathcal{U},\Omega)$ let $\mathscr{N}(v)$ be as in
Theorem~\ref{th:babySmoothness}. By the latter  (or its proof ) we have
\begin{eqnarray*}
&&\|f\circ u-f\circ v\|_{k,p}\\&&\le\sum^N_{t=1}\sum_{|\alpha|\le k}\alpha!
\sum_{(s,\sigma, \omega)\in\mathscr{D}(\alpha)}C_{l_\omega}(m,k,p,V)
\sup_{y\in K}|D^\omega f_t|
\prod^s_{j=1}\prod^n_{i=1}(\left\|u_i-v_i\right\|_{k,p})^{\omega_{ji}}
\end{eqnarray*}
for any $u\in\mathscr{N}(v)$. Hence $\Phi_f$ is continuous at $v$.
This and Lemma~\ref{lemma:babySmoothness}  lead to other claims.
\end{proof}

\begin{corollary}\label{cor:smoothness}
Let $kp>m$.  Then the map
\begin{equation}
\label{eqn:T2+}
T : W^{k,p}(\mathcal{U},\Omega) \to \mathcal{L}(C^k_b(\Omega,\mathbb{R}^N), W^{k,p}(\mathcal{U},\mathbb{R}^N) ).
\end{equation}
defined by $T(u) f := f \circ u$ is continuous.
Consequently, (by Lemma~\ref{lemma:smoothness}) for any $r \in \mathbb{N}$, the map
$$
\Psi : C^{k+r}_b(\Omega,\mathbb{R}^N) \times W^{k,p}(\mathcal{U},\Omega) \to
W^{k,p}(\mathcal{U},\mathbb{R}^N) : (f,u) \mapsto f \circ u
$$
is of class $C^r$ and has derivative
$$
d\Psi(f,u)(g,\eta) = g \circ u + (df \circ u) \eta.
$$
\end{corollary}
\begin{proof}
For $u,v\in W^{k,p}(\mathcal{U},\Omega)$, by Theorem~\ref{th:babySmoothness} we have
\begin{eqnarray*}
&&\|(T(u)-T(v))f\|_{k,p}=\|f\circ (u-v)\|_{k,p}\\
&\le&\sum^N_{t=1}\sum_{|\alpha|\le k}\alpha!
\sum_{(s,\sigma, \omega)\in\mathscr{D}(\alpha)}C_{l_\omega}(m,k,p,V)\sup_{y\in \Omega}|D^\omega f_t|
\prod^s_{j=1}\prod^n_{i=1}(\left\|u_i-v_i\right\|_{k,p})^{\omega_{ji}}\\
&\le&\|f\|_{k,b}\sum^N_{t=1}\sum_{|\alpha|\le k}\alpha!
\sum_{(s,\sigma, \omega)\in\mathscr{D}(\alpha)}C_{l_\omega}(m,k,p,V)
\prod^s_{j=1}\prod^n_{i=1}(\left\|u_i-v_i\right\|_{k,p})^{\omega_{ji}}
\end{eqnarray*}
 where $l_\omega=\sum^s_{j=1}\sum^n_{i=1}\omega_{ji}$, and hence
 \begin{eqnarray*}
&&\|T(u)-T(v)\|_{\mathcal{L}(E,F)}\\&&\le\sum^N_{t=1}\sum_{|\alpha|\le k}\alpha!
\sum_{(s,\sigma, \omega)\in\mathscr{D}(\alpha)}C_{l_\omega}(m,k,p,V)
\prod^s_{j=1}\prod^n_{i=1}(\left\|u_i-v_i\right\|_{k,p})^{\omega_{ji}}
\end{eqnarray*}
with $E=C^k_b(\Omega,\mathbb{R}^N)$ and $F=W^{k,p}(\mathcal{U},\mathbb{R}^N)$.
\end{proof}

In order to use the Sobolev imbedding theorem we require that
$\mathcal{U}$ satisfies the cone condition. If $\mathcal{U}$ is bounded
this condition may be replaced by $C^1$-smoothness of $\partial\mathcal{U}$.

\begin{proof}[Proof of Proposition~\ref{prop:Jia1}]
%\noindent{\bf Proof of Proposition~\ref{prop:Jia1}}.\quad
Let $\kappa=\dim(\prod^m_{k=0}\mathbb{R}^{N\times M_0(k)})= (m+1)N+
  \sum^m_{k=0}M_0(k)$. For $\vec{u}\in X_{k,p}$ and $x\in\overline{\Omega}$
  put ${\bf u}(x)=(x, \vec{u}(x),\cdots, D^m \vec{u}(x))$. Then
  $$
  \Upsilon:X_{k,p}\to W^{k-m, p}(\Omega, \mathbb{R}^{n+\kappa}),\;\vec{u}\mapsto{\bf u}
  $$
  is  a smooth map.
   ({\it Note}: $\Upsilon(\vec{u})$ does not belong to
    $W^{k-m, p}(\Omega, \Omega\times\mathbb{R}^\kappa)$
     according to (\ref{e:restrict}).)

We  prove the second claim. The first is easily seen from the following proof.
For $\rho>0$ let $B_\rho(X_{k,p})=\{\vec{u}\in X_{k,p}\,|\, \|\vec{u}\|_{k,p}<\rho\}$.
It suffices to prove that
$A$ is $C^r$ in $Z\times B_\rho(X_{k,p})$.

Since  $X_{k,p}\subset C^m(\overline{\Omega}, \mathbb{R}^N)$,
we have $\Gamma_\rho>0$ such that
\begin{eqnarray}\label{e:App1}
&{\bf u}( \overline{\Omega})\subset \overline{\Omega}\times B^\kappa(0, \Gamma_\rho/2)=\overline{\Omega}\times \{\omega\in\mathbb{R}^\kappa\,|\, |\omega|<\Gamma_\rho/2\},\notag\\
&\quad\forall \vec{u}\in B_\rho(X_{k,p}).
\end{eqnarray}
Let $\Omega_\delta$ denote the $\delta$-neighborhood of $\Omega$, i.e.,
$\Omega_\delta=\{x\in\mathbb{R}^n\,|\, {\rm dist}(x,\Omega)<\delta\}$.
For $\delta>0$ small enough,
the restriction of each $C^{k-m+2r}$ function
$\overline\Omega\times\mathbb{R}^\kappa\times Z\ni (x,\xi,\lambda)\mapsto
{\bf F}^i_\alpha(x,\xi;\lambda)\in\mathbb{R}$ to $\overline\Omega\times B^\kappa(0, 2\Gamma_\rho)\times Z$
 can extend a $C^{k-m+2r}$ function on
 $\Omega_{2\delta}\times B^\kappa(0, 2\Gamma_\rho)\times Z$, which
 restricts to a function in
 $C^{k-m+2r}_b(\Omega_\delta\times B^\kappa(0, \Gamma_\rho)\times Z,\mathbb{R})$,
 still denoted by ${\bf F}^i_\alpha$.  By Corollary~\ref{cor:Ir} we get a $C^r$ map
 \begin{equation}\label{e:Ir+}
Z\ni \lambda\mapsto {\bf F}^i_\alpha(\cdot;\lambda)|_{\Omega_\delta\times B^\kappa(0, \Gamma_\rho)}\in C^{k-m+r}_b(\Omega_\delta\times B^\kappa(0, \Gamma_\rho),\mathbb{R}).
\end{equation}

   Since $(k-m)p>n$, by  Corollary~\ref{cor:smoothness}, for any $s\in\N$,
    the composition map
$$
\Psi : C^{k-m+s}_b(\Omega_\delta\times B^\kappa(0, \Gamma_\rho),\mathbb{R})
 \times W^{k-m, p}(\Omega,\Omega_\delta\times
B^\kappa(0, \Gamma_\rho))$$ $$ \to
W^{k-m,p}(\Omega, \mathbb{R}),\; (f,{\bf u}) \mapsto f \circ {\bf u}
$$
is of class $C^s$.    Hence the composition
%$$
%Z\times W^{k-m, p}(\Omega,\Omega_\delta\times
%B^\kappa(0, \Gamma_\rho))$$ $$ \to
%W^{k-m,p}(\Omega, \mathbb{R}) : (\lambda,{\bf u}) \mapsto \Psi({\bf F}^i_\alpha(\cdot;\lambda)|_{\Omega\times B^\kappa(0, \Gamma_\rho)}, {\bf u})
%$$
\begin{eqnarray*}
Z\times W^{k-m, p}(\Omega,\Omega_\delta\times
B^\kappa(0, \Gamma_\rho))&\to&
W^{k-m,p}(\Omega, \mathbb{R}),\\
  (\lambda,{\bf u}) &\mapsto& \Psi({\bf F}^i_\alpha(\cdot;\lambda)|_{\Omega\times B^\kappa(0, \Gamma_\rho)}, {\bf u})
\end{eqnarray*}
 is $C^r$. For any $\vec{u}\in B_\rho(X_{k,p})$, by (\ref{e:App1})
we have
$$
\overline{\Upsilon(\vec{u})(\overline{\Omega})}\in\overline{\Omega}\times \overline{B^\kappa(0, \Gamma_\rho/2)}\subset
 \Omega_\delta\times B^\kappa(0, \Gamma_\rho).
 $$
  Then   according to (\ref{e:restrict}), $\Upsilon(\vec{u})\in W^{k-m, p}(\Omega,\Omega_\delta\times
B^\kappa(0, \Gamma_\rho))$, and so $\Upsilon$ is a smooth map from
$B_\rho(X_{k,p})$ to
$W^{k-m, p}(\Omega,\Omega_\delta\times B^\kappa(0, \Gamma_\rho))$.
 It follows that the map
 $\mathbb{A}_{i,\alpha}:Z\times B_\rho(X_{k,p})\to W^{k-m,p}(\Omega, \mathbb{R})$
  given by
 $$
   \mathbb{A}_{i,\alpha}(\lambda,\vec{u})=\Psi\left(F^i_\alpha(\cdot;\lambda)|_{\Omega\times B^\kappa(0, \Gamma_\rho)}, \Upsilon(\vec{u})\right)
  =F^i_\alpha(\cdot, \vec{u}(\cdot),\cdots, D^m \vec{u}(\cdot);\lambda)
 $$
  is $C^r$. Obverse that
  $$
(A(\lambda,\vec{u}))^i=\triangle^{-m}\sum_{|\alpha|\le m}(-1)^{m+|\alpha|}D^\alpha(\mathbb{A}_{i,\alpha}(\lambda,\vec{u})),\quad i=1,\cdots,N.
$$
 Since $\triangle^{-m}:W^{k-2m,p}(\Omega)\to W^{k,p}(\Omega)\cap W^{m,2}_0(\Omega)$
is a Banach space isomorphism and \linebreak
$D^\alpha: W^{k-m,p}(\Omega)\to W^{k-m-|\alpha|,p}(\Omega)$
is a continuous linear operator,
we deduce that ${A}$ is $C^r$ in $Z\times B_\rho(X_{k,p})$.

For the first claim, since ${\bf F}$ is also $C^{k-m+2r}$,
replacing ${\bf F}^i_\alpha$ by ${\bf F}$ in the above arguments
we can obtain that  the map
 $Z\times B_\rho(X_{k,p})\to W^{k-m,p}(\Omega, \mathbb{R})$  given by
 $$
  (\lambda,\vec{u})\mapsto\Psi({\bf F}(\cdot;\lambda)|_{\Omega\times B^\kappa(0, \Gamma_\rho)}, \Upsilon(\vec{u}))
  ={\bf F}(\cdot, \vec{u}(\cdot),\cdots, D^m \vec{u}(\cdot);\lambda)
 $$
  is $C^r$. The integral $W^{k-m,p}(\Omega, \mathbb{R})\ni w\mapsto\int_\Omega w(x)dx$ is linear and continuous. The claim holds.

Other claims are easily derived from the relation $X_{k,p}\subset C^m(\overline{\Omega}, \mathbb{R}^N)$ as  before.
\end{proof}
  %\qed

\begin{remark}\label{rm:corr}
{\rm According to Remark~B.1 and Theorem~B.5 in
\cite{Lu7}, and Corollary~\ref{cor:babySmoothness}
the function $F$ in \cite[Theorem~4.21]{Lu8} should be required to be $C^{k-m+4}$
so that the corresponding map $\mathbb{A}:X_{k,p}\to X_{k,p}$ is $C^3$.}
\end{remark}

\vspace{0.8cm}

%\vspace{-8pt}
\section{Appendix:\quad
 Proofs of Lemmas~\ref{lem:BifE.10}, \ref{lem:BifE.11}, \ref{lem:BifE.12}}\label{app:B}\setcounter{equation}{0}

\subsection{Proof of Lemma~\ref{lem:BifE.10}}\label{app:B.1}
As in the proof of \cite[(4.8)]{Lu6}, using the mean value theorem
we get a $\tau\in (0,1)$ such that
%\ \vspace{-10pt}
\begin{eqnarray}\label{e:B.1}
|D_\lambda F(x,\xi;\lambda)-D_\lambda F(x,0;\lambda)|&\le& \sum^N_{i=1}\sum_{|\alpha|\le m}|\frac{D_\lambda F}{\partial\xi^i_\alpha}(x,\tau\xi;\lambda)|\cdot|\xi_\alpha^i|
\nonumber\\
&\le&\sum^N_{i=1}\sum_{|\alpha|\le m}|D_\lambda F^i_{\alpha}(x,\tau\xi;\lambda)|\cdot|\xi_\alpha^i|.
\end{eqnarray}
because $D_\lambda F^i_{\alpha}(x,\tau\xi;\lambda)=D_\lambda F^i_{\alpha}(x,\tau\xi;\lambda)$  by (i) in Theorem~\ref{th:BifE.9}.
By (iii) we derive
\begin{eqnarray}\label{e:B.2}
&&|D_\lambda F^i_\alpha(x,\tau\xi;\lambda)|\le|D_\lambda F^i_\alpha(x,0;\lambda)|\nonumber\\
&&+ \mathfrak{g}(\sum^N_{k=1}|\xi_\circ^k|)\sum_{|\beta|<m-n/2}\bigg(1+
\sum^N_{k=1}\sum_{m-n/2\le |\gamma|\le
m}|\xi^k_\gamma|^{2_\gamma }\bigg)^{2_{\alpha\beta}}\nonumber\\
&&+\mathfrak{g}(\sum^N_{k=1}|\xi^k_\circ|)\sum^N_{l=1}\sum_{m-n/2\le |\beta|\le m}
\bigg(1+ \sum^N_{k=1}\sum_{m-n/2\le
|\gamma|\le m}|\xi^k_\gamma|^{2_\gamma }\bigg)^{2_{\alpha\beta}}|\xi^l_\beta|
\end{eqnarray}
because $\mathfrak{g}(\sum^N_{k=1}|\tau\xi_\circ^k|)\le \mathfrak{g}(\sum^N_{k=1}|\xi_\circ^k|)$.
Recall that $\xi^k_\circ=\{\xi^k_\alpha\,:\,|\alpha|<m-n/2\}$ for $k=1,\cdots,N$.
There exists a constant $C(m,n,N)>0$ such that
\begin{eqnarray}\label{e:B.3}
\sum^N_{i=1}\sum_{|\alpha|\le m}|\xi_\alpha^i|&=&\sum^N_{i=1}\sum_{|\alpha|<m-n/2}|\xi_\alpha^i|+
\sum^N_{i=1}\sum_{m-n/2\le |\alpha|\le m}|\xi_\alpha^i|\nonumber\\
&\le&\sqrt{M(m-n/2)}\sum^N_{i=1}|\xi_\circ^i|+
\sum^N_{i=1}\sum_{m-n/2\le |\alpha|\le m}|\xi_\alpha^i|\nonumber\\
&\le&C(m,n,N)\left(1+\sum^N_{i=1}|\xi_\circ^i|+
\sum^N_{i=1}\sum_{m-n/2\le |\alpha|\le m}|\xi_\alpha^i|^{2_\alpha}\right).
\end{eqnarray}
Moreover,
$|D_\lambda F^i_\alpha(x,0;\lambda)|\cdot|\xi_\alpha^i|\le \frac{1}{2_\alpha'}
|D_\lambda F^i_\alpha(x,0;\lambda)|^{2'_\alpha}+ \frac{1}{2^\alpha}|\xi_\alpha^i|^{2^\alpha}$ we obtain
\begin{eqnarray}\label{e:B.4}
&&\sum^N_{i=1}\sum_{|\alpha|\le m}|D_\lambda F^i_\alpha(x,0;\lambda)|\cdot|\xi_\alpha^i|\nonumber\\
&=&\sum^N_{i=1}\sum_{|\alpha|<m-n/2}|D_\lambda F^i_\alpha(x,0;\lambda)|\cdot|\xi_\alpha^i|+
\sum^N_{i=1}\sum_{m-n/2\le|\alpha|\le m}|D_\lambda F^i_\alpha(x,0;\lambda)|\cdot|\xi_\alpha^i|\nonumber\\
&\le&(\sum^N_{i=1}|\xi_\circ^i|)\sum^N_{i=1}\sum_{|\alpha|<m-n/2}|D_\lambda F^i_\alpha(x,0;\lambda)|\nonumber\\
&&+\sum^N_{i=1}\sum_{m-n/2\le|\alpha|\le m}|D_\lambda F^i_\alpha(x,0;\lambda)|^{2'_\alpha}+
\sum^N_{i=1}\sum_{m-n/2\le|\alpha|\le m}|\xi_\alpha^i|^{2_\alpha}.
\end{eqnarray}
Let $J_1=\sum^N_{i=1}|\xi_\circ^i|$ and $J_2=\sum^N_{i=1}\sum_{m-n/2\le |\alpha|\le m}|\xi_\alpha^i|^{2_\alpha}$.
Hence from (\ref{e:B.1}) and (\ref{e:B.2}) we derive
\begin{eqnarray}\label{e:B.5}
|D_\lambda F(x,\xi;\lambda)|&\le&|D_\lambda F(x,0;\lambda)|+ \sum^N_{i=1}\sum_{|\alpha|\le m}|D_\lambda F^i_{\alpha}(x,\tau\xi;\lambda)|\cdot|\xi_\alpha^i|\nonumber\\
&&\hspace{-1.5cm}\le|D_\lambda F(x,0;\lambda)|+
\sum^N_{i=1}\sum_{|\alpha|\le m}|D_\lambda F^i_\alpha(x,0;\lambda)|\cdot|\xi_\alpha^i|+T_1+T_2,
\end{eqnarray}
where \allowdisplaybreaks
{\footnotesize
\begin{eqnarray*}
T_1&=& \mathfrak{g}(\sum^N_{k=1}|\xi_\circ^k|)\sum_{|\alpha|\le m}\sum_{|\beta|<m-n/2}\bigg(1+
\sum^N_{k=1}\sum_{m-n/2\le |\gamma|\le
m}|\xi^k_\gamma|^{2_\gamma }\bigg)^{2_{\alpha\beta}}\sum^N_{i=1}|\xi_\alpha^i|\nonumber\\
T_2&=&\mathfrak{g}(\sum^N_{k=1}|\xi^k_\circ|)\sum_{|\alpha|\le m}\sum^N_{l=1}\sum_{m-n/2\le |\beta|\le m} \bigg(1+
\sum^N_{k=1}\sum_{m-n/2\le |\gamma|\le m}|\xi^k_\gamma|^{2_\gamma }\bigg)^{2_{\alpha\beta}}%\\&&
|\xi^l_\beta|\sum^N_{i=1}|\xi_\alpha^i|.
\end{eqnarray*}}
 Then $T_1=T_{11}+T_{12}$, where
\begin{eqnarray*}
T_{11}&=&\mathfrak{g}(\sum^N_{k=1}|\xi_\circ^k|)\sum_{|\alpha|< m-n/2}\sum_{|\beta|<m-n/2}\bigg(1+
\sum^N_{k=1}\sum_{m-n/2\le |\gamma|\le
m}|\xi^k_\gamma|^{2_\gamma }\bigg)^{2_{\alpha\beta}}\sum^N_{i=1}|\xi_\alpha^i|\nonumber\\
&\le&\mathfrak{g}(\sum^N_{k=1}|\xi_\circ^k|)M(m)\bigg(1+
\sum^N_{k=1}\sum_{m-n/2\le |\gamma|\le
m}|\xi^k_\gamma|^{2_\gamma }\bigg)\sum^N_{i=1}\sum_{|\alpha|< m-n/2}|\xi_\alpha^i|\nonumber\\
&\le&M(m)^2\mathfrak{g}(\sum^N_{k=1}|\xi_\circ^k|)\bigg(1+
\sum^N_{k=1}\sum_{m-n/2\le |\gamma|\le
m}|\xi^k_\gamma|^{2_\gamma }\bigg)\sum^N_{i=1}|\xi_\circ^i|
\end{eqnarray*}
because $\sum^N_{i=1}\sum_{|\alpha|<m-n/2}|\xi_\alpha^i|\le M(m)\sum^N_{i=1}|\xi_\circ^i|$ as above, and
\begin{eqnarray*}
T_{12}&&\hspace{-0.8cm}=\mathfrak{g}(\sum^N_{k=1}|\xi_\circ^k|)\sum_{m-n/2\le|\alpha|\le m}\sum_{|\beta|<m-n/2}\bigg(1+
\sum^N_{k=1}\sum_{m-n/2\le |\gamma|\le
m}|\xi^k_\gamma|^{2_\gamma }\bigg)^{2_{\alpha\beta}}\sum^N_{i=1}|\xi_\alpha^i|\nonumber\\
&\le&M(m)\mathfrak{g}(\sum^N_{k=1}|\xi_\circ^k|)\sum_{m-n/2\le|\alpha|\le m}\bigg(1+
\sum^N_{k=1}\sum_{m-n/2\le |\gamma|\le
m}|\xi^k_\gamma|^{2_\gamma }\bigg)^{1/2'_{\alpha}}\sum^N_{i=1}|\xi_\alpha^i|\nonumber\\
&&\hspace{-1cm}\le M(m)\mathfrak{g}(\sum^N_{k=1}|\xi_\circ^k|)\sum^N_{i=1}\sum_{m-n/2\le|\alpha|\le m}\left[\bigg(1+
\sum^N_{k=1}\sum_{m-n/2\le |\gamma|\le
m}|\xi^k_\gamma|^{2_\gamma }\bigg)+ |\xi_\alpha^i|^{2_\alpha}\right]\nonumber\\
&\le&2NM(m)^2\mathfrak{g}(\sum^N_{k=1}|\xi_\circ^k|)\bigg(1+
\sum^N_{k=1}\sum_{m-n/2\le |\gamma|\le m}|\xi^k_\gamma|^{2_\gamma }\bigg).
\end{eqnarray*}
because of the definition of $2_{\alpha\beta}$ and the H\"older inequality. Similarly,
we write $T_2=T_{21}+T_{22}$, where
{\footnotesize
\begin{eqnarray*}
T_{21}&&\hspace{-0.8cm}=\mathfrak{g}(\sum^N_{k=1}|\xi^k_\circ|)
\sum_{|\alpha|<m-n/2}\sum^N_{l=1}\sum_{m-n/2\le |\beta|\le m} \bigg(1+
\sum^N_{k=1}\sum_{m-n/2\le |\gamma|\le m}|\xi^k_\gamma|^{2_\gamma }\bigg)^{2_{\alpha\beta}}|\xi^l_\beta|\sum^N_{i=1}|\xi_\alpha^i|\\
&&\hspace{-1cm}=\mathfrak{g}(\sum^N_{k=1}|\xi^k_\circ|)
\sum_{|\alpha|<m-n/2}\sum^N_{l=1}\sum_{m-n/2\le |\beta|\le m}
\bigg(1+ \sum^N_{k=1}\sum_{m-n/2\le |\gamma|\le m}
|\xi^k_\gamma|^{2_\gamma }\bigg)^{1/2'_{\beta}}|\xi^l_\beta|\sum^N_{i=1}|\xi_\alpha^i|\\
&\le&M(m)\mathfrak{g}(\sum^N_{k=1}|\xi^k_\circ|)\sum^N_{i=1}|\xi_\circ^i|\sum^N_{l=1}
\sum_{m-n/2\le |\beta|\le m}
\bigg(1+ \sum^N_{k=1}\sum_{m-n/2\le |\gamma|\le m}|\xi^k_\gamma|^{2_\gamma }\bigg)^{1/2'_{\beta}}|\xi^l_\beta|
\end{eqnarray*}}because
$\sum^N_{i=1}\sum_{|\alpha|<m-n/2}|\xi_\alpha^i|\le M(m)\sum^N_{i=1}|\xi_\circ^i|$.
 But the H\"older inequality leads to
\begin{eqnarray*}
&&\sum^N_{l=1}\sum_{m-n/2\le |\beta|\le m}
\bigg(1+\sum^N_{k=1}\sum_{m-n/2\le
|\gamma|\le m}|\xi^k_\gamma|^{2_\gamma }\bigg)^{1/2'_{\beta}}|\xi^l_\beta|\\
&\le&\sum^N_{l=1}\sum_{m-n/2\le |\beta|\le m} \left[\bigg(1+\sum^N_{k=1}\sum_{m-n/2\le |\gamma|\le m}|\xi^k_\gamma|^{2_\gamma }\bigg)+
|\xi^l_\beta|^{2_\beta}\right]\\
&\le& 2NM(m)\bigg(1+\sum^N_{k=1}\sum_{m-n/2\le
 |\gamma|\le m}|\xi^k_\gamma|^{2_\gamma }\bigg).
\end{eqnarray*}
Hence
$$
T_{21}\le 2NM(m)^2\mathfrak{g}(\sum^N_{k=1}|\xi^k_\circ|)\sum^N_{i=1}|\xi_\circ^i|
\bigg(1+\sum^N_{k=1}\sum_{m-n/2\le |\gamma|\le m}|\xi^k_\gamma|^{2_\gamma }\bigg).
$$

For $m-n/2\le|\alpha|\le m$ and $m-n/2\le|\beta|\le m$, since
$0<2_{\alpha\beta}\le 1-2_\alpha^{-1}-2_\beta^{-1}=\tau^{-1}$,
using the H\"older inequality we get
\begin{eqnarray*}
&&\bigg(1+\sum^N_{k=1}\sum_{m-n/2\le |\gamma|\le m}|\xi^k_\gamma|^{2_\gamma }\bigg)^{2_{\alpha\beta}}|\xi^l_\beta||\xi_\alpha^i|\\
&\le&\bigg(1+\sum^N_{k=1}\sum_{m-n/2\le |\gamma|\le m}|\xi^k_\gamma|^{2_\gamma }\bigg)^{2_{\alpha\beta}\tau}+ |\xi^l_\beta|^{2_\beta}+ |\xi_\alpha^i|^{2_\alpha}\\
&\le&1+\sum^N_{k=1}\sum_{m-n/2\le |\gamma|\le m}|\xi^k_\gamma|^{2_\gamma }+ |\xi^l_\beta|^{2_\beta}+ |\xi_\alpha^i|^{2_\alpha}
\end{eqnarray*}
and therefore
{\footnotesize
\begin{eqnarray*}
&&T_{22}\\&=&\mathfrak{g}(\sum^N_{k=1}|\xi^k_\circ|)\sum_{m-n/2\le|\alpha|\le m}\sum^N_{l=1}\sum_{m-n/2\le |\beta|\le m} \bigg(1+
\sum^N_{k=1}\sum_{m-n/2\le |\gamma|\le m}|\xi^k_\gamma|^{2_\gamma }\bigg)^{2_{\alpha\beta}}|\xi^l_\beta|\sum^N_{i=1}|\xi_\alpha^i|\\
&\le&N^2M(m)^2\mathfrak{g}(\sum^N_{k=1}|\xi^k_\circ|)\left(1+\sum^N_{k=1}\sum_{m-n/2\le |\gamma|\le m}|\xi^k_\gamma|^{2_\gamma }\right)\\
&&+NM(m)\mathfrak{g}(\sum^N_{k=1}|\xi^k_\circ|)\sum_{m-n/2\le|\beta|\le m}\sum^N_{l=1}|\xi^l_\beta|^{2_\beta}\\
&&+NM(m)\mathfrak{g}(\sum^N_{k=1}|\xi^k_\circ|)\sum_{m-n/2\le|\alpha|\le m}\sum^N_{i=1}|\xi^i_\alpha|^{2_\alpha}\\
&\le&3N^2M(m)^2\mathfrak{g}(\sum^N_{k=1}|\xi^k_\circ|)\left(1+\sum^N_{k=1}\sum_{m-n/2\le |\gamma|\le m}|\xi^k_\gamma|^{2_\gamma }\right).
\end{eqnarray*}}

It follows that
\begin{eqnarray*}
T_1&\le& M(m)^2\mathfrak{g}(J_1)J_1(1+J_2)+2NM(m)^2\mathfrak{g}(J_1)(1+J_2)\\
&=&M(m)^2[\mathfrak{g}(J_1)J_1+2N\mathfrak{g}(J_1)](1+J_2),\\
T_2&\le&2NM(m)^2\mathfrak{g}(J_1)J_1(1+J_2)+3N^2M(m)^2\mathfrak{g}(J_1)(1+J_2)\\
&=&M(m)^2[2N\mathfrak{g}(J_1)J_1+3N^2\mathfrak{g}(J_1)](1+J_2).
\end{eqnarray*}
Finally, by (\ref{e:B.4}) and (\ref{e:B.5}) we deduce
\begin{eqnarray}\label{e:B.6}
|D_\lambda F(x,\xi;\lambda)|&\le&|D_\lambda F(x,0;\lambda)|+
\sum^N_{i=1}\sum_{m-n/2\le|\alpha|\le m}|D_\lambda F^i_\alpha(x,0;\lambda)|^{2'_\alpha}\nonumber\\
&&+J_1\sum^N_{i=1}\sum_{|\alpha|<m-n/2}|D_\lambda F^i_\alpha(x,0;\lambda)|+
J_2\nonumber\\
&&+M(m)^2[3N\mathfrak{g}(J_1)J_1+4N^2\mathfrak{g}(J_1)](1+J_2).
\end{eqnarray}
Define $\widehat{\mathfrak{g}}:[0,\infty)\to\R$ by
$\widehat{\mathfrak{g}}(t)=1+M(m)^2(3N\mathfrak{g}(t)t+4N^2\mathfrak{g}(t))$.
This is a continuous, positive, nondecreasing function
$\widehat{\mathfrak{g}}:[0,\infty)\to\mathbb{R}$.
Lemma~\ref{lem:BifE.10} follows from (\ref{e:B.6}).

\subsection{Proof of Lemma~\ref{lem:BifE.11}}\label{app:B.2}

For $\vec{u}\in V$ and $w\in V_0$, by (\ref{e:6.4}) we have
\begin{eqnarray}\label{e:B.6+}
&&|(\nabla\mathfrak{F}_{\lambda_1}(\vec{u})- \nabla\mathfrak{F}_{\lambda_2}(\vec{u}), \vec{w})_{m,2}|
\notag\\&=&\Bigg|\sum^N_{j=1}\sum_{|\alpha|\le m}\int_\Omega F^j_\alpha(x,
\vec{u}(x),\cdots, D^m \vec{u}(x);\lambda_1)D^\alpha w^j dx\nonumber\\
&-&\sum^N_{j=1}\sum_{|\alpha|\le m}\int_\Omega F^j_\alpha(x,
\vec{u}(x),\cdots, D^m \vec{u}(x);\lambda_2)D^\alpha w^j dx\Bigg|\nonumber\\
&\le&|\lambda_1-\lambda_2|\sum^N_{j=1}\sum_{|\alpha|\le m}I_{j\alpha},
\end{eqnarray}
where
$$
I_{j\alpha}=\sup_{\lambda\in[\lambda_1,\lambda_2]}\int_\Omega |D_\lambda F^j_\alpha(x,
\vec{u}(x),\cdots, D^m \vec{u}(x);\lambda)|\cdot|D^\alpha w^j| dx.
$$

By the Sobolev embedding theorem there exists  a constant $C>0$ such that
for any $\vec{w}\in W^{m,2}(\Omega,\R^N)$ and $j=1,\cdots,N$,
\begin{eqnarray}\label{e:B.7}
&&\|D^\alpha w^j\|_{C^0}\le C\|w^j\|_{m,2}\le C\|\vec{w}\|_{m,2}\quad\hbox{if}\quad |\alpha|<m-n/2,\\
&&\hspace{-.6cm}\|D^\alpha w^j\|_{2_\alpha}\le C\|D^\alpha w^j\|_{m-|\alpha|,2}\le C\|\vec{w}\|_{m,2}\quad\hbox{if}\quad m-n/2\le |\alpha|\le m.\label{e:B.8}
\end{eqnarray}
It follows from  the condition (iii) of Theorem~\ref{th:BifE.9} that
{\footnotesize
 \begin{eqnarray*}
 &&\int_\Omega |D_\lambda F^j_\alpha(x, \vec{u}(x),\cdots, D^m \vec{u}(x);\lambda)|\cdot|D^\alpha w^j| dx
\le \int_\Omega |D_\lambda F^j_\alpha(x,0;\lambda)|\cdot|D^\alpha w^j| dx\\
&+& \sup_{k<m-n/2}\mathfrak{g}(\|\vec{u}\|_{C^k})\sum_{|\beta|<m-n/2}\int_\Omega \bigg(1+
\sum^N_{k=1}\sum_{m-n/2\le |\gamma|\le
m}|D^\gamma u^k|^{2_\gamma }\bigg)^{2_{\alpha\beta}}\cdot|D^\alpha w^j| dx\nonumber\\
&+&\sup_{k<m-n/2}\mathfrak{g}(\|\vec{u}\|_{C^k})\sum^N_{l=1}\sum_{m-n/2\le |\beta|\le m} \int_\Omega \bigg(1+
\sum^N_{k=1}\sum_{m-n/2\le |\gamma|\le m}|D^\gamma u^k|^{2_\gamma }\bigg)^{2_{\alpha\beta}}\\&&|D^\beta u^l|\cdot|D^\alpha w^j| dx\\
&=&J_1+J_2+J_3.
\end{eqnarray*}}
Since $2_\alpha'=1$ for $|\alpha|<m-n/2$, by the H\"older inequality and (\ref{e:B.7})-(\ref{e:B.8}) we always have
\begin{eqnarray}\label{e:B.9}
J_1&=&\int_\Omega |D_\lambda F^j_\alpha(x,0;\lambda)|\cdot|D^\alpha w^j| dx \notag\\&\le&
C\left(\int_\Omega |D_\lambda F^j_\alpha(x,0;\lambda)|^{2'_\alpha}dx\right)^{1/2'_\alpha}\|\vec{w}\|_{m,2}.
\end{eqnarray}
In order to estimate $J_2$, we first consider the case $|\alpha|<m-n/2$. Because $2_{\alpha\beta}=1$
for $|\alpha|<m-n/2$ and $|\beta|<m-n/2$, by (\ref{e:B.7}) we obtain
\begin{eqnarray*}
&&\sum_{|\beta|<m-n/2}\int_\Omega \bigg(1+
\sum^N_{k=1}\sum_{m-n/2\le |\gamma|\le
m}|D^\gamma u^k|^{2_\gamma }\bigg)^{2_{\alpha\beta}}\cdot|D^\alpha w^j| dx\\
&\le&C\|\vec{w}\|_{m,2}\sum_{|\beta|<m-n/2}\int_\Omega \bigg(1+
\sum^N_{k=1}\sum_{m-n/2\le |\gamma|\le
m}|D^\gamma u^k|^{2_\gamma }\bigg)dx\\
&\le& C(m,n,N,\Omega)\left(1+\sum_{m-n/2\le |\gamma|\le
m}\|\vec{u}\|_{m,2}^{2_\gamma}\right)\|\vec{w}\|_{m,2}.
\end{eqnarray*}
Hereafter we use $C(m,n,N,\Omega)$ to denote different constants only depending on $m,n,N$ and $\Omega$.
Similarly, for the case $m-n/2\le|\alpha|\le m$ we have $2_{\alpha\beta}=  1-\frac{1}{2_\alpha}= \frac{1}{2'_\alpha}$ if $|\beta|<m-n/2$,
 and hence the H\"older inequality and (\ref{e:B.8}) lead to
\begin{eqnarray*}
&&\sum_{|\beta|<m-n/2}\int_\Omega \bigg(1+
\sum^N_{k=1}\sum_{m-n/2\le |\gamma|\le
m}|D^\gamma u^k|^{2_\gamma }\bigg)^{2_{\alpha\beta}}\cdot|D^\alpha w^j| dx\\
&\le&M(m)\left(\int_\Omega \bigg(1+\sum^N_{k=1}\sum_{m-n/2\le |\gamma|\le
m}|D^\gamma u^k|^{2_\gamma }\bigg)dx\right)^{\frac{1}{2'_\alpha}}\left(\int_\Omega |D^\alpha w^j|^{2_\alpha}dx\right)^{\frac{1}{2_\alpha}}\\
&\le& C(m,n,N,\Omega)\left(1+\sum_{m-n/2\le |\gamma|\le
m}\|\vec{u}\|_{m,2}^{2_\gamma}\right)\|\vec{w}\|_{m,2}.
\end{eqnarray*}
These imply
\begin{eqnarray}\label{e:B.10}
J_2\le C(m,n,N,\Omega)\sup_{k<m-n/2}\mathfrak{g}(\|\vec{u}\|_{C^k}|)\left(1+\sum_{m-n/2\le |\gamma|\le
m}\|\vec{u}\|_{m,2}^{2_\gamma}\right)\|\vec{w}\|_{m,2}.
\end{eqnarray}

The estimation of $J_3$ is as follows. Firstly we consider the case $|\alpha|<m-n/2$. Then
$2_{\alpha\beta}=1-1/2_\beta=1/2_\beta'$ for $m-n/2\le |\beta|\le m$. By (\ref{e:B.8}) and the H\"older inequality we derive
\begin{eqnarray*}
&&\sum_{m-n/2\le |\beta|\le m}\int_\Omega \bigg(1+\sum^N_{k=1}\sum_{m-n/2\le |\gamma|\le m}|D\gamma u^k|^{2_\gamma }\bigg)^{2_{\alpha\beta}}|D^\beta u^l|\cdot|D^\alpha w^j| dx\\
&\le&C\|\vec{w}\|_{m,2}\sum_{m-n/2\le |\beta|\le m}\int_\Omega \bigg(1+\sum^N_{k=1}\sum_{m-n/2\le |\gamma|\le m}|D\gamma u^k|^{2_\gamma }\bigg)^{1/2'_{\beta}}|D^\beta u^l|dx\\
&\le&C'\|\vec{u}\|_{m,2}\|\vec{w}\|_{m,2}\left(\int_\Omega \bigg(1+\sum^N_{k=1}\sum_{m-n/2\le |\gamma|\le m}|D\gamma u^k|^{2_\gamma }\bigg)dx\right)^{1/2'_{\beta}}\\
&\le& C(m,n,N,\Omega)\left(1+\sum_{m-n/2\le |\gamma|\le
m}\|\vec{u}\|_{m,2}^{2_\gamma}\right)\|\vec{u}\|_{m,2}\|\vec{w}\|_{m,2}.
\end{eqnarray*}
Next we study the case $m-n/2\le|\alpha|\le m$. If $m-n/2\le |\beta|\le m$ satisfies $|\alpha|+|\beta|=2m$ (so $|\alpha|=|\beta|=m$)
then $2_{\alpha\beta}=  1-\frac{1}{2_\alpha}-\frac{1}{2_\beta}$. It follows from the H\"older inequality and (\ref{e:B.7})-(\ref{e:B.8}) that
{\footnotesize
\begin{eqnarray*}
&&\int_\Omega \bigg(1+\sum^N_{k=1}\sum_{m-n/2\le |\gamma|\le m}|D\gamma u^k|^{2_\gamma }\bigg)^{2_{\alpha\beta}}|D^\beta u^l|\cdot|D^\alpha w^j| dx\\
&\le&\left(\int_\Omega \bigg(1+\sum^N_{k=1}\sum_{m-n/2\le |\gamma|\le m}|D\gamma u^k|^{2_\gamma }\bigg)dx\right)^{2_{\alpha\beta}}
\left(\int_\Omega|D^\beta u^l|^{2_\beta}dx\right)^{1/2_\beta}\\&&\left(\int_\Omega|D^\alpha w^j|^{2_\alpha} dx\right)^{1/2_\alpha}\\
&\le& C(m,n,N,\Omega)\left(1+\sum_{m-n/2\le |\gamma|\le
m}\|\vec{u}\|_{m,2}^{2_\gamma}\right)\|\vec{u}\|_{m,2}\|\vec{w}\|_{m,2}.
\end{eqnarray*}}
If $m-n/2\le |\beta|\le m$ satisfies $|\alpha|+|\beta|<2m$
then  $0<2_{\alpha\beta}<1-\frac{1}{2_\alpha}-\frac{1}{2_\beta}$, $2_\alpha>1$ and $2_\beta>1$.
We can choose  $q_{\alpha\beta}>1$ such that
$\frac{1}{q_{\alpha\beta}}+{2_{\alpha\beta}}+\frac{1}{2_\alpha}+\frac{1}{2_\beta}=1$.
Let $|\Omega|$ denote the Euclidean volume of $\Omega$. Using  the H\"older inequality and (\ref{e:B.7})-(\ref{e:B.8}) as above we get
{\scriptsize
\begin{eqnarray*}
&&\int_\Omega \bigg(1+\sum^N_{k=1}\sum_{m-n/2\le |\gamma|\le m}|D\gamma u^k|^{2_\gamma }\bigg)^{2_{\alpha\beta}}|D^\beta u^l|\cdot|D^\alpha w^j| dx\\
&\le&|\Omega|^{1/q_{\alpha\beta}}\left(\int_\Omega \bigg(1+\sum^N_{k=1}\sum_{m-n/2\le |\gamma|\le m}|D\gamma u^k|^{2_\gamma }\bigg)dx\right)^{2_{\alpha\beta}}\\&&
\left(\int_\Omega|D^\beta u^l|^{2_\beta}dx\right)^{1/2_\beta}\left(\int_\Omega|D^\alpha w^j|^{2_\alpha} dx\right)^{1/2_\alpha}\\
&\le& C(m,n,N,\Omega)\left(1+\sum_{m-n/2\le |\gamma|\le
m}\|\vec{u}\|_{m,2}^{2_\gamma}\right)\|\vec{u}\|_{m,2}\|\vec{w}\|_{m,2}.
\end{eqnarray*}}
All above arguments show
%\begin{eqnarray}\label{e:B.11}
%&&J_3\le C(m,n,N,\Omega)\sup_{k<m-n/2}\mathfrak{g}\notag\\&&\quad(\|\vec{u}\|_{C^k})(1+\sum_{m-n/2\le |\gamma|\le
%m}\|\vec{u}\|_{m,2}^{2_\gamma})\|\vec{u}\|_{m,2}\|\vec{w}\|_{m,2}.
%\end{eqnarray}
\begin{eqnarray}\label{e:B.11}
\hspace{-2.4cm}J_3\le C(m,n,N,\Omega)\sup_{k<m-n/2}\mathfrak{g}%\notag\\&&\quad
(\|\vec{u}\|_{C^k})\bigg(1+\sum_{m-n/2\le |\gamma|\le
m}\|\vec{u}\|_{m,2}^{2_\gamma}\bigg)\|\vec{u}\|_{m,2}\|\vec{w}\|_{m,2}.
\hspace{-3.8cm}
\nonumber\\
\end{eqnarray}
This and (\ref{e:B.9})-(\ref{e:B.10}) lead to
\begin{eqnarray*}
 &&\int_\Omega |D_\lambda F^j_\alpha(x, \vec{u}(x),\cdots, D^m \vec{u}(x);\lambda)|\cdot|D^\alpha w^j| dx\\&&\le C\left(\int_\Omega |D_\lambda F^j_\alpha(x,0;\lambda)|^{2'_\alpha}dx\right)^{1/2'_\alpha}\|\vec{w}\|_{m,2}\\
&& +C(m,n,N,\Omega)\sup_{k<m-n/2}\mathfrak{g}(\|\vec{u}\|_{C^k})
\bigg(1+\sum_{m-n/2\le |\gamma|\le
m}\|\vec{u}\|_{m,2}^{2_\gamma}\bigg)\|\vec{w}\|_{m,2}\\&&
 +C(m,n,N,\Omega)\sup_{k<m-n/2}\mathfrak{g}(\|\vec{u}\|_{C^k})\bigg(1+\sum_{m-n/2\le |\gamma|\le m}\|\vec{u}\|_{m,2}^{2_\gamma}\bigg)\|\vec{u}\|_{m,2}\|\vec{w}\|_{m,2},
 \end{eqnarray*}
 and therefore via (\ref{e:B.6+})
\begin{eqnarray*}
&&\|\nabla\mathfrak{F}_{\lambda_1}(\vec{u})- \nabla\mathfrak{F}_{\lambda_2}(\vec{u})\|_{m,2}\\
&\le&|\lambda_1-\lambda_2|\bigg[
C\sum^N_{j=1}\sum_{|\alpha|\le m}\left(\int_\Omega |D_\lambda F^j_\alpha(x,0;\lambda)|^{2'_\alpha}dx\right)^{1/2'_\alpha}\\
&&\hspace{-8pt}+NM(m)C(m,n,N,\Omega)\sup_{k<m-n/2}\mathfrak{g}(\|\vec{u}\|_{C^k})
\bigg(1+\sum_{m-n/2\le |\gamma|\le
m}\|\vec{u}\|_{m,2}^{2_\gamma}\bigg)\\
&&\hspace{-8pt}+NM(m)C(m,n,N,\Omega)\sup_{k<m-n/2}\mathfrak{g}(\|\vec{u}\|_{C^k})
\bigg(1+\sum_{m-n/2\le |\gamma|\le
m}\|\vec{u}\|_{m,2}^{2_\gamma}\bigg)\|\vec{u}\|_{m,2}\bigg].
\end{eqnarray*}
The proof of (\ref{lem:BifE.11}) is complete.

\subsection{Proof of Lemma~\ref{lem:BifE.12}}\label{app:B.3}

Fix $R>0$.
By (i), all $F(\cdot;\lambda)$  uniformly satisfy  (\ref{e:6.2})
for all $\lambda\in [0,1]$.
It follows from (\ref{e:6.6}) that there exists a positive constant $c_0$
such that for all $\vec{u}\in\bar{B}_V(0, R)$ and $\vec{v}\in V_0$,
\begin{eqnarray}\label{e:B.12}
({P}_\lambda(\vec{u})\vec{v},\vec{v})_{m,2}\ge c_0\|\vec{v}\|_{m,2}^2.
\end{eqnarray}

Let $(\lambda, \vec{u}_k)\in [0,1]\times\bar{B}_V(0,R)$ such that
$\nabla\mathfrak{F}_{\lambda}(\vec{u}_k)\to 0$ as $k\to\infty$. We can assume
that $\vec{u}_k\rightharpoonup\vec{u}\in \bar{B}_V(0,R)$.
By the mean value theorem we have
a sequence $(\tau_k)\subset (0,1)$ such that
{\footnotesize
\begin{eqnarray}\label{e:B.12+}
&&(\nabla\mathfrak{F}_{\lambda}(\vec{u}_k)- \nabla\mathfrak{F}_{\lambda}(\vec{u}), \vec{u}_k-\vec{u})_{m,2}
=(B_{\lambda}(\tau_k\vec{u}_k+ (1-\tau_k)\vec{u})(\vec{u}_k-\vec{u}), \vec{u}_k-\vec{u})_{m,2}\nonumber\\
&=&(P_{\lambda}(\tau_k\vec{u}_k+ (1-\tau_k)\vec{u})(\vec{u}_k-\vec{u}), \vec{u}_k-\vec{u})_{m,2}+(Q_{\lambda}(\tau_k\vec{u}_k+ (1-\tau_k)\vec{u})(\vec{u}_k-\vec{u}), \vec{u}_k-\vec{u})_{m,2}\nonumber\\
&\ge&c_0\|\vec{u}_k-\vec{u}\|^2_{m,2}+(Q_{\lambda}(\tau_k\vec{u}_k+ (1-\tau_k)\vec{u})(\vec{u}_k-\vec{u}), \vec{u}_k-\vec{u})_{m,2}.
\end{eqnarray}}
Define $\vec{w}_k=\tau_k\vec{u}_k+ (1-\tau_k)\vec{u}$ and
$\vec{v}_k=\vec{u}_k-\vec{u}$ for $k=1,2,\cdots$.
Then  (\ref{e:6.7}) produces
 \begin{eqnarray}\label{e:B.13}
   && (Q_\lambda(\vec{w}_k)\vec{v}_k,\vec{v}_k)_{m,2}\notag \\&=&\sum^N_{i,j=1}\sum_{|\alpha|+|\beta|<2m}
   \int_\Omega F^{ij}_{\alpha\beta}(x, \vec{w}_k(x),\cdots,
   D^m \vec{w}_k(x);\lambda)D^\beta v^j_k\cdot D^\alpha v^i_k dx\nonumber\\
  &&-\sum^N_{i=1}\sum_{|\alpha|\le m-1}
  \int_\Omega  D^\alpha v^i_k\cdot D^\alpha v^i_k dx.
    \end{eqnarray}
By the condition (i) of Theorem~\ref{th:BifE.9}
there exists a continuous, positive, nondecreasing functions
$\mathfrak{g}_1$ such that for $i,j=1,\cdots,N$, $|\alpha|, |\beta|\le m$  and
$(x, \xi, \lambda)\in\overline\Omega\times\R^{M(m)}\times [0,1]$,
\begin{eqnarray}\label{e:B.14}
 |F^{ij}_{\alpha\beta}(x,\xi;\lambda)|\le
\mathfrak{g}_1(\sum^N_{l=1}|\xi_\circ^l|)\left(1+
\sum^N_{l=1}\sum_{m-n/2\le |\gamma|\le
m}|\xi^l_\gamma|^{2_\gamma}\right)^{2_{\alpha\beta}}.
\end{eqnarray}
Since $\|\vec{w}_k\|_{m,2}\le R$, by the Sobolev embedding
theorem we have a constant $C_0=C_0(m,n,N,\Omega)>0$ such that
$\sup_{l<m-n/2}\mathfrak{g}_1(\|\vec{w}_k\|_{C^l})\le \mathfrak{g}_1(RC_0)$.

Using these inequalities we can estimate (\ref{e:B.13}) as follows.

\vspace{4pt}\noindent
{\bf Case 1} ($m-n/2\le |\alpha|\le  m,\; |\beta|<m-n/2$).\quad
Then $2_{\alpha\beta}=  1-\frac{1}{2_\alpha}$ and $D^\beta v^j_k\in C^0(\overline{\Omega})$.
It follows from these that
{\footnotesize
\begin{eqnarray*}
&&\left|\int_\Omega F^{ij}_{\alpha\beta}(x, \vec{w}_k(x),\cdots, D^m \vec{w}_k(x);\lambda)D^\beta v^j_k\cdot D^\alpha v^i_k dx\right|\nonumber\\
 &\le&\mathfrak{g}_1(RC_0)\int_\Omega
  \bigg(1+\sum^N_{l=1}\sum_{m-n/2\le |\gamma|\le
m}|D^\gamma w_k^l(x)|^{2_\gamma}\bigg)^{2_{\alpha\beta}}|D^\beta v_k^j(x)|\cdot |D^\alpha v_k^i(x)|dx\\
&\le&\mathfrak{g}_1(RC_0)\|D^\beta v^j_k\|_{C^0}\int_\Omega
  \bigg(1+\sum^N_{l=1}\sum_{m-n/2\le |\gamma|\le
m}|D^\gamma w_k^l(x)|^{2_\gamma}\bigg)^{1/2'_{\alpha}}\cdot |D^\alpha v_k^i(x)|dx\\
&&\hspace{-0.8cm}\le \mathfrak{g}_1(RC_0)\|D^\beta v^j_k\|_{C^0}\bigg(\int_\Omega
  \bigg(1+\sum^N_{l=1}\sum_{m-n/2\le |\gamma|\le
m}|D^\gamma w_k^l(x)|^{2_\gamma}\bigg)\bigg)^{1/2'_{\alpha}}
\bigg(\int_\Omega|D^\alpha v_k^i(x)|^{2_\alpha}\bigg)^{1/2_\alpha}\\
&\le&  C(m,n,N,\Omega)\mathfrak{g}_1(RC_0)\bigg(1+\sum_{m-n/2\le |\gamma|\le
m}\|\vec{w}_k\|_{m,2}^{2_\gamma}\bigg)^{1/2'_{\alpha}}\|D^\beta v^j_k\|_{C^0}\|\vec{v}_k\|_{m,2}\\
&\le& \|D^\beta v^j_k\|_{C^0} C(m,n,N,\Omega)\mathfrak{g}_1(RC_0)\bigg(1+\sum_{m-n/2\le |\gamma|\le
m}R^{2_\gamma}\bigg)^{1/2'_{\alpha}}\|\vec{v}_k\|_{m,2}.
    \end{eqnarray*}}

\noindent
{\bf Case 2} ($m-n/2\le |\beta|\le  m,\; |\alpha|<m-n/2$).\quad Commuting $\alpha$ and $\beta$ in the last case we obtain
  \begin{eqnarray*}
&&\left|\int_\Omega F^{ij}_{\alpha\beta}(x, \vec{w}_k(x),\cdots, D^m \vec{w}_k(x);\lambda)D^\beta v^j_k\cdot D^\alpha v^i_k dx\right|\nonumber\\
&\le& \|D^\alpha v^i_k\|_{C^0} C(m,n,N,\Omega)\mathfrak{g}_1(RC_0)\bigg(1+\sum_{m-n/2\le |\gamma|\le m}R^{2_\gamma}\bigg)^{1/2'_{\beta}}\|\vec{v}_k\|_{m,2}.
    \end{eqnarray*}

\noindent
{\bf Case 3} ($|\alpha|, |\beta|<m-n/2$).\quad
We have $2_{\alpha\beta}= 1$, $2'_{\alpha}=1$ and so
\begin{eqnarray*}
&&\left|\int_\Omega F^{ij}_{\alpha\beta}(x, \vec{w}_k(x),\cdots, D^m \vec{w}_k(x);\lambda)D^\beta v^j_k\cdot D^\alpha v^i_k dx\right|\nonumber\\
 &\le&\mathfrak{g}_1(RC_0)\int_\Omega
  \bigg(1+\sum^N_{l=1}\sum_{m-n/2\le |\gamma|\le
m}|D^\gamma w_k^l(x)|^{2_\gamma}\bigg)^{2_{\alpha\beta}}|D^\beta v_k^j(x)|\cdot |D^\alpha v_k^i(x)|dx\\
&\le&  C(m,n,N,\Omega)\mathfrak{g}_1(RC_0)\bigg(1+\sum_{m-n/2\le |\gamma|\le
m}\|\vec{w}_k\|_{m,2}^{2_\gamma}\bigg)\|D^\beta v^j_k\|_{C^0}\|D^\alpha\vec{v}^i_k\|_{C^0}\\
&\le& \|D^\beta v^j_k\|_{C^0} C(m,n,N,\Omega)\mathfrak{g}_1(RC_0)\bigg(1+\sum_{m-n/2\le |\gamma|\le
m}R^{2_\gamma}\bigg)^{1/2'_{\alpha}}\|\vec{v}_k\|_{m,2}.
    \end{eqnarray*}

\noindent
{\bf Case 4} ($|\alpha|,\;|\beta|\ge m-n/2,\;|\alpha|+|\beta|<2m$).\quad
Then $0<2_{\alpha\beta}<1-\frac{1}{2_\alpha}-\frac{1}{2_\beta}$, $2_\alpha>1$ and $2_\beta>1$.
We can choose $\epsilon>0$ and $q_{\alpha\beta}>1$ such that $2_\alpha-\epsilon>1$, $2_\beta-\epsilon>1$ and
$$
\frac{1}{q_{\alpha\beta}}+{2_{\alpha\beta}}+
\frac{1}{2_\alpha-\epsilon}+\frac{1}{2_\beta-\epsilon}=1.
$$
Using the H\"older inequality we get
{\small
\begin{eqnarray*}
&&\left|\int_\Omega F^{ij}_{\alpha\beta}(x, \vec{w}_k(x),\cdots, D^m \vec{w}_k(x);\lambda)D^\beta v^j_k\cdot D^\alpha v^i_k dx\right|\nonumber\\
 &\le&\mathfrak{g}_1(RC_0)\int_\Omega
  \bigg(1+\sum^N_{l=1}\sum_{m-n/2\le |\gamma|\le
m}|D^\gamma w_k^l(x)|^{2_\gamma}\bigg)^{2_{\alpha\beta}}|D^\beta v_k^j(x)|\cdot |D^\alpha v_k^i(x)|dx\\
&\le&\mathfrak{g}_1(RC_0)|\Omega|^{1/q_{\alpha\beta}}\bigg(\int_\Omega
  \bigg(1+\sum^N_{l=1}\sum_{m-n/2\le |\gamma|\le
m}|D^\gamma w_k^l(x)|^{2_\gamma}\bigg)\bigg)^{2_{\alpha\beta}}\times\\
&&\hspace{30mm}\times\bigg(\int_\Omega|D^\beta v_k^j(x)|^{2_\beta-\epsilon}\bigg)^{1/(2_\beta-\epsilon)}
\bigg(\int_\Omega|D^\alpha v_k^i(x)|^{2_\alpha-\epsilon}\bigg)^{1/(2_\alpha-\epsilon)}\\
&\le&|\Omega|^{1/q_{\alpha\beta}} C(m,n,N,\Omega)\mathfrak{g}_1(RC_0)
\bigg(1+\sum_{m-n/2\le |\gamma|\le
m}\|\vec{w}_k\|_{m,2}^{2_\gamma}\bigg)\times\\
&&\hspace{30mm}\times\bigg(\int_\Omega|D^\beta v_k^j(x)|^{2_\beta-\epsilon}\bigg)^{1/(2_\beta-\epsilon)}
\bigg(\int_\Omega|D^\alpha v_k^i(x)|^{2_\alpha-\epsilon}\bigg)^{1/(2_\alpha-\epsilon)}.
    \end{eqnarray*}}

Note that $\vec{v}_k\rightharpoonup 0$.
In Case 1,  since $|\beta|<m-n/2$, $m-|\beta|>n/2$ and the embedding \linebreak $W^{m-|\beta|,2}(\Omega)\hookrightarrow C^0(\overline{\Omega})$ is compact.
Hence $\|D^\beta v^j_k\|_{C^0}\to 0$ as $k\to\infty$.

Similarly, in Case 2 we get that $\|D^\alpha v^i_k\|_{C^0}\to 0$ as $k\to\infty$.

In Case 3, since $|\beta|<m-n/2$ we have $\|D^\beta v^j_k\|_{C^0}\to 0$ as above.

In Case 4, since $|\alpha|+|\beta|<2m$, either $|\alpha|<m$ or $|\beta|<m$. It follows that the embedding
$$
\hbox{either}\;W^{m-|\alpha|,2}(\Omega)\hookrightarrow L^{2_\alpha-\epsilon}(\Omega)\quad\hbox{or}\quad
W^{m-|\beta|,2}(\Omega)\hookrightarrow L^{2_\beta-\epsilon}(\Omega)
$$
is compact. Then either $\|D^\beta v_k^j(x)\|_{L^{2_\beta-\epsilon}}\to 0$ or
$\|D^\alpha v_k^i\|_{L^{2_\alpha-\epsilon}}\to 0$.

Moreover, the embedding $W^{m,2}(\Omega)\hookrightarrow W^{m-1,2}(\Omega)$ is compact and
 \begin{eqnarray*}
   \bigg|\sum^N_{i=1}\sum_{|\alpha|\le m-1}\int_\Omega  D^\alpha v^i_k\cdot D^\alpha v^i_k dx\bigg|\le \|\vec{v}_k\|_{m-1,2}^2.
    \end{eqnarray*}
It follows from these and (\ref{e:B.13}) that there exists a positive integer $K>0$ such that
\begin{eqnarray*}
    |(Q_\lambda(\vec{w}_k)\vec{v}_k,\vec{v}_k)_{m,2}|\le \frac{c_0}{2}\|\vec{v}_k\|_{m,2},\quad\forall k>K.
    \end{eqnarray*}
Then (\ref{e:B.12+}) gives rise to
\begin{eqnarray*}
c_0\|\vec{u}_k-\vec{u}\|^2_{m,2}-\frac{c_0}{2}\|\vec{u}_k-\vec{u}\|_{m,2}
&\le&(\nabla\mathfrak{F}_{\lambda}(\vec{u}_k)- \nabla\mathfrak{F}_{\lambda}(\vec{u}), \vec{u}_k-\vec{u})_{m,2}\\
&=&(\nabla\mathfrak{F}_{\lambda}(\vec{u}_k),\vec{u}_k-\vec{u})_{m,2}-
(\nabla\mathfrak{F}_{\lambda}(\vec{u}), \vec{u}_k-\vec{u})_{m,2}
\end{eqnarray*}
for all $k>K$. However, $\nabla\mathfrak{F}_{\lambda}(\vec{u}_k)\to 0$, $\sup_k\|\vec{u}_k-\vec{u}\|_{m,2}\le 2R$ and
$\vec{u}_k\rightharpoonup\vec{u}$. Hence the right side converges to zero.
It follows that $\vec{u}_k\to\vec{u}$.

\section{Appendix:\quad
 A note on  the Nemytski operators}\label{app:C}\setcounter{equation}{0}

When $\Lambda$ consists of a single point,
the following proposition is the standard result concerning the continuity of
the Nemytski operator (cf. \cite[Lemma 3.2]{Bro-},
   \cite[Proposition 1.1, page 3]{Skr3} and \cite[Corollary~1.16]{Ch1}).

\begin{proposition}\label{prop:C.2}
Let $G$ be as above, $\Lambda$ a sequential compact topology space, and let
$f:G\times\R^N\times A\to\R$ satisfy the following conditions:
\begin{enumerate}
\item[\rm (a)] $f(x,\xi_1,\cdots, \xi_N;\lambda)$ is continuous in
$(\xi_1,\cdots, \xi_N; \lambda)$ for almost all $x\in G$;

\item[\rm (b)] $f(x,\xi_1,\cdots, \xi_N;\lambda)$ is measurable in $x$ for any fixed $(\xi_1,\cdots, \xi_N;\lambda)\in\R^N\times\Lambda$;

\item[\rm (c)] there exist positive numbers $C$, $1<p, p_1,\cdots, p_N<\infty$ and a function $g\in L^p(G)$ such that
\begin{equation}\label{e:C.1}
|f(x,\xi_1,\cdots,\xi_N;\lambda)|\le C\sum^N_{i=1}|\xi_i|^{\frac{p_i}{p}}+ g(x),\quad
\forall (x,\xi)\in\overline{\Omega}\times\R^N.
\end{equation}
\end{enumerate}
Then the  Nemytskii operators $F_\lambda: \prod^N_{i=1}L_{p_i}(G)\to L_p(G)$, $\lambda\in\Lambda$, defined by
$$
F_\lambda(u_1,\cdots,u_N)(x)=f(x,u_1(x),\cdots,u_N(x);\lambda)
$$
have an uniform bound on any bounded subset,  and uniform continuity at any point with respect to $\lambda\in\Lambda$, i.e.,  given a point $(u_1,\cdots,u_N)\in\prod^N_{i=1}L_{p_i}(G)$, $\forall\epsilon>0$, $\exists\;\delta>0$
such that
$$
\|F_\lambda(v_1,\cdots,v_N)-F_\lambda(u_1,\cdots,u_N)\|_p<\epsilon
$$
for all $\lambda\in\Lambda$ and $(v_1,\cdots,v_N)\in\prod^N_{i=1}L_{p_i}(G)$ with $\sum^N_{i=1}\|v_i-u_i\|_{p_i}<\delta$.
\end{proposition}
\begin{proof}
As in the case of $\Lambda=\{\hbox{a single point}\}$, the conditions
(a) and (b) imply that
$F_\lambda(u_1,\cdots,u_N)$ is measurable if all $u_1,\cdots,u_N$ are measurable.
Thus by (\ref{e:C.1}) all $F_\lambda$ map a bounded subset of $\prod^N_{i=1}L_{p_i}(G)$ into a common bounded subset of $L_p(G)$.

The proof for the second claim can be obtained by slightly changing the short
and elegant proof of \cite[Theorem~1.1.5]{Ch1}.
Indeed, by contradiction there exist sequences
$(\lambda_k)\subset\Lambda$, $(u_{i,k})\subset L^{p_i}(G)$, $i=1,\cdots,N$, a point $(u_1,\cdots,u_N)\in\prod^N_{i=1}L_{p_i}(G)$ and a real $\epsilon_0>0$ such that
\begin{eqnarray}\label{e:C.2}
\left.\begin{array}{ll}
&\|u_{i,k}-u_i\|_{p_i}\to 0,\;i=1,\cdots,N,\quad\hbox{and}\quad\\ %\nonumber\\
& \|F_{\lambda_k}(u_{1,k},\cdots,u_{N,k})-F_{\lambda_k}(u_1,\cdots,u_N)\|_{p}
\ge\epsilon_0\quad\forall k\in\N.\end{array}\right\}
\end{eqnarray}
As in the proof of \cite[Theorem~1.15]{Ch1}, passing to subsequences we may assume
$$
u_{i,k}(x)\to u_i(x)\;\hbox{a.e. in $G$,\quad and\quad} |u_{i,k}(x)|\le\Phi_i(x),
$$
where $\Phi_i\in L_{p_i}(\Omega)$, $i=1,\cdots,N$. Moreover, we can also assume $\lambda_k\to\lambda\in\Lambda$
since $\Lambda$ is sequential compact. Noting that $f:G\times\R^N\times A\to\R$ is continuous, we get
\begin{eqnarray*}
F_{\lambda_k}(u_{1,k},\cdots,u_{N,k})(x)&=&f(x,u_{1,k}(x),\cdots,u_{N,k}(x);\lambda_k)\\
&&\to F_{\lambda_0}(u_1,\cdots,u_N)(x)\;\hbox{a.e. in $G$},\\
F_{\lambda_k}(u_{1},\cdots,u_{N})(x)&=&f(x,u_{1}(x),\cdots,u_{N}(x);\lambda_k)\\
&&\to F_{\lambda_0}(u_1,\cdots,u_N)(x)\;\hbox{a.e. in $G$}.
\end{eqnarray*}
%\begin{eqnarray*}
%&&\hspace{-0.6cm}F_{\lambda_k}(u_{1,k},\cdots,u_{N,k})(x)=
%f(x,u_{1,k}(x),\cdots,u_{N,k}(x);\lambda_k)\to
%F_{\lambda_0}(u_1,\cdots,u_N)(x)\;\hbox{a.e. in $G$},\\
%&&F_{\lambda_k}(u_{1},\cdots,u_{N})(x)=f(x,u_{1}(x),\cdots,u_{N}(x);\lambda_k)\to F_{\lambda_0}(u_1,\cdots,u_N)(x)\;\hbox{a.e. in $G$}.
%\end{eqnarray*}
Moreover, (\ref{e:C.1}) implies that for some constant $C'=C(N, p_i,p,C)>0$
and  almost all $x\in G$,
\begin{eqnarray*}
|f(x,u_{1,k}(x),\cdots,u_{N,k}(x);\lambda_k)|^p&\le& (C\sum^N_{i=1}|u_{i,k}(x)|^{\frac{p_i}{p}}+ g(x))^p\\
&\le& C'(\sum^N_{i=1}\Phi_i(x)^{{p_i}}+ g(x)^p
\end{eqnarray*}
Hence the Lebesgue dominated convergence theorem leads to
\begin{eqnarray*}
&&\|F_{\lambda_k}(u_{1,k},\cdots,u_{N,k})-F_{\lambda_0}(u_1,\cdots,u_N)\|_{p}\to 0, \\
&&\|F_{\lambda_k}(u_{1},\cdots,u_{N})-F_{\lambda_0}(u_1,\cdots,u_N)\|_{p}\to 0
\end{eqnarray*}
and therefore $\|F_{\lambda_k}(u_{1,k},\cdots,u_{N,k})-F_{\lambda_k}(u_{1},\cdots,u_{N})\|_{p}\to 0$,
which is a contradiction to (\ref{e:C.2}).
\end{proof}

\noindent{\bf Acknowledgments}.
The author is deeply grateful to Professor Xiaochun Rong for his invitation. % and helps.
I also deeply thank the anonymous referees for useful remarks.

\renewcommand{\refname}{REFERENCES}

\medskip
% The data information below will be filled by AIMS editorial staff
%Received for publication April 2021; early access October 2021.
%\medskip

\begin{tabular}{l}
 School of Mathematical Sciences, Beijing Normal University\\
 Laboratory of Mathematics and Complex Systems, Ministry of Education\\
 Beijing 100875, The People's Republic of China\\
 E-mail address: gclu@bnu.edu.cn\\
\end{tabular}

%\medskip
%\quad% The data information below will be filled by AIMS editorial staff
%%Received August 2017; revised October 2017.
%\medskip
\end{document}